\documentclass[11pt,reqno]{amsart}
\usepackage{CJK,CJKpunct}
\usepackage{subeqnarray}
\usepackage{cases}
\usepackage{stmaryrd}
\usepackage{color}
\usepackage{amsfonts}
\usepackage{mathrsfs}
\usepackage{amssymb,amsmath,amsthm}
\usepackage{array}
\usepackage{booktabs}
\usepackage{epsfig}
\usepackage{graphicx}
\usepackage{appendix}
\usepackage{lineno}
\usepackage{enumerate}
\usepackage{comment}
\usepackage{caption}
\newcommand{\dps}{\displaystyle}
\newtheorem{theorem}{\indent Theorem}[section]
\newtheorem{lemma}{\indent Lemma}[section]

\newtheorem{remark}{\indent Remark}[section]

\newtheorem{example}{\indent Example}[section]
\newcommand{\ba}{\begin{array}}\newcommand{\ea}{\end{array}}
\newcommand{\be}{\begin{eqnarray}}\newcommand{\ee}{\end{eqnarray}}
\newcommand{\beq}{\begin{equation*}}\newcommand{\eeq}{\end{equation*}}
\newcommand{\bex}{\begin{eqnarray*}}
\newcommand{\eex}{\end{eqnarray*}}
\newcommand{\tabincell}[2]{\begin{tabular}{@{}#1@{}}#2\end{tabular}}
\newcommand{\PreserveBackslash}[1]{\let\temp=\\#1\let\\=\temp}
\newcolumntype{C}[1]{>{\PreserveBackslash\centering}p{#1}}
\newcolumntype{R}[1]{>{\PreserveBackslash\raggedleft}p{#1}}
\newcolumntype{L}[1]{>{\PreserveBackslash\raggedright}p{#1}}

\def\bq{\begin{equation}}
\def\eq{\end{equation}}
\def\beq{\begin{equation*}}
\def\eeq{\end{equation*}}
\def\br{\begin{eqnarray}}
\def\er{\end{eqnarray}}
\def\brr{\bq\begin{array}{r@{}l}}
\def\err{\end{array}\eq}
\def\bry{\beq\begin{array}{r@{}l}}
\def\ery{\end{array}\eeq}

\font\tenbi=cmmib10   at 11 pt
\font\sevenbi=cmmib10 at 9pt
\font\fivebi=cmmib7 at 6pt
\newfam\bifam
\textfont\bifam=\tenbi \scriptfont\bifam=\sevenbi  \scriptscriptfont\bifam=\fivebi

\def\bi{\fam\bifam\tenbi}
\font\sixtdb=msbm10 at 16 pt \font\tendb=msbm10 at 12 pt  \font\sevendb=msbm7
\newfam\dbfam
\textfont\dbfam=\sixtdb

\textfont\dbfam=\tendb \scriptfont\dbfam=\sevendb

\def\Dt {\triangle t}

\def\n{{\bi n}}
\def\x{{\bi x}}

\title[Unconditionally stable schemes for time fractional Allen-Cahn]
{Highly efficient schemes for time fractional Allen-Cahn equation using extended SAV approach$^*$}
\author[Dianming Hou, Hongyi Zhu, and Chuanju Xu]
{Dianming Hou$^{1}$
\quad
Hongyi Zhu$^{2}$
\quad
Chuanju Xu$^{2,3,4}$}
\thanks{\hskip -12pt
${}^*$This research is partially supported by NSF of China
(Grant numbers
11971408, 51661135011, 11421110001, and
91630204).
\\
$^{1}$School of Mathematics and Statistics, Jiangsu Normal
University, 221116 Xuzhou, China.\\
${}^{2}$School of Mathematical Sciences and
Fujian Provincial Key Laboratory of Mathematical Modeling and High Performance
Scientific Computing, Xiamen
University, 361005 Xiamen, China.\\
${}^{3}$Bordeaux INP, Laboratoire I2M UMR 5295, 33607 Pessac, France.\\
${}^{4}$Corresponding author. Email: cjxu@xmu.edu.cn (C. Xu)}

\begin{document}
\graphicspath{{figures/},}

\date {\today}
\maketitle

\begin{abstract}
In this paper, we propose and analyze high order efficient schemes
for the time fractional Allen-Cahn equation.
The proposed schemes are based on the L1 discretization for the time fractional derivative and
the extended scalar auxiliary variable (SAV)
approach developed very recently to deal with the nonlinear terms in the equation.
The main contributions of the paper consist in: 1)
constructing first and higher order unconditionally stable schemes for different mesh types, and
proving the unconditional stability of the constructed schemes for the uniform mesh; 2) carrying out numerical experiments
to verify the efficiency of the schemes and to investigate the coarsening dynamics governed by the time fractional Allen-Cahn equation. Particularly, the influence of the fractional order on the coarsening behavior is carefully examined.
Our numerical evidence shows that the proposed schemes are more robust than the existing methods,
and their efficiency is
less restricted to particular forms of the nonlinear potentials.
\end{abstract}

\section{Introduction}
\setcounter{equation}{0}

As a useful modelling tool, gradient flows has been used to model many physical problems,
particularly dissipative systems, which are systems driven by
dissipation of free energy.
This variational modelling framework generally leads to
partial differential equations having the mathematical structure of gradient flows.
In its abstract form, variational modelling of such systems consists of choosing a
state space, a driving functional, and a dissipation mechanism. Precisely,
models of gradient flows take the general form:
\be\label{prob0}
\dps\frac{\partial \phi}{\partial t}=-\mbox{grad}_{H} E(\phi),
\ee
where $\phi$ is the state function (also called phase function in many cases),
$E[\phi(\x,t)]$ is the free energy driving functional associated to the physical problem, and
$\mbox{grad}_{H} E(\cdot)$ is the functional derivative of $E$ in the Sobolev space $H$.
Multiplying both sides of \eqref{prob0} by $\delta E/\delta\phi$ and integrating the resulting equation
gives the energy dissipation law:
 \bq\label{EDlaw}
 \frac{d}{dt}E(\phi)
 =\Big(\mbox{grad}_{H} E(\phi), \frac{\partial \phi}{\partial t}\Big)
 =-\|\mbox{grad}_{H} E(\phi)\|_{0}^{2},
 \eq
 where $(\cdot,\cdot)$ is the $L^{2}-$inner product, and $\|\cdot\|_0$ is the associated norm.
 This means that
the state function $\phi$ evolves in such a way as to have the energy functional $E$ decrease in
time (least action principle from physics), i.e., in the opposite direction to the gradient of
$E$ at $\phi$.

The energy-based variational framework makes the equations a
thermodynamically-consistent and physically attractive approach to model multi-phase flows.
It has long been used in many fields of science and engineering,
particularly in materials science and fluid dynamics, see, e.g., \cite{Cahn59,All79,Les79,Doi88,Gur96,And97,Eld02,Yue04}
and the references therein.
Typical examples include the Cahn-Hilliard and Allen-Cahn equations for multi-phase flows, for which
the evolution PDE system is resulted from the energetic variation of the action functional of the
total free energy in different Sobolev spaces.

In this paper, we are interested in models deriving from gradient flows having modified dissipation mechanism.
That is, we consider the gradient flows of the form
\be\label{prob}
^{C}_{0}{}\!\!D^{\alpha}_{t}\phi=-\mbox{grad}_{H} E(\phi),
\ee
where $0<\alpha<1$, and $^{C}_{0}{}\!\!D^{\alpha}_{t}$ is the Caputo fractional derivative defined as follows:
\beq
^{C}_{0}{}\!\!D^{\alpha}_{t}\phi(t):=\frac{1}{\Gamma(1-\alpha)}\int_{0}^{t}(t-s)^{-\alpha}\phi'(s)ds.
\eeq
It is seen from the definition that the fractional derivative is some kind of weighted average of the traditional derivative
in the history. This means that the change rate, i.e., the derivative, at the current time is affected by the historical rates.
This property is often used to describe the memory effect which can be present in some
materials such as viscoelastic materials or polymers.
Intuitively, the gradient flows model (\ref{prob0}) is useful for describing the systems in which
dissipation of the associated free energy has memory effect in some circumstances.

Mathematically, such fractional-type gradient flows have been studied in the last few years and have quickly
attracted more and more attentions \cite{CZCWZ18,ZLWW14,DJLQ18,AM17,AIM17,ASS16,CZW18,LCWZ18,SXK16,LWY17,JLGZ19}.
Wang et al. \cite{CZW18,LCWZ18} numerically studied the time fractional phase field models, including the Cahn-Hilliard equations
with different variable mobilities and molecular beam epitaxy models.
Their numerical tests indicated that the effective free
energy and roughness of the time fractional phase field models obey a universal power law scaling dynamics during coarsening. Tang et al. \cite{TYZ18} proved
that the time-fractional phase field models indeed admit a modified energy dissipation law of an integral type.
Moreover, they applied the L1 approach to discretize the time fractional derivative on the uniform mesh, leading to a scheme of
first order accuracy.
Recently, Du et al. \cite{DYZ19} developed several time stepping schemes, i.e., convex
splitting scheme, weighted convex splitting scheme, and linear weighted stabilized scheme,
for the time fractional Allen-Cahn equation.
They proved that the convergence rates of the proposed time-stepping schemes are $O(\tau^{\alpha})$ for typical solutions of the
underlying equation. {\color{black}Very recently, Liao et al. \cite{LTZ19} developed a second-order nonuniform time-stepping scheme for the
time-fractional Allen-Cahn equation. They showed that the proposed scheme preserves the
discrete maximum principle and presented a sharp maximum-norm error estimate which reflects the temporal regularity.}

The goal of this paper is to construct and analyze some more efficient schemes
for the time fractional Allen-Cahn equation as follows:
\bq\label{prob}
^{C}_{0}{}\!\!D^{\alpha}_{t}\phi-\varepsilon^{2}\Delta\phi+F'(\phi)=0.
\eq
Precisely, our aim is to propose unconditionally stable schemes
which have provable higher order convergence rate than the existing schemes in the literature.
The main idea is to combine existing efficient discretizations for the time fractional derivative and
SAV approaches for the traditional gradient flows. In fact, there has been significant progress in the
numerical development for both the time fractional differential equation and gradient flows in recent years.
Among the existing approaches for the time fractional derivative,
the so-called L1 scheme is probably the best known; see, e.g., Sun and Wu
\cite{SW06} and Lin and Xu \cite{LX07}.
The L1 scheme makes use of a piecewise linear approximation and attains $(2-\alpha)$-order convergence rate for the $\alpha$-order fractional derivative \cite{LX07,GSZ12,LX15}.
Based on piecewise linear approximation at the closest interval to the time instant and piecewise quadratic interpolation at the previous time intervals,
Alikhanov \cite{Ali15} constructed a L2-$1_{\sigma}$ scheme and proved it is second order accurate when applied to the time fractional diffusion equation.
Gao et al. \cite{GSZ14} and Lv and Xu \cite{LX16} developed $(3-\alpha)$-order schemes
by using piecewise quadratic interpolation.
In order to reduce the high storage requirement,
several authors have proposed to approximate the weakly singular
kernel function by a sum-of-exponentials (SOE),
resulting in some new stepping schemes with reduced storage; see, e.g.,
\cite{BH17,BH17SIAM,JZ17,ZZJZ17,YSZ17,ZTB17}.
Concerning numerical approaches for gradient flows, we would like to mention the most recent results
in the past few years. There appeared two novel strategies: the invariant energy quadratization (IEQ)
method \cite{Yan16,Zha17} and the so-called scalar auxiliary variable (SAV) approach \cite{Shen17_1,Shen17_2}. Based on these two approaches, second-order
unconditionally stable schemes have been successfully constructed for a large class of gradient flow models.
Very recently, Hou et al. \cite{HAX19} proposed an extension of the SAV approach,
which leads to more robust unconditionally stable schemes and achieves more ability to guarantee the dissipation law of the total free energy.

The purpose of the current study is twofold. The first is to make use of the above mentioned results to construct efficient numerical
schemes for the time fractional Allen-Cahn equation, and analyze the stability properties of the proposed methods.
The convergence behavior will be investigated by means of numerical examples.
The second is to carry out numerical experiments by using the proposed schemes
to study coarsening dynamics governed by the fractional Allen-Cahn equation. In particular, we are interested to the effect of the fractional order on the coarsening behavior.
The paper is organized as follows: In the next section, we propose the extended SAV reformulation for the fractional
Allen-Cahn equation.
In Section 3, we describe the numerical schemes
for both uniform and nonuniform meshes, and
carry out the stability analysis for the proposed schemes on the uniform mesh.
In Section 4, we construct and analyze a higher order scheme.
The numerical experiments are carried out in Section 5, not only to validate the proposed methods, but also to
numerically investigate the coarsening dynamics.
Finally, the paper ends with some concluding remarks.

\section {Extended SAV reformulation}\label{sect2a}
\setcounter{equation}{0}

We first introduce some notations which will be used throughout the paper. Let $H^{m}(\Omega)$ and $\|\cdot\|_{m},
m=0,\pm1,\cdots$, denote the standard Sobolev spaces and their norms, respectively. In particular, the norm and inner product of $L^{2}(\Omega):=H^{0}(\Omega)$ are denoted by $\|\cdot\|_{0}$ and $(\cdot,\cdot)$ respectively.

For the sake of simplicity we consider the time fractional Allen-Cahn equation {\color{black}on a bounded domain $\Omega\in \mathbb{R}^{n}~(n=1,2,3)$}  subject to
the periodic boundary condition or Neumann boundary condition, i.e.,
\be\label{FAC}
\ba{r}
^{C}_{0}{}\!\!D^{\alpha}_{t}\phi-\varepsilon^{2}\Delta\phi+F'(\phi)=0, \ \mbox{in } (0,T]\times \Omega;\\[2mm]
\frac{\partial\phi}{\partial \n}\big|_{\partial\Omega}=0 \ \mbox{or periodic boundary condition};\\[2mm]
\phi(\x,0)=\phi_{0}(\x) , \ \mbox{in } \Omega.
\ea
\ee
It has been proved in \cite{TYZ18} that this problem admits the following dissipation law:
\bq\label{AClaw}
E(\phi(\x,t))-E(\phi(\x,0))=-\int_{\Omega}\mathcal{A}_{\alpha}(\phi_{t},\phi_{t})d\x\leq 0, \mbox{ for all }t\geq 0,
\eq
where
\beq
E(\phi):=\dps\int_{\Omega}\Big[\frac{\varepsilon^{2}}{2}|\nabla\phi|^{2}+F(\phi)\Big]d\x, \ \ \mathcal{A}_{\alpha}(g,h):=\frac{1}{\Gamma(1-\alpha)}\int_{0}^{t}\int_{0}^{s}(s-\sigma)^{-\alpha}g(\sigma)h(s)d\sigma ds.
\eeq

{\color{black}
We also define the partial free energy $E_{\theta}(\phi)$ by
\beq
E_{\theta}(\phi):=\dps\int_{\Omega}\Big[\frac{\theta}{2}|\nabla\phi|^{2}+F(\phi)\Big]d\x,\ \ 0\leq\theta\leq\varepsilon^{2}.
\eeq
Throughout this paper, we will assume that $E_{\theta}(\phi)$ is bounded from below, i.e., there always exists a positive constant $C_{0}$ such that $E_{\theta}(\phi)+C_{0}>0$.

We first derive a $H^{2}$-regularity result for the time fractional Allen-Cahn equation \eqref{FAC} with a generalized nonlinear free energy density $F(\phi)$. For nonlinear equations the $H^{2}$-regularity plays a crucial role in both theoretical and numerical analysis for the reason that the boundedness of the solution can be guaranteed by the $H^{2}$-regularity (for equations of dimension less than four).
We would like to point out that the $H^{2}$-regularity for the classical Allen-Cahn equation, i.e., $\alpha=1$, has recently been established
\cite{SX18}
by employing a type of Galerkin approach. We will make use of a similar approach to study the regularity of the fractional version.

\begin{theorem}\label{thm-regu}
Assume $\phi_{0}\in H^{2}(\Omega)$ and $F\in C^{3}(\mathbb{R})$ satisfying
\bq\label{asum1}
|F''(x)|<c(|x|^{p}+1),\quad p>0 \mbox{ if }n=1,2; \quad 0<p<4 \mbox{ if }n=3.
\eq
Then for any $T>0$, the problem \eqref{FAC} has a unique solution $\phi$ satisfying
\beq
\dps \phi\in C([0,T];H^{2}(\Omega))\cap L^{2}([0,T];H^{3}(\Omega)).
\eeq
\end{theorem}
\begin{proof}
Denote by $\{\omega_{j}\}$ the orthonormal basis in $L^{2}(\Omega)$ consisting of the eigenfunctions of $-\Delta,$ i.e.
\beq
-\Delta\omega_{j}=\lambda_{j}\omega_{j}.
\eeq
We first consider the approximate solution under the form
\beq
\phi_{m}(\cdot,t)=\dps\sum_{j=1}^{m}g_{jm}(t)\omega_{j}
\eeq
to the problem
\bq\label{req1}
(^{C}_{0}{}\!\!D^{\alpha}_{t}\phi_{m},\omega_{j})+\varepsilon^{2}(\nabla\phi_{m},\nabla\omega_{j})+(F'(\phi),\omega_{j})=0,
\eq
where $\phi_{m}(\x,0)$ is the $L^{2}$-projection of $\phi_{0}(\x)$ on the space spanned by $\{\omega_{j}\}$.
\\[1.5mm]
Multiplying both sides of \eqref{req1} by $g'_{jm}(t)$ then summing up for $j = 1,\dots, m$,
we get the following energy inequality
\beq
E(\phi_{m}(\x,t))-E(\phi_{m}(\x,0))=-\int_{\Omega}\mathcal{A}_{\alpha}(\partial_{t}\phi_{m},\partial_{t}\phi_{m})d\x\leq 0.
\eeq
Under the assumption
$\phi_{0}\in H^{2}(\Omega)\subseteq L^{\infty}(\Omega), E_{\theta}(\phi)+C_{0}>0, F\in C^{3}(\mathbb{R})$, and using the definition of $E(\phi)$, we deduce from the above energy inequality
the following $H^1$-bound
\beq
\|\phi_{m}\|_{H^{1}}\leq c E(\phi_{m}(0))+c\leq M_{0},
\eeq
where $M_{0}$ is positive constant depending on $\|\phi_{0}\|_{H^{2}}$.
\\[1.5mm]
To derive the $H^2$-bound, we
multiply \eqref{req1} by $\lambda^{2}_{j}g_{jm}(t)$ and sum up for $j=1,\cdots,m$. This gives
\bq\label{req2}
\dps\int_{\Omega} {}^{C}_{0}{}\!\!D^{\alpha}_{t}\Delta\phi_{m}\Delta\phi_{m}d\x+\|\nabla\Delta\phi_{m}\|^{2}_{0}=-(\nabla F'(\phi_{m}),\nabla\Delta\phi_{m})\leq\frac{1}{2}\|\nabla F'(\phi_{m})\|^{2}_{0}+\frac{1}{2}\|\nabla\Delta\phi_{m}\|^{2}_{0}.
\eq
Using Lemma 1 in \cite{Ali11} and Lemma 2.3 in \cite{SX18}, we get
\bq\label{req3}
\dps\int_{\Omega} {}^{C}_{0}{}\!\!D^{\alpha}_{t}\Delta\phi_{m}\Delta\phi_{m}d\x\geq \frac{1}{2}{}^{C}_{0}{}\!\!D^{\alpha}_{t}\|\Delta\phi_{m}\|^{2}_{0},\ \
\|\nabla F'(\phi_{m})\|^{2}_{0}\leq\frac{1}{2}\|\nabla\Delta\phi_{m}\|^{2}_{0}+c(M_{0}).
\eq
Combining \eqref{req2} and \eqref{req3} gives
\beq
{}^{C}_{0}{}\!\!D^{\alpha}_{t}\|\Delta\phi_{m}\|^{2}_{0}+\frac{1}{2}\|\nabla\Delta\phi_{m}\|^{2}_{0}\leq c(M_{0}).
\eeq
Then, applying the Riemann-Liouville fractional integral operator $_{0}I^{\alpha}_{t}$ to
the above inequality, we obtain
\beq
\|\Delta\phi_{m}\|^{2}_{0}-\|\Delta\phi_{m}(0)\|^{2}_{0}+\frac{1}{2}{}_{0}I^{\alpha}_{t}\|\nabla\Delta\phi_{m}\|^{2}_{0}\leq c(M_{0})T^{\alpha}.
\eeq
Furthermore, we have
\beq
\|\Delta\phi_{m}\|^{2}_{0}\leq c(M_{0})T^{\alpha},\ \ \int_{0}^{t}\|\nabla\Delta\phi_{m}\|^{2}_{0}dt\leq c(M_{0})(T+T^{1-\alpha}).
\eeq
This means $\phi_{m}\in L^{\infty}([0,T];H^{2}(\Omega))\cap L^{2}([0,T];H^{3}(\Omega))$.
\\[1.5mm]
Following a similar discussion as in Theorem 2.6 in \cite{SX18},
we can prove $\phi_{m}\rightarrow \phi$ strongly in $L^{2}([0,T];H^{1}(\Omega))$ as $m\rightarrow\infty$, and
$\phi\in C([0,T];H^{2}(\Omega))\cap L^{2}([0,T];H^{3}(\Omega))$.
\\[1.5mm]
The uniqueness of the solution can also be established by following the same lines as in \cite{SX18},
Lemma 1 in \cite{Ali11}, and Gronwall lemma.
This completes the proof.
\end{proof}
}

The key to construct our schemes for \eqref{FAC} is to rewrite the original equation
into the following equivalent form,
inspired by the idea of the extended SAV approach introduced in \cite{HAX19}:
 \be\label{re_prob}
 \dps ^{C}_{0}{}\!\!D^{\alpha}_{t}\phi-\varepsilon^{2}\Delta\phi+\Big(1-\frac{R(t)}{\sqrt{E_{\theta}(\phi)+C_{0}}}\Big)\theta\Delta\phi+\frac{R(t)}{\sqrt{E_{\theta}(\phi)+C_{0}}}F'(\phi)=0,
\ee
where
\beq
R(t)=\sqrt{E_{\theta}(\phi)+C_{0}}.
\eeq

At the discrete level, we need to find a suitable way to evaluate the auxiliary variable $R(t)$. To this end, we
formally take the derivative of $R(t)$ with respect to $t$ to obtain the following auxiliary equation:
\be\label{re_prob2}
 \dps\frac{d R}{d t}=\frac{1}{2\sqrt{E_{\theta}(\phi)+C_{0}}}\int_{\Omega}(-\theta\Delta\phi+F'(\phi))
\frac{\partial \phi}{\partial t}d\x.
\ee
Then we have two coupled equations \eqref{re_prob} and \eqref{re_prob2} governing two unknown functions, i.e., the original
two variable function $\phi$ and the auxiliary one variable function $R(t)$. These two equations will be used to construct our
efficient time-stepping schemes to numerically compute the approximate solutions.
Before that, we first realize, by
taking $L^2-$inner products and integrating from 0 to $t$ of the two equations \eqref{re_prob} and \eqref{re_prob2} with $\frac{\partial \phi}{\partial t}$
 and $2R(t)$ respectively, that
 \be\label{law2}
\Big[R^{2}(t)+\frac{\varepsilon^{2}-\theta}{2}\|\nabla\phi(\x,t)\|^{2}_{0}\Big]-\Big[R^{2}(0)+\frac{\varepsilon^{2}-\theta}{2}\|\nabla\phi(\x,0)\|^{2}_{0}\Big]=-\int_{\Omega}\mathcal{A}_{\alpha}(\phi_{t},\phi_{t})d\x\leq 0.
 \ee
 Noticing that
 $$\dps R^{2}+\frac{\varepsilon^{2}-\theta}{2}\|\nabla\phi\|^{2}_{0}=E(\phi)+C_{0},
 $$
 we obtain exactly the same dissipation law as \eqref{AClaw}. This is a natural result since the two coupled equations
 \eqref{re_prob} and \eqref{re_prob2}, together with corresponding boundary and initial conditions,
 are strictly equivalent to the original problem \eqref{FAC}.
 However, as we are going to see, at the discrete level the energy dissipation law
 will take the form \eqref{law2} more than \eqref{AClaw}.

\section {A first order scheme}\label{sect2}
\setcounter{equation}{0}

 Starting with the equivalent equations \eqref{re_prob} and \eqref{re_prob2}, we are able to construct various time stepping
 schemes to calculate the solution $\phi$.  Let $n=0,1,\dots, M$ be the time step, $\Dt_{n}:=t_{n}-t_{n-1}$ be the time step size,
 and $\tau=\max\{\Dt_{n},n=1,2,\cdots\}$ be the maximum time step size of the temporal mesh.

 \subsection{First order scheme}
 \label{s2.1}
 We first recall the very popular L1 approximation \cite{LX07} to discretize the Caputo fractional derivative $^{C}_{0}{}\!\!D^{\alpha}_{t}\phi$:
  \bry
  &\dps ^{C}_{0}{}\!\!D^{\alpha}_{t_{n+1}}\phi=\frac{1}{\Gamma(1-\alpha)}\int_{0}^{t_{n+1}}(t_{n+1}-s)^{-\alpha}\phi_{s}(\x,s)ds\\[9pt]
  =&\dps\sum_{k=0}^{n}\frac{\phi(\x,t_{k+1})-\phi(\x,t_{k})}{\Dt_{k+1}}\frac{1}{\Gamma(1-\alpha)}\int_{t_{k}}^{t_{k+1}}(t_{n+1}-s)^{-\alpha}ds+e^{n+1}_{\tau}\\[12pt]
  :=&L^{\alpha}_{n+1}\phi+e^{n+1}_{\tau},
  \ery
 where
\be\label{L1}
 L^{\alpha}_{n+1}\phi
 =\dps\sum_{k=0}^{n}b_{n-k}\frac{\phi(\x,t_{k+1})-\phi(\x,t_{k})}{\Dt_{k+1}}, \ \ b_{n-k}=\frac{1}{\Gamma(1-\alpha)}\int_{t_{k}}^{t_{k+1}}(t_{n+1}-s)^{-\alpha}ds,
\ee
and the truncation error $e^{n+1}_{\tau}$ is defined by
\bex
e^{n+1}_{\tau}
= \frac{1}{\Gamma(1-\alpha)}\Big[\int_{0}^{t_{n+1}}(t_{n+1}-s)^{-\alpha}\phi_{s}(\x,s)ds
- \dps\sum_{k=0}^{n}\frac{\phi(\x,t_{k+1})-\phi(\x,t_{k})}{\Dt_{k+1}}\int_{t_{k}}^{t_{k+1}}(t_{n+1}-s)^{-\alpha}ds\Big].
\eex

- In the uniform mesh case, i.e., $t_{n}={n\over M}T, n=0,1,\dots,M$, we have
\bq\label{uniform}
b_{j}=\frac{\tau^{1-\alpha}}{\Gamma(2-\alpha)}[(j+1)^{1-\alpha}-j^{1-\alpha}], ~~j=0,1,\cdots,n.
\eq
In this case the truncation error $e^{n+1}_{\tau}$ can be bounded by $c_{\phi}\tau^{2-\alpha}$ if the exact solution is sufficiently smooth \cite{LX07}.
The following property has also been known; see, e.g., \cite{TYZ18}:
for any $(u_{0},u_{1},\cdots,u_{n})^T\in \mathbb{R}^{n+1},$ it holds
\bq\label{posit}
\dps\sum_{k=0}^{n}\sum_{j=0}^{k}b_{k-j}u_{j}u_{k}>0.
\eq

- In the graded mesh case, i.e., $t_{n}=\big({n\over M}\big)^rT, r\geq1, n=0,1,\cdots,M$,
we have
\beq
b_j=\frac{T^{1-\alpha}}{\Gamma(2-\alpha)M^{(1-\alpha)r}}\Big[\big[(n+1)^{r}-(n-j)^{r}\big]^{1-\alpha}-\big[(n+1)^{r}-(n-j+1)^{r}\big]^{1-\alpha}\Big].
\eeq

Applying the above L1 approximation and some other first order finite difference approximations to \eqref{re_prob} and \eqref{re_prob2}, we propose the first scheme as follows:
 \begin{subequations}\label{eq1}
 \begin{align}
 &\dps L^{\alpha}_{n+1}\phi-\varepsilon^{2}\Delta\phi^{n+1}+\Big(1-\frac{R^{n+1}}{\sqrt{E_{\theta}^{n}+C_{0}}}\Big)\theta\Delta\phi^{n}+\frac{R^{n+1}}{\sqrt{E_{\theta}^{n}+C_{0}}}F'(\phi^{n})=0,\label{eq1_2}\\
 &\dps\frac{R^{n+1}-R^{n}}{\Delta t_{n+1}}=\frac{1}{2\sqrt{E_{\theta}^{n}+C_{0}}}\int_{\Omega}(-\theta\Delta\phi^{n}+F'(\phi^{n}))
\frac{\phi^{n+1}-\phi^{n}}{\Delta t_{n+1}}d\x,\label{eq1_3}
 \end{align}
 \end{subequations}
where $\phi^{n}$ is an approximation to $\phi(t^n)$ and $E^{n}_{\theta}:=E_{\theta}(\phi^{n})$.
Intuitively, this scheme is of first-order convergence with respect to the time step size.
Although a rigorous proof remains an open question, this intuitive convergence rate will be investigated
through the numerical experiments to be presented later on.

\subsection{Stability}
 \label{s2.2}
 Next we will analyze the stability property of the scheme in the following theorem.
\begin{theorem}\label{th1}
If the mesh is uniform, then the scheme \eqref{eq1} is unconditionally stable in the sense that the following discrete energy law holds
\bq\label{eq2}
\dps\frac{E^{n+1}_{\varepsilon,\theta}-E^{0}_{\varepsilon,\theta}}{\tau}
\leq-\sum_{k=0}^{n}(L^{\alpha}_{k+1}\phi,\frac{\phi^{k+1}-\phi^{k}}{\tau})<0,
\eq
where
\beq
\dps E^{n}_{\varepsilon,\theta}:=\frac{1}{2}(\varepsilon^{2}-\theta)\|\nabla\phi^{n}\|^{2}_{0}+|R^{n}|^{2}.
 \eeq
\end{theorem}
\begin{proof}
For the uniform mesh, we deduce from taking the $L^2-$inner products of \eqref{eq1_2} and \eqref{eq1_3} with
$\frac{\phi^{n+1}-\phi^{n}}{\tau}$ and $2R^{n+1}$ respectively:
\bry
 &\dps\Big(L^{\alpha}_{n+1}\phi,\frac{\phi^{n+1}-\phi^{n}}{\tau}\Big)
 +\dps\frac{\varepsilon^{2}-\theta}{\tau}\big(\nabla\phi^{n+1},\nabla(\phi^{n+1}-\phi^{n})\big)+\frac{\theta}{\tau}\|\nabla(\phi^{n+1}-\phi^{n})\|^{2}_{0}\\[9pt]
 &\hspace{5.3cm}\dps+\frac{R^{n+1}}{\sqrt{E_{\theta}^{n}+C_{0}}}\Big(-\theta\Delta\phi^{n}+F'(\phi^{n}),\frac{\phi^{n+1}-\phi^{n}}{\tau}\Big)=0, \\[9pt]
& \dps\frac{2R^{n+1}(R^{n+1}-R^{n})}{\tau}
=\dps\frac{R^{n+1}}{\sqrt{E_{\theta}^{n}+C_{0}}}\Big(-\theta\Delta\phi^{n}+F'(\phi^{n}),\frac{\phi^{n+1}-\phi^{n}}{\tau}\Big).
\ery
Then using the identity
\bq\label{q1}
2a^{n+1}(a^{n+1}-a^n)=|a^{n+1}|^{2}-|a^{n}|^{2}+|a^{n+1}-a^{n}|^{2}.
\eq
and dropping some non-essential positive terms, we obtain
\beq
\dps\frac{E^{n+1}_{\varepsilon,\theta}-E^{n}_{\varepsilon,\theta}}{\tau}
\leq-\Big(L^{\alpha}_{n+1}\phi,\frac{\phi^{n+1}-\phi^{n}}{\tau}\Big),~~n=0,1,2,\cdots,
\eeq
Summing up the above equality from 0 to $n$ and using \eqref{L1} and \eqref{posit}, we obtain \eqref{eq2}.
\end{proof}
{\color{black}
\begin{remark}\label{rek1}
Theorem \ref{th1} indicates that the scheme \eqref{eq1} is unconditionally stable in the sense that the
modified energy $E^{n}_{\varepsilon,\theta}$ is a decreasing function of the time step. Unfortunately,
as all other SAV-based
schemes, there is no proof for the decay of the original energy $E(\phi^n)$, although many (but not all) numerical
investigation have shown that the discrete original energy is also decreasing. Our numerical experiments will also
confirm this point.
It is notable that $E^{n}_{\varepsilon,\theta}$ is a first order approximation to $E(\phi^n)$. Furthermore,
the energy stability \eqref{eq2} implies the $H^{1}$-boundedness of $\phi^{n}$, i.e., there
 exists a positive constant $M_{0}$, which may depend on $\Omega$ and $\phi_{0}$, such that
 \bq\label{req4}
 \|\phi^{n}\|_{H^{1}}+|R^{n}|\leq M_{0}, \mbox{ for all }n=1,2,\cdots, \mbox{ if } \phi_{0}\in H^{1}(\Omega).
 \eq
\end{remark}

The next theorem gives the $H^2$ bound, consequently $L^\infty$ bound, of the numerical solutions.

\begin{theorem}\label{the1}
Assume $\phi_{0}\in H^{3}(\Omega)$ and $F\in C^{3}(\mathbb{R})$ satisfying \eqref{asum1}.
Let $M_{0}$ be the bound given in \eqref{req4}.
Then for the uniform mesh and any $(n+1)<T/\tau$, the scheme \eqref{eq1} satisfies
\beq
\dps\sum_{k=0}^{n}b_{n-k}\|\Delta\phi^{k+1}\|^{2}_{0}
+\tau\varepsilon^{2}\|\nabla\Delta\phi^{n+1}\|^{2}_{0}
\leq c(M_{0},\varepsilon)T+\frac{T^{1-\alpha}}{\Gamma(2-\alpha)}\|\Delta\phi_{0}\|^{2}_{0}+\tau\varepsilon^{2}\|\nabla\Delta\phi_{0}\|^{2}_{0}.
\eeq
\end{theorem}
\begin{proof}
We deduce from taking the $L^2-$inner products of \eqref{eq1_2} with $2\tau\Delta^{2}\phi^{n+1}$,
\bry
 &\dps2\tau\Big(L^{\alpha}_{n+1}\Delta\phi,\Delta\phi^{n+1}\Big)
 +2\tau\varepsilon^{2}\|\nabla\Delta\phi^{n+1}\|^{2}_{0}=2\tau\Big(1-\frac{R^{n+1}}{\sqrt{E_{\theta}^{n}+C_{0}}}\Big)\theta(\nabla\Delta\phi^{n},\nabla\Delta\phi^{n+1})\\[9pt]
 &\hspace{7.3cm}\dps+\frac{2\tau R^{n+1}}{\sqrt{E_{\theta}^{n}+C_{0}}}(\nabla F'(\phi^{n}),\nabla\Delta\phi^{n+1}).
\ery
For the terms in the right hand side,  with a similar discussion as in \cite{SX18} Lemma 2.4,
we can obtain that
\bq\label{re_equ1}
 \dps2\tau\Big(L^{\alpha}_{n+1}\Delta\phi,\Delta\phi^{n+1}\Big)
 +2\tau\varepsilon^{2}\|\nabla\Delta\phi^{n+1}\|^{2}_{0}\leq\tau c(M_{0},\varepsilon)+\tau\varepsilon^{2}\|\nabla\Delta\phi^{n}\|^{2}_{0}+\tau\varepsilon^{2}\|\nabla\Delta\phi^{n+1}\|^{2}_{0}.
\eq
Moreover, using Lemma 1 in \cite{Ali15} gives
\bq\label{re_equ2}
 \Big(L^{\alpha}_{n+1}\Delta\phi,\Delta\phi^{n+1}\Big)\geq \frac{1}{2}L^{\alpha}_{n+1}\|\Delta\phi\|^{2}_{0}.
\eq
Combining \eqref{re_equ1} and \eqref{re_equ2}, we have
\beq
 \dps\tau L^{\alpha}_{n+1}\|\Delta\phi\|^{2}_{0}
 +\tau\varepsilon^{2}\|\nabla\Delta\phi^{n+1}\|^{2}_{0}\leq\tau c(M_{0},\varepsilon)+\tau\varepsilon^{2}\|\nabla\Delta\phi^{n}\|^{2}_{0}.
\eeq
Taking the sum of the above inequality from 0 to n and noticing that
\beq
\dps\sum_{k=0}^{n}\tau L^{\alpha}_{k+1}\|\Delta\phi\|^{2}_{0}
=\sum_{k=0}^{n}b_{n-k}\|\Delta\phi^{k+1}\|^{2}_{0}-\sum_{k=0}^{n}b_{k}\|\Delta\phi_{0}\|^{2}_{0}=\sum_{k=0}^{n}b_{n-k}\|\Delta\phi^{k+1}\|^{2}_{0}-\frac{[(n+1)\tau]^{1-\alpha}}{\Gamma(2-\alpha)}\|\Delta\phi_{0}\|^{2}_{0},
\eeq
we get
\bry
 \dps\sum_{k=0}^{n}b_{n-k}\|\Delta\phi^{k+1}\|^{2}_{0}
 +\tau\varepsilon^{2}\|\nabla\Delta\phi^{n+1}\|^{2}_{0}&\dps\leq(n+1)\tau c(M_{0},\varepsilon)+\frac{[(n+1)\tau]^{1-\alpha}}{\Gamma(2-\alpha)}\|\Delta\phi_{0}\|^{2}_{0}+\tau\varepsilon^{2}\|\nabla\Delta\phi_{0}\|^{2}_{0}\\[9pt]
 &\dps\leq c(M_{0},\varepsilon)T+\frac{T^{1-\alpha}}{\Gamma(2-\alpha)}\|\Delta\phi_{0}\|^{2}_{0}+\tau\varepsilon^{2}\|\nabla\Delta\phi_{0}\|^{2}_{0}.
\ery
The proof is completed.
\end{proof}
}
\subsection{Implementation}
\label{2.3}
Beside of its unconditional stability,
another notable property of the new scheme is that it can be efficiently implemented.
To see that, let's consider the case of uniform mesh.
We first eliminate $R^{n+1}$ from \eqref{eq1_2} by using \eqref{eq1_3} to obtain
\bq\label{te1}
\dps\sum_{k=0}^{n}b_{n-k}\frac{\phi^{k+1}-\phi^{k}}{\tau}-\varepsilon^{2}\Delta\phi^{n+1}+\theta \Delta\phi^{n}
+\Big[\frac{R^{n}}{\sqrt{E_{\theta}^{n}+C_{0}}}+\frac{1}{2(E_{\theta}^{n}+C_{0})}(\gamma^{n},\phi^{n+1}-\phi^{n})\Big]\gamma^{n}=0,
\eq
where
\beq
\gamma^{n}:=-\theta\Delta\phi^{n}+F'(\phi^{n}).
\eeq
Then reformulating \eqref{te1} gives
\be\label{eq3}
&&\big(\frac{b_{0}}{\tau}-\varepsilon^{2}\Delta\big)\phi^{n+1}+(\gamma^{n},\phi^{n+1})\frac{\gamma^{n}}{2(E_{\theta}^{n}+C_{0})}
\\ \nonumber
&=&-\frac{1}{\tau}\Big[\sum_{k=0}^{n-1}(b_{n-k}-b_{n-k-1})\phi^{k+1}+b_{n}\phi^{0}\Big]-\theta\Delta\phi^{n}-\Big[\frac{R^{n}}{\sqrt{E_{\theta}^{n}+C_{0}}}-\frac{1}{2(E_{\theta}^{n}+C_{0})}(\gamma^{n},\phi^{n})\Big]\gamma^{n}.
\ee
Denoting the right hand side in \eqref{eq3} by $g(\phi^{n})$,
we see that problem \eqref{eq3} can be solved in two steps
as follows:
\begin{subequations}\label{semi_D}
 \begin{align}
 &\begin{cases}
   \begin{array}{r@{}l}
    &\dps\big(\frac{b_{0}}{\tau}-\varepsilon^{2}\Delta\big)\phi^{n+1}_{1}=-\frac{\gamma^{n}}{2(E_{\theta}^{n}+C_{0})},\\[9pt]
    &\dps\mbox{Neumann boundary condition or periodic boundary condition on $\phi^{n+1}_1$;}\\
    \end{array}
 \end{cases}\label{semi_D_1}\\
 &\begin{cases}
    \begin{array}{r@{}l}
      &\dps\big(\frac{b_{0}}{\tau}-\varepsilon^{2}\Delta\big)\phi^{n+1}_{2}=g(\phi^{n}),\\[9pt]
      &\dps\mbox{Neumann boundary condition or periodic boundary condition on $\phi^{n+1}_2$;}\\
    \end{array}
  \end{cases}\label{semi_D_2}\\
 &\dps \phi^{n+1}=(\gamma^{n},\phi^{n+1})\phi^{n+1}_{1}+\phi^{n+1}_{2}.\label{semi_D_3}
 \end{align}
 \end{subequations}
In a first look it seems that \eqref{semi_D_3} satisfied by $\phi^{n+1}$ is an implicit equation.
However a careful examination shows that
$(\gamma^{n},\phi^{n+1})$ in \eqref{semi_D_3} can be determined explicitly.
In fact, taking the inner product of \eqref{semi_D_3} with $\gamma^{n}$ yields
\bq\label{eq4}
(\gamma^{n},\phi^{n+1})+ \sigma^{n}(\gamma^{n},\phi^{n+1})=(\gamma^{n},\phi^{n+1}_{2}),
\eq
where
\bq\label{posi}
\sigma^{n}=-(\gamma^{n},\phi^{n+1}_{1})
=\big(\gamma^{n},A^{-1}\frac{\gamma^{n}}{2(E_{\theta}^{n}+C_{0})}\big)\ \ \mbox{with } A=\frac{b_{0}}{\tau}-\varepsilon^{2}\Delta.
\eq
Note that
$A^{-1}$ is a positive definite operator and $E_{\theta}^{n}+C_{0}>0$. Thus $\sigma^{n}\ge 0$.
Then it follows from \eqref{eq4} that
\bq\label{eq5}
(\gamma^{n},\phi^{n+1})=\frac{(\gamma^{n},\phi^{n+1}_{2})}{1+\sigma^{n}}.
\eq
Using this expression, $\phi^{n+1}$ can be explicitly computed from \eqref{semi_D_3}.

In detail, the scheme \eqref{eq1} results in the following algorithm at each time step:

$(i)$ Calculation of $\phi^{n+1}_{1}$ and $\phi^{n+1}_{2}$: solving the elliptic  problems
\eqref{semi_D_1} and \eqref{semi_D_2} respectively, which can be realized in parallel.

$(ii)$ Evaluation of $(\gamma^{n},\phi^{n+1})$ using \eqref{eq5}, then $\phi^{n+1}$ using \eqref{semi_D_3}.

Thus the overall computational complexity at each time step is essentially
solving two second-order elliptic problems with constant coefficients, for which there exist different fast solvers depending on
the spatial discretization method.

\section {A higher order scheme}\label{sect3}
\setcounter{equation}{0}

\subsection {$2-\alpha$ order scheme}

Here we propose a semi-implicit scheme, called L1-CN scheme hereafter,
by applying the L1 discretization to the fractional derivative and
the Crank-Nicolson discretization to the remaining terms
for the system \eqref{re_prob} and \eqref{re_prob2}:
 \begin{subequations}\label{L1_CN}
 \begin{align}
 &\dps L^{\alpha}_{n+{1\over 2}}\phi-\varepsilon^{2}\frac{\Delta(\phi^{n+1}+\phi^{n})}{2}+\Big(1-\frac{R^{n+1}+R^{n}}{2\sqrt{E_{\theta}^{n+\frac{1}{2}}+C_{0}}}\Big)\theta\Delta\phi^{n+\frac{1}{2}}+\frac{R^{n+1}+R^{n}}{2\sqrt{E_{\theta}^{n+\frac{1}{2}}+C_{0}}}F'(\phi^{n+\frac{1}{2}}),\label{L1_CN_1}\\
 &\dps\frac{R^{n+1}-R^{n}}{\Delta t_{n+1}}=\frac{1}{2\sqrt{E_{\theta}^{n+\frac{1}{2}}+C_{0}}}\int_{\Omega}(-\theta\Delta\phi^{n+\frac{1}{2}}+F'(\phi^{n+\frac{1}{2}}))
\frac{\phi^{n+1}-\phi^{n}}{\Delta t_{n+1}}d\x,\label{L1_CN_2}
 \end{align}
 \end{subequations}
 where $\phi^{n+\frac{1}{2}}:=\phi^{n}+\frac{\Delta t_{n+1}}{2\Delta t_{n}}[\phi^{n}-\phi^{n-1}]$ is an explicit approximation to $\phi(t^{n+\frac{1}{2}})$, $E_{\theta}^{n+\frac{1}{2}}:=E_{\theta}(\phi^{n+\frac{1}{2}})$, for $n=0,1,\cdots$ with $\phi^{-1}:=\phi^{0}$, and
 \be\label{LCN}
  L^{\alpha}_{n+{1\over 2}}\phi &\ :=\dps\sum_{k=0}^{n-1}\frac{\phi^{k+1}-\phi^{k}}{\Dt_{k+1}}\frac{1}{\Gamma(1-\alpha)}\int_{t_{k}}^{t_{k+1}}(t_{n+\frac{1}{2}}-s)^{-\alpha}ds
  \nonumber \\[3pt]
  &\dps\quad+\frac{\phi^{n+1}-\phi^{n}}{\Dt_{n+1}}\frac{1}{\Gamma(1-\alpha)}\int_{t_{n}}^{t_{n+\frac{1}{2}}}(t_{n+\frac{1}{2}}-s)^{-\alpha}ds \nonumber \\[3pt]
   &\ :=\dps\sum_{k=0}^{n}\widetilde{b}_{n-k}\frac{\phi^{k+1}-\phi^{k}}{\Dt_{k+1}}, \mbox{ with $t_{n+\frac{1}{2}}=(t_{n}+t_{n+1})/2.$}
 \ee

- In the uniform mesh case, the coefficient $\widetilde{b}_{k}$ can be expressed as follows:
\bq\label{uniform_CN}
\widetilde{b}_{0}=\frac{\tau^{1-\alpha}}{\Gamma(2-\alpha)2^{1-\alpha}},~~
\widetilde{b}_{k}=\frac{\tau^{1-\alpha}}{\Gamma(2-\alpha)}\Big[\big(k+\frac{1}{2}\big)^{1-\alpha}-\big(k-\frac{1}{2}\big)^{1-\alpha}\Big], ~~k=1,2,\cdots,n.
\eq

- In the graded mesh case,
a direct calculation using $t_{n}=\big({n\over M}\big)^rT$ gives $\tau=\Delta t_{M}$
\bex
\widetilde{b}_{0}=\frac{T^{1-\alpha}}{\Gamma(2-\alpha)(2M^{r})^{1-\alpha}[(n+1)^{r}-n^{r}]^{\alpha}},
\eex
\bex
\widetilde{b}_{k}=\frac{T^{1-\alpha}}{\Gamma(2-\alpha)(2M^{r})^{1-\alpha}}\frac{\big[(n+1)^{r}+n^{r}-2(n-k)^{r}\big]^{1-\alpha}-\big[(n+1)^{r}+n^{r}-2(n-k+1)^{r}\big]^{1-\alpha}}{(n-k+1)^{r}-(n-k)^{r}},\\
k=1,2,\cdots,n.
\eex
The scheme \eqref{L1_CN} is expected to have $2-\alpha$ order
convergence, since intuitively the approximation $L^{\alpha}_{n+{1\over 2}}\phi$ to $^{C}_{0}{}\!\!D^{\alpha}_{t_{n+{1\over 2}}}\phi$ is of $2-\alpha$ order,
and the remaining approximation is of second order.
Indeed the approximation order of the operator $L^{\alpha}_{n+{1\over 2}}$ can be derived
rigorously, as shown in the following lemma.

\begin{lemma}
For the graded mesh, $\alpha\in(0,1)$ and $\phi\in C^2[0,t_{n+1}]$, we have
\beq
\Big|{}^{C}_{0}{}\!\!D^{\alpha}_{t_{n+{1\over 2}}}\phi-L^{\alpha}_{n+{1\over 2}}\phi\Big|=\mathcal{O}(\tau^{2-\alpha}).
\eeq
\end{lemma}
\begin{proof}
On each small interval $[t_{k},t_{k+1}] (0\leq k\leq n),$ denoting the linear interpolation function of $\phi(t)$ as $\Pi_{1,k}\phi$:
\beq
\Pi_{1,k}\phi=\frac{t_{k+1}-t}{\Delta t_{k}}\phi(\x,t_{k})+\frac{t-t_{k}}{\Delta t_{k}}\phi(\x,t_{k+1}),
\eeq
it follows from the linear interpolation theorem that
\beq
\phi-\Pi_{1,k}\phi=\frac{\partial^2_{t}\phi(\x,\xi_{k})}{2}(t-t_{k})(t-t_{k+1}),\quad t\in[t_{k},t_{k+1}],\xi_{k}\in(t_{k},t_{k+1}), 0\leq k\leq n.
\eeq
Let ${}^{C}_{0}{}\!\!D^{\alpha}_{t_{n+{1\over 2}}}\phi-L^{\alpha}_{n+{1\over 2}}\phi=\widehat{R}^{n}+\widehat{R}^{n+{1\over 2}},$ where
\bry
\widehat{R}^{n}&\dps=\frac{1}{\Gamma(1-\alpha)}\sum_{k=0}^{n-1}\int_{t_{k}}^{t_{k+1}}(t_{n+{1\over 2}}-s)^{-\alpha}\phi_{s}(\x,s)ds-\frac{1}{\Gamma(1-\alpha)}\sum_{k=0}^{n-1}\int_{t_{k}}^{t_{k+1}}(t_{n+{1\over 2}}-s)^{-\alpha}\partial_{s}(\Pi_{1,k}\phi) ds\\[8pt]
&\dps=\frac{1}{\Gamma(1-\alpha)}\sum_{k=0}^{n-1}\int_{t_{k}}^{t_{k+1}}(t_{n+{1\over 2}}-s)^{-\alpha}\partial_{s}(\phi(\x,s)-\Pi_{1,k}\phi) ds\\[8pt]
&\dps=\frac{-\alpha}{\Gamma(1-\alpha)}\sum_{k=0}^{n-1}\int_{t_{k}}^{t_{k+1}}(t_{n+{1\over 2}}-s)^{-\alpha-1}(\phi(\x,s)-\Pi_{1,k}\phi) ds\\[8pt]
&\dps=\frac{-\alpha}{2\Gamma(1-\alpha)}\sum_{k=0}^{n-1}\int_{t_{k}}^{t_{k+1}}\partial_{t}^{2}\phi(\x,\xi_{k})(s-t_{k})(s-t_{k+1})(t_{n+{1\over 2}}-s)^{-\alpha-1}ds,\\[8pt]
\widehat{R}^{n+{1\over 2}}&\dps=\frac{1}{\Gamma(1-\alpha)}\int_{t_{n}}^{t_{n+{1\over 2}}}(t_{n+{1\over 2}}-s)^{-\alpha}\phi_{s}(\x,s)ds\\[8pt]
&\quad\quad\quad \dps-\frac{1}{\Gamma(1-\alpha)}\int_{t_{n}}^{t_{n+{1\over 2}}}(t_{n+{1\over 2}}-s)^{-\alpha}\frac{\phi(\x,t_{n+1})-\phi(\x,t_{n})}{\Delta t_{n}} ds\\[8pt]
&\dps=\frac{1}{\Gamma(1-\alpha)}\int_{t_{n}}^{t_{n+{1\over 2}}}(t_{n+{1\over 2}}-s)^{-\alpha}(\phi_{t}(\x,t_{n+\frac{1}{2}})-\frac{u^{n+1}-u^{n}}{\Delta t_{n}})ds\\[8pt]
&\quad\quad\quad \dps+\frac{1}{\Gamma(1-\alpha)}\int_{t_{n}}^{t_{n+{1\over 2}}}(t_{n+{1\over 2}}-s)^{-\alpha}\partial_{t}^{2}\phi(\x,\theta_{s})(s-t_{n+\frac{1}{2}})ds,\quad \theta_{s}\in(s,t_{n+\frac{1}{2}}).\\[8pt]
\ery
Noting $\Delta t_{k-1}\leq \Delta t_{k}$ for graded mesh, we have following estimates
\bry
|\widehat{R}^{n}|&\dps\leq\frac{\alpha\dps\max_{0<t<t_{n+1}}|\partial_{t}^{2}\phi(\x,t)|}{2\Gamma(1-\alpha)}\sum_{k=0}^{n-1}\int_{t_{k}}^{t_{k+1}}(s-t_{k})(t_{k+1}-s)(t_{n+{1\over 2}}-s)^{-\alpha-1}ds\\[8pt]
&\dps\leq\frac{\alpha\dps\max_{0<t<t_{n+1}}|\partial_{t}^{2}\phi(\x,t)|\Delta t_{n-1}^{2}}{8\Gamma(1-\alpha)}\sum_{k=0}^{n-1}\int_{t_{k}}^{t_{k+1}}(t_{n+{1\over 2}}-s)^{-\alpha-1}ds\\[8pt]
&\dps=\frac{\alpha\dps\max_{0<t<t_{n+1}}|\partial_{t}^{2}\phi(\x,t)|\Delta t_{n-1}^{2}}{8\Gamma(1-\alpha)}\int_{0}^{t_{n}}(t_{n+{1\over 2}}-s)^{-\alpha-1}ds\\[8pt]
&\dps=\frac{\dps\max_{0<t<t_{n+1}}|\partial_{t}^{2}\phi(\x,t)|\Delta t_{n-1}^{2}}{8\Gamma(1-\alpha)}\Big(\frac{2^{\alpha}}{\Delta t_{n}^{\alpha}}-\frac{1}{t^{\alpha}_{n+\frac{1}{2}}}\Big)\leq\frac{\dps\max_{0<t<t_{n+1}}|\partial_{t}^{2}\phi(\x,t)|\Delta t_{n}^{2-\alpha}}{2^{3-\alpha}\Gamma(1-\alpha)}\\[8pt]
\ery
and
\bry
|\widehat{R}^{n+{1\over 2}}|&\dps\leq\frac{\dps\max_{0<t<t_{n+1}}|\partial_{t}^{2}\phi(\x,t)|}{\Gamma(1-\alpha)}\int_{t_{n}}^{t_{n+{1\over 2}}}(t_{n+{1\over 2}}-s)^{1-\alpha}ds+\mathcal{O}(\Delta t_{n}^{3-\alpha})\\[8pt]
&\dps=\frac{\dps\max_{0<t<t_{n+1}}|\partial_{t}^{2}\phi(\x,t)|\Delta t_{n}^{2-\alpha}}{2^{2-\alpha}(2-\alpha)\Gamma(1-\alpha)}+\mathcal{O}(\Delta t_{n}^{3-\alpha}).\\[8pt]
\ery
Then we complete the proof.
\end{proof}

It is readily seen that the approximation $L^{\alpha}_{n+{1\over 2}}\phi$ to $^{C}_{0}{}\!\!D^{\alpha}_{t_{n+{1\over 2}}}\phi$
is of $2-\alpha$ order,
and the remaining approximation is of second order. Thus the scheme \eqref{L1_CN} is expected to have $2-\alpha$ order
convergence.

\subsection {Stability}
We will provide a stability proof for the scheme \eqref{L1_CN} in the uniform mesh case.  The case of the graded mesh seems
to be much technique, and requires further investigation.
The following lemma plays a key role in proving the stability for the uniform mesh.

\begin{lemma}\label{lem1}
If the mesh is uniform, it holds for any $(u_{0},u_{1},\cdots,u_{n})^T \in\mathbb{R}^{n+1}:$
\bq\label{posit_d}
\dps \sum_{k=0}^{n}\sum_{j=0}^{k}\widetilde{b}_{k-j}u_{j}u_{k}>0,
\eq
where the coefficient set $\{\widetilde{b}_{k}\}_{0}^{n}$ is defined in \eqref{uniform_CN}.
\end{lemma}
\begin{proof}
We first define a piecewise constant function $u(t)$ as follows:
\beq
u(t):=
\begin{cases}
\begin{array}{r@{}l}
&u_{\lfloor t/\tau\rfloor},\quad 0\leq t<(n+1)\tau, \\[9pt]
&0,\quad\quad\mbox{otherwise,}
\end{array}
\end{cases}
\eeq
where $\lfloor t/\tau\rfloor$ denote the integer part of the real number $t/\tau.$
Obviously $u(t)$ is a function in $L^{2}(0,(n+1)\tau)$.
Let's define
\bex
\mathcal{A}^{n}_{\alpha}(u,u):=\frac{1}{\Gamma(1-\alpha)}\int_{0}^{(n+1)\tau}\int_{0}^{t}(t-s)^{-\alpha}u(s)u(t)dsdt.
\eex
It has been proved in \cite{TYZ18} Lemma 2.1 that
\be\label{An}
\mathcal{A}^{n}_{\alpha}(u,u) \ge 0.
\ee
Furthermore, a direct calculation shows
\brr\label{posit_conti}
\mathcal{A}^{n}_{\alpha}(u,u)
&\dps=\frac{1}{2\Gamma(1-\alpha)}\int_{0}^{(n+1)\tau}\int_{0}^{(n+1)\tau}|t-s|^{-\alpha}u(s)u(t)dsdt\\[9pt]
&\dps=\frac{1}{2\Gamma(1-\alpha)}\sum_{k=0}^{n}u_{k}\sum_{j=0}^{n}u_{j}\int_{k\tau}^{(k+1)\tau}\int_{j\tau}^{(j+1)\tau}|t-s|^{-\alpha}dsdt\\[9pt]
&\dps:=\frac{\tau}{2}\Big[\widehat{b}_{0}\sum_{k=0}^{n}u^{2}_{k}+\sum_{k=0}^{n}\sum_{j=0}^{n}\widehat{b}_{|k-j|}u_{k}u_{j}\Big],
\err
where $\widehat{b}_{0}=\frac{1}{\Gamma(3-\alpha)}\tau^{1-\alpha}$ and
\bq\label{coef_hatb}
\dps\widehat{b}_{k}=\frac{\tau^{1-\alpha}}{\Gamma(3-\alpha)}[(k+1)^{2-\alpha}-2k^{2-\alpha}+(k-1)^{2-\alpha}], ~k=1,2,\cdots.
\eq
Let
$$B:=(\widetilde{b}_{0}-\widehat{b}_{0})\sum_{k=0}^{n}u^{2}_{k}+\sum_{k=0}^{n}\sum_{j=0}^{n}(\widetilde{b}_{|k-j|}-\widehat{b}_{|k-j|})u_{j}u_{k}.$$
Then it follows from \eqref{posit_conti} and \eqref{An}
\beq
\dps 2\sum_{k=0}^{n}\sum_{j=0}^{k}\widetilde{b}_{k-j}u_{j}u_{k}=\widetilde{b}_{0}\sum_{k=0}^{n}u^{2}_{k}+\sum_{k=0}^{n}\sum_{j=0}^{n}\widetilde{b}_{k-j}u_{j}u_{k}=B+\frac{2}{\tau}\mathcal{A}^{n}_{\alpha}(u,u)>B.
\eeq
Therefore
\eqref{posit_d} is true if $B>0$, which we want to prove below.\\
Let $c_{k}=\widetilde{b}_{|k|}-\widehat{b}_{|k|},~k=\pm1,\pm2,\cdots,n$. Define the function
$$\dps g(x):=\frac{(x+1)^{2-\alpha}-x^{2-\alpha}}{2-\alpha}, \ \ x\geq0.$$
We deduce from the Taylor expansion and the definitions of $\widetilde{b}_{k}$ and $\widehat{b}_{k}$ in \eqref{uniform_CN} and \eqref{coef_hatb} that $c_{k}=c_{-k}$, and, for $k\geq1$,
\bry
c_{k}&\dps=\widetilde{b}_{|k|}-\widehat{b}_{|k|}=\frac{\tau^{1-\alpha}}{\Gamma(2-\alpha)}\Big[(k+\frac{1}{2})^{1-\alpha}-(k-\frac{1}{2})^{1-\alpha}-\frac{(k+1)^{2-\alpha}-2k^{2-\alpha}+(k-1)^{2-\alpha}}{2-\alpha}\Big]\\[9pt]
&\dps=\frac{\tau^{1-\alpha}}{\Gamma(2-\alpha)}[g'(k-\frac{1}{2})-g(k)+g(k-1)]\\[9pt]
&=\dps\frac{\tau^{1-\alpha}}{\Gamma(2-\alpha)}\Big[g'(k-\frac{1}{2})
-\Big(g(k-\frac{1}{2})+\frac{g'(k-\frac{1}{2})}{2}+\frac{g''(\zeta_{1})}{8}\Big)
+\Big(g(k-\frac{1}{2})-\frac{g'(k-\frac{1}{2})}{2}+\frac{g''(\zeta_{2})}{8}\Big)\Big]\\[9pt]
&=\dps\frac{\tau^{1-\alpha}}{\Gamma(2-\alpha)}\frac{g''(\zeta_{2})-g''(\zeta_{1})}{8},
\ery
where $\zeta_{1}\in (k-\frac{1}{2},k)$ and $\zeta_{2}\in (k-1,k-\frac{1}{2})$.
Since $g'''(x)>0$ for all $x>0$, then it holds $g''(\zeta_{2})<g''(\zeta_{1})$. Therefore, we have $c_{k}=c_{-k}<0$ for $k=1,2,\cdots,n$, and
\brr\label{equ1}
\dps\sum_{k=1}^{n}|c_{k}|&\dps=\sum_{k=1}^{n}\frac{\tau^{1-\alpha}}{\Gamma(2-\alpha)}\Big[\frac{(k+1)^{2-\alpha}-2k^{2-\alpha}+(k-1)^{2-\alpha}}{2-\alpha}-\Big((k+\frac{1}{2})^{1-\alpha}-(k-\frac{1}{2})^{1-\alpha}\Big)\Big]\\[11pt]
&\dps=\frac{\tau^{1-\alpha}}{\Gamma(2-\alpha)}\Big[\frac{(n+1)^{2-\alpha}-2n^{2-\alpha}}{2-\alpha}-\frac{1}{2-\alpha}-(n+\frac{1}{2})^{1-\alpha}+\frac{1}{2^{1-\alpha}}\Big].
\err
Making use of Taylor expansion with integral remainder and differential mean value theorem, we obtain
\brr\label{equ2}
&\dps\frac{(n+1)^{2-\alpha}-2n^{2-\alpha}}{2-\alpha}-(n+\frac{1}{2})^{1-\alpha}\\[9pt]
=&\dps\frac{1}{2-\alpha}\Big\{\Big[(n+\frac{1}{2})^{2-\alpha}+\frac{2-\alpha}{2}(n+\frac{1}{2})^{1-\alpha}+(2-\alpha)(1-\alpha)\int_{n+1/2}^{n+1}(n+1-s)s^{-\alpha}ds\Big]\\[9pt]
&\dps-\Big[(n+\frac{1}{2})^{2-\alpha}-\frac{2-\alpha}{2}(n+\frac{1}{2})^{1-\alpha}+(2-\alpha)(1-\alpha)\int_{n+1/2}^{n}(n-s)s^{-\alpha}ds\Big]\Big\}-(n+\frac{1}{2})^{1-\alpha}\\[9pt]
=&\dps(1-\alpha)\Big[\int_{n+1/2}^{n+1}(n+1-s)s^{-\alpha}ds-\int_{n+1/2}^{n}(n-s)s^{-\alpha}ds\Big]\\[9pt]
=&\dps(1-\alpha)\int_{0}^{1/2}s[(n+1-s)^{-\alpha}-(n+s)^{-\alpha}]d s\leq\dps-\alpha(1-\alpha)(n+1)^{-\alpha-1}\int_{0}^{1/2}s(1-2s)ds\\[9pt]
=&\dps-\frac{\alpha(1-\alpha)}{24}(n+1)^{-\alpha-1}.
\err
Combining \eqref{equ1} and \eqref{equ2} gives
\beq
\dps\sum_{k=1}^{n}|c_{k}|\leq\frac{\tau^{1-\alpha}}{\Gamma(2-\alpha)}\Big[\frac{1}{2^{1-\alpha}}-\frac{1}{2-\alpha}-\frac{\alpha(1-\alpha)}{24}(n+1)^{-\alpha-1}\Big].
\eeq
Let $c_{0}=\frac{\tau^{1-\alpha}}{\Gamma(2-\alpha)}[2^{\alpha}-\frac{2}{2-\alpha}-\frac{\alpha(1-\alpha)}{12}(n+1)^{-\alpha-1}]$.
We construct the matrix $C:=\{c_{k-j}\}_{k,j=0}^{n}$. Then
\beq
\dps\sum_{j=0,j\neq k}^{n}|C_{k,j}|=\sum_{j=0,j\neq k}^{n}|c_{k-j}|<2\sum_{l=1}^{n}|c_{l}|\leq c_{0}, \mbox{ for any $k=0,1,2,\cdots,n$}.
\eeq
This means $C$ is a symmetric positive definite matrix. Thus $\dps\sum_{k=0}^{n}\sum_{j=0}^{n}C_{k,j}u_{j}u_{k}>0$,
for any $(u_{0},u_{1},\cdots,u_{n})^T \in\mathbb{R}^{n+1}$.\\
It then follows from the definitions of $B$ and $C_{k,j}$ that
\bry
\dps B&=\dps\sum_{k=0}^{n}(2\widetilde{b}_{0}-\widehat{b}_{0}-c_{0})u^{2}_{k}+\sum_{k=0}^{n}\sum_{j=0}^{n}C_{k,j}u_{j}u_{k}\\[9pt]
&=\dps\frac{\tau^{1-\alpha}}{\Gamma(2-\alpha)}\frac{\alpha(1-\alpha)}{12}(n+1)^{-\alpha-1}\sum_{k=0}^{n}u^{2}_{k}+\sum_{k=0}^{n}\sum_{j=0}^{n}C_{k,j}u_{j}u_{k}>0.
\ery
This completes the proof.
\end{proof}
Now we are in a position to establish the unconditionnal stability of the L1-CN scheme \eqref{L1_CN}.

\begin{theorem} In the case of uniform mesh,
the L1-CN scheme \eqref{L1_CN} is unconditionally stable in the sense that
a discrete energy is always bounded during the time stepping.
More precisely, it holds
\bq\label{cn_s}
\dps \widetilde{E}^{n}_{\varepsilon,\theta} \le \widetilde{E}^{0}_{\varepsilon,\theta}, \
n=1,2\dots,
\eq
where the discrete energy $E^{n}_{\varepsilon,\theta}$ is defined as:
\bq
\widetilde{E}^{n}_{\varepsilon,\theta}:=\frac{\varepsilon^{2}-\theta}{2}\|\nabla\phi^{n}\|^{2}_{0}+\frac{\theta}{4}\|\nabla(\phi^{n}-\phi^{n-1})\|_{0}^{2}+|R^{n}|^{2}.
\eq
\end{theorem}
\begin{proof}
By taking the inner products of \eqref{L1_CN_1} and \eqref{L1_CN_2} with $\frac{\phi^{n+1}-\phi^{n}}{\tau}$ and $R^{n+1}+R^{n}$ respectively, we have
\bry
 &\dps\Big(L^{\alpha}_{n+{1\over 2}}\phi,\frac{\phi^{n+1}-\phi^{n}}{\tau}\Big)
 +\frac{\varepsilon^{2}-\theta}{2\tau}\big(\|\nabla\phi^{n+1}\|^{2}-\|\nabla\phi^{n}\|^{2}\big)
 +\frac{\theta}{2\tau}\big(\nabla(\phi^{n+1}-2\phi^{n}+\phi^{n-1}),\nabla(\phi^{n+1}-\phi^{n})\big)\\[9pt]
 &\hspace{4cm}\dps+\frac{R^{n+1}+R^{n}}{2\sqrt{E_{\theta}^{n+\frac{1}{2}}+C_{0}}}\Big(-\theta\Delta\phi^{n+\frac{1}{2}}+F'(\phi^{n+\frac{1}{2}}),\frac{\phi^{n+1}-\phi^{n}}{\tau}\Big)=0,\\[20pt]
 &\dps\frac{|R^{n+1}|^{2}-|R^{n}|^{2}}{\tau}
 =\frac{R^{n+1}+R^{n}}{2\sqrt{E_{\theta}^{n+\frac{1}{2}}+C_{0}}}\Big(-\theta\Delta\phi^{n+\frac{1}{2}}+F'(\phi^{n+\frac{1}{2}}),\frac{\phi^{n+1}-\phi^{n}}{\tau}\Big).
\ery
Then we apply the identities
\beq
2(a^{n+1}-a^{n})(a^{n+1}-2a^{n}+a^{n-1})=(a^{n+1}-a^{n})^{2}-(a^{n}-a^{n-1})^{2}+(a^{n+1}-2a^{n}+a^{n-1})^2.
\eeq
 to the above equations
and drop some non-essential positive terms to obtain
\beq
\dps\frac{\widetilde{E}^{n+1}_{\varepsilon,\theta}-\widetilde{E}^{n}_{\varepsilon,\theta}}{\tau}
\leq-\Big(L^{\alpha}_{n+{1\over 2}}\phi,\frac{\phi^{n+1}-\phi^{n}}{\tau}\Big),~~n=0,1,2,\cdots.
\eeq
Summing up the above inequality from 0 to $n$ yields
\bex
\dps\frac{\widetilde{E}^{n+1}_{\varepsilon,\theta}-\widetilde{E}^0_{\varepsilon,\theta}}{\tau}
\leq - \sum^n_{k=0} \Big(L^{\alpha}_{k+{1\over 2}}\phi,\frac{\phi^{k+1}-\phi^k}{\tau}\Big),~~n=0,1,2,\cdots.
\eex
We know from \eqref{LCN} that the right hand side is equal to
$-\sum_{k=0}^{n}\sum_{j=0}^{k}\widetilde{b}_{k-j} \frac{\phi^{j+1}-\phi^j}{\tau} \frac{\phi^{k+1}-\phi^k}{\tau}$, which is non-positive according to
Lemma \ref{lem1}. Thus $\widetilde{E}^{n}_{\varepsilon,\theta} \le \widetilde{E}^0_{\varepsilon,\theta}, n=1,2,\cdots$.
This proves the theorem.
\end{proof}

Since the L1-CN scheme \eqref{L1_CN} also enjoys the unconditional energy stability similar to \eqref{eq2} for the uniform mesh, we can derive $H^{2}$ bound for $\phi^{n}$ similar to Theorem \ref{the1}.
\begin{theorem}\label{the2}
Assume $\phi_{0}\in H^{3}(\Omega)$ and $F\in C^{3}(\mathbb{R})$ satisfying \eqref{asum1}.
Then for the uniform mesh and any $(n+1)<T/\tau$, the scheme \eqref{L1_CN} satisfies
\beq
\dps\sum_{k=0}^{n}\widetilde{b}_{n-k}\|\Delta\phi^{k+1}\|^{2}_{0}+\frac{\tau\varepsilon^{2}}{2}\|\nabla\Delta\phi^{n+1}\|^{2}_{0}\leq c(M_{0},\varepsilon)T+\frac{T^{1-\alpha}}{\Gamma(2-\alpha)}\|\Delta\phi_{0}\|^{2}_{0}+\frac{\tau\varepsilon^{2}}{2}\|\nabla\Delta\phi_{0}\|^{2}_{0},
\eeq
where $M_{0}$ is a positive constant depending only on $\Omega$ and $\phi_{0}$.
\end{theorem}

The L1-CN scheme \eqref{L1_CN} can be efficiently implemented
by following the lines similar to the algorithm $(i)$-$(ii)$ in the previous section.

\section{Numerical results}
\label{sect4}
\setcounter{equation}{0}

In this section, we present some numerical examples to demonstrate the efficiency of the proposed schemes
in terms of stability and accuracy. In all tests that follow, we always set $\theta=\varepsilon^{2}$ in the schemes.
The spatial discretization is the Fourier method or Legendre spectral method using numerical quadratures.
In order to test the accuracy, the error is measured by the maximum norm, i.e., $\dps\max_{1\leq n\leq M}\|\phi^{n}-\phi(t_{n})\|_{\infty}$
or $\|\phi^{M}-\phi(T)\|_{\infty}$ depending on whether or not the exact solution is exactly known.
It is worth to mention that a fast evaluation technique
based on the so-called sum-of-exponentials approach to the time fractional derivative
is used to accelerate the calculation (also to reduce the storage); see, e.g., \cite{JZZ17,YSZ17}.

\subsection{Convergence order test}
\begin{example}\label{expl1}
Consider the following time fractional Allen-Cahn equation with the periodic boundary condition:
\bq\label{TAC}
\dps^{C}_{0}{}\!\!D^{\alpha}_{t}\phi-\varepsilon^{2}\Delta \phi-\phi(1-\phi^{2})=s(\x,t),\quad
(\x,t)\in (0,2\pi)^2\times(0,T],
\eq
 where $s(\x,t)$ is a fabricated source term chosen such that
the exact solution is
$$\phi(\x,t)=0.2 t^5 \sin(x)\cos(y).$$

\end{example}
The spatial discretization used in the calculation is
the Fourier method with $128\times128$ modes.
It has been checked that this Fourier mode number is large enough
so that the spatial discretization error is negligible compared to the temporal discretization.
We present in Figure \ref{fig1}
the error at $T=1$ as functions of the time
step sizes in log-log scale. It is clearly observed that the first order scheme \eqref{semi_D} and L1-CN scheme \eqref{L1_CN} achieve the expected convergence rate  for all tested $\alpha$, i.e., first order and $2-\alpha$ order, respectively. Note that there is no
numerical instability observed during the calculation for all the time step sizes we tested.

\begin{figure*}[htbp]
\begin{minipage}[t]{0.49\linewidth}
\centerline{\includegraphics[scale=0.55]{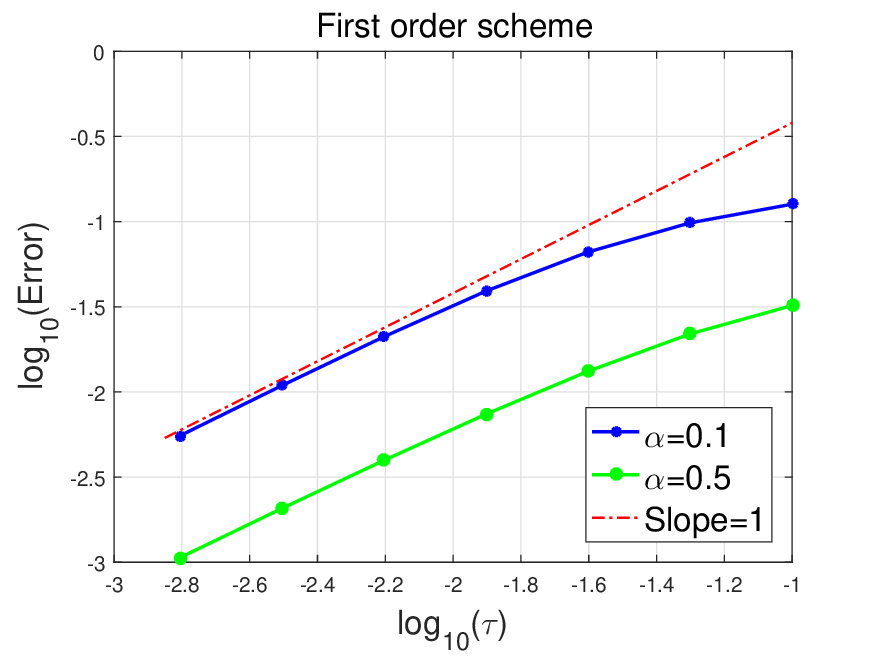}}
\centerline{(a) first order scheme for $\alpha=$0.1 and 0.5}
\end{minipage}
\vskip 3mm
\begin{minipage}[t]{0.49\linewidth}
\centerline{\includegraphics[scale=0.55]{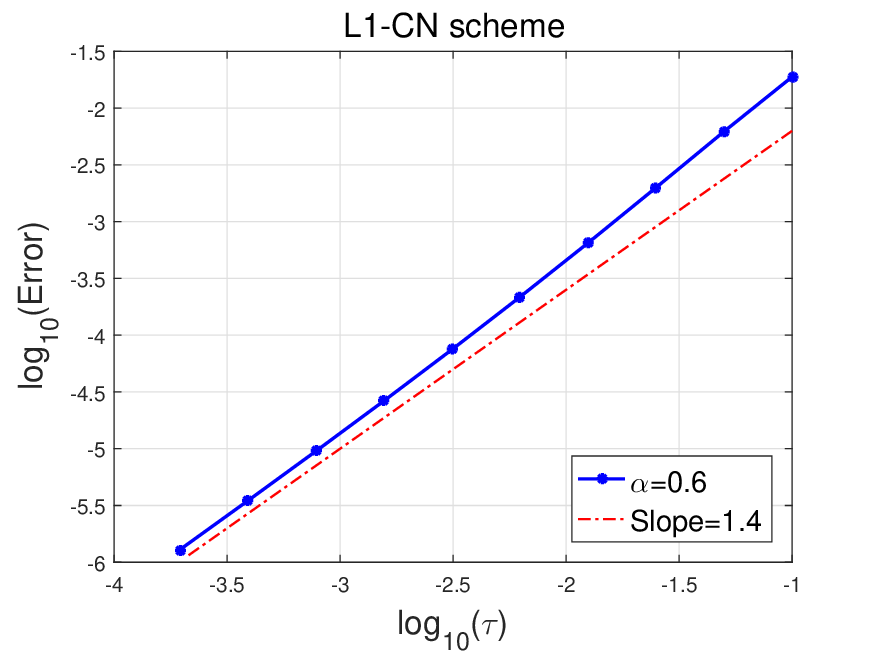}}
\centerline{(b) L1-CN scheme for $\alpha=0.6$}
\end{minipage}
\begin{minipage}[t]{0.49\linewidth}
\centerline{\includegraphics[scale=0.55]{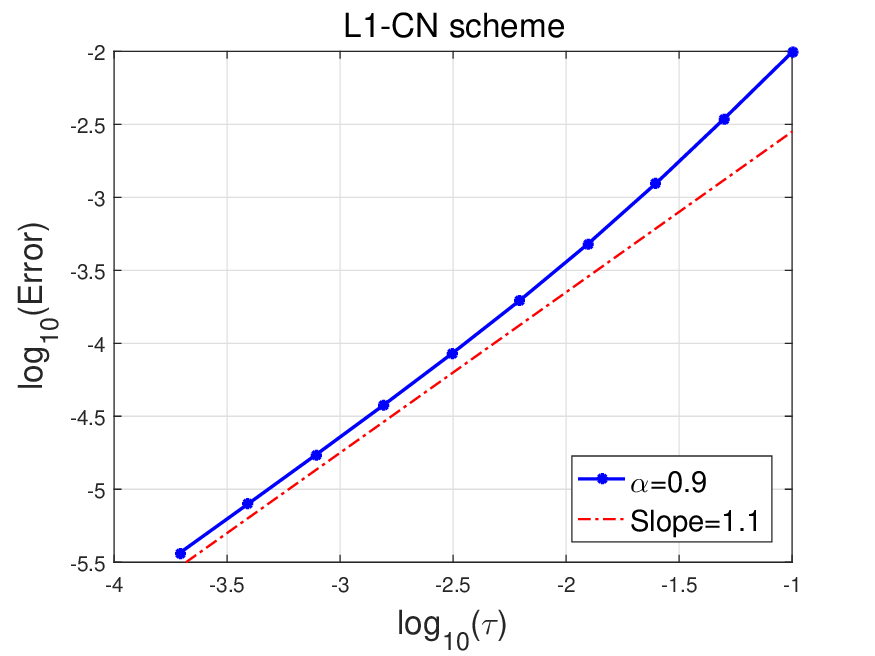}}
\centerline{(c) L1-CN scheme for $\alpha=0.9$}
\end{minipage}
\caption{
Example \ref{expl1}: Error decay at $T=1$ versus the time step sizes for the first order scheme and L1-CN scheme with $C_{0}=0$.
}\label{fig1}
\end{figure*}

\begin{example}\label{expl2}
Consider the same equation as in Example \ref{expl1}, but the Neumann boundary condition, with the exact solution
$$\phi(\x,t)=0.2(t^{\mu}+1)\cos(\pi x)\cos(\pi y), ~(\x,t)\in (-1,1)^2\times(0,T],$$
which has limited regularity at the initial time $t=0$.
\end{example}
For this Neumann problem,
the space variable is discretized by the Legendre Galerkin spectral method using polynomials of degree 32
in each spatial direction.  The purpose of
this test is to not only verify the accuracy of the schemes,
but also investigate how the regularity of the solution affects the accuracy.
In particular, we are interested to study the impact of the graded mesh parameter $r$ on the convergence rate.
The calculation is performed by using the L1-CN scheme with $M=2^{k},k=4,5,\cdots,13$.
In Figure \ref{fig2}, we
plot the $ L^{\infty}$ errors in log-log scale with respect to the maximum time step sizes $\tau$ for each fixed $M$.
It is shown in the figure that the numerical solution achieves the convergence rate $\min\{\mu r,2-\alpha\}$.
It is worth to point out that the error curves can be misleading if one only looks at the error behaviour for relatively larger time step sizes.
For example, the numerical results given in
Figure \ref{fig2} (d) and (f) for the case $\mu r>2-\alpha$ seem to make us expect a higher order accuracy.
More precisely, the error curve in Figure \ref{fig2} (d) indicates a convergence rate like $\mu r$ in the range of larger $\tau$.
However when $\tau$ decreases we clearly observe the convergence order $2-\alpha$, which is what we can expect
from a truncation error analysis for the L1 scheme.

\begin{figure*}[htbp]
\begin{minipage}[t]{0.49\linewidth}
\centerline{\includegraphics[scale=0.55]{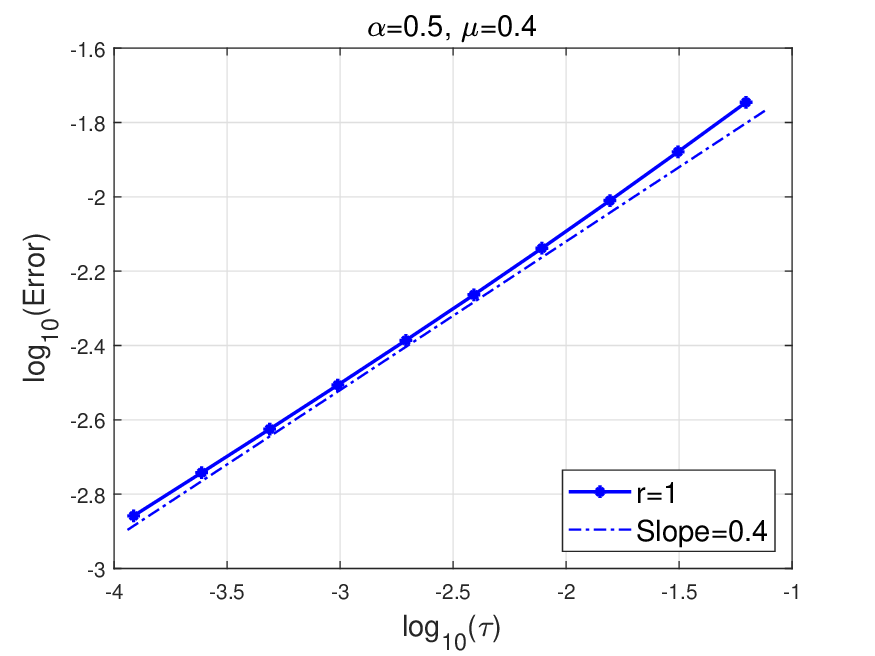}}
\centerline{(a) $r=1$}
\end{minipage}
\begin{minipage}[t]{0.49\linewidth}
\centerline{\includegraphics[scale=0.55]{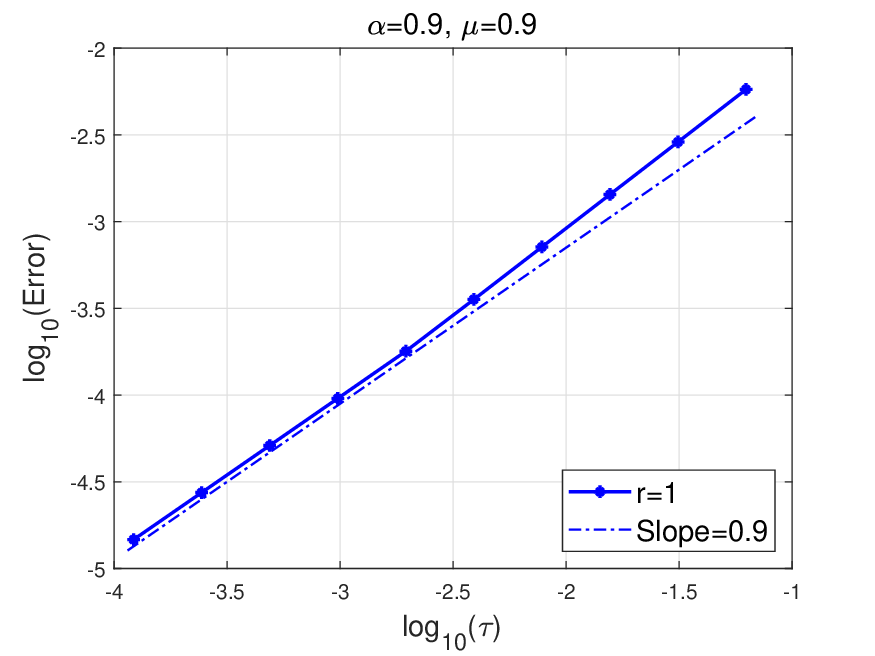}}
\centerline{(b) $r=1$}
\end{minipage}
\vskip 3mm
\begin{minipage}[t]{0.49\linewidth}
\centerline{\includegraphics[scale=0.55]{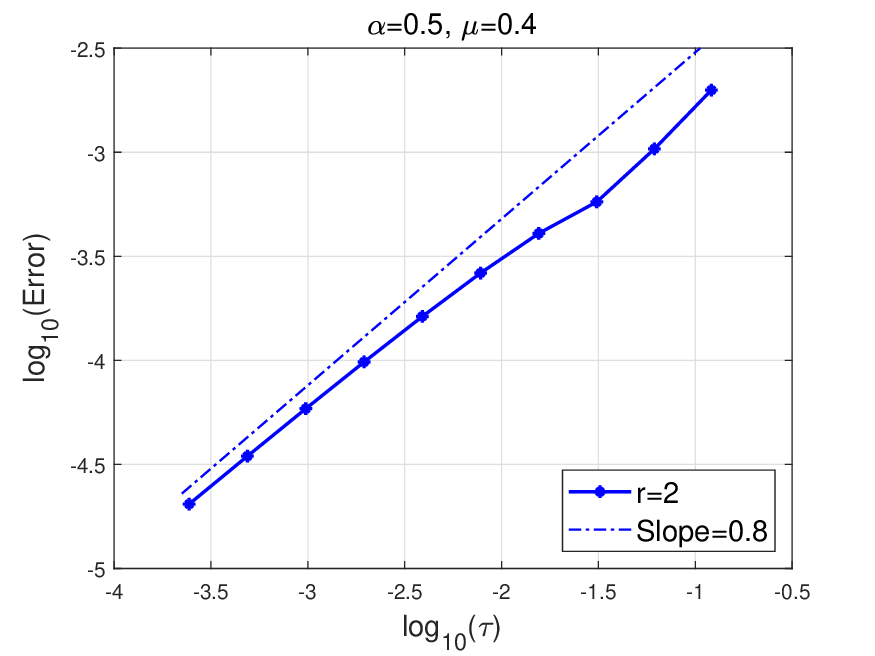}}
\centerline{(c) $r=2$}
\end{minipage}
\begin{minipage}[t]{0.49\linewidth}
\centerline{\includegraphics[scale=0.55]{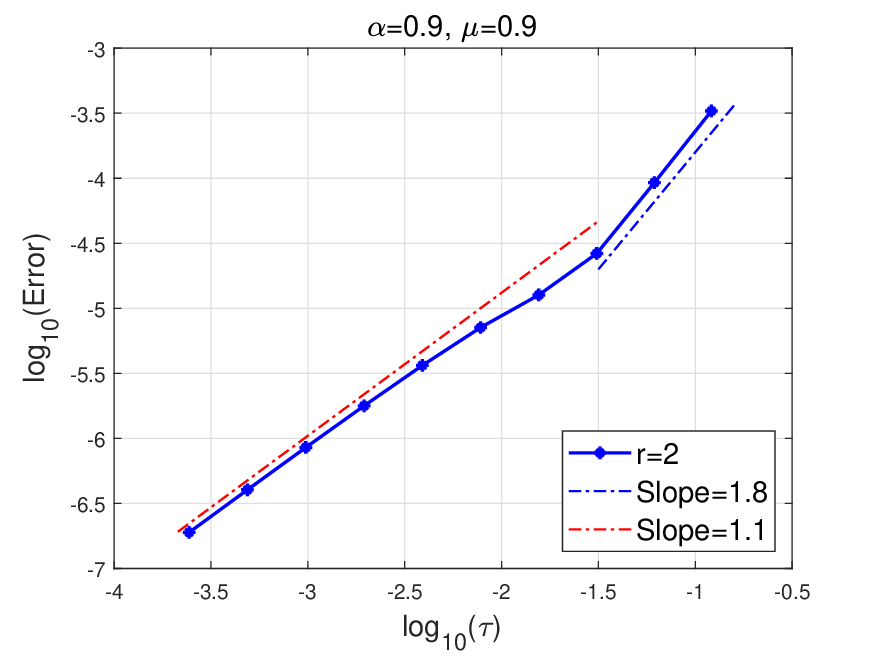}}
\centerline{(d) $r=2$}
\end{minipage}
\vskip 3mm
\begin{minipage}[t]{0.49\linewidth}
\centerline{\includegraphics[scale=0.55]{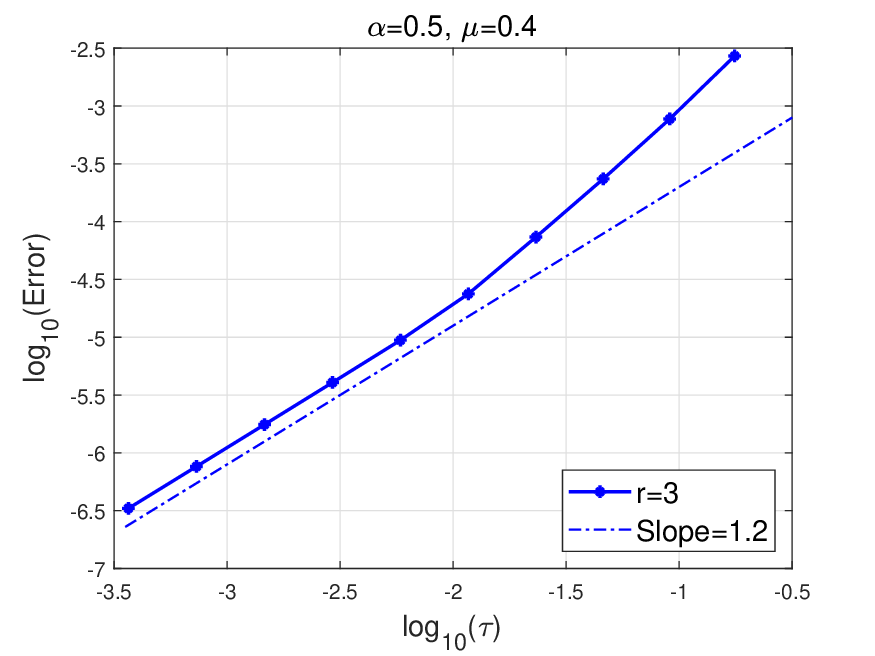}}
\centerline{(e) $r=3$}
\end{minipage}
\begin{minipage}[t]{0.49\linewidth}
\centerline{\includegraphics[scale=0.55]{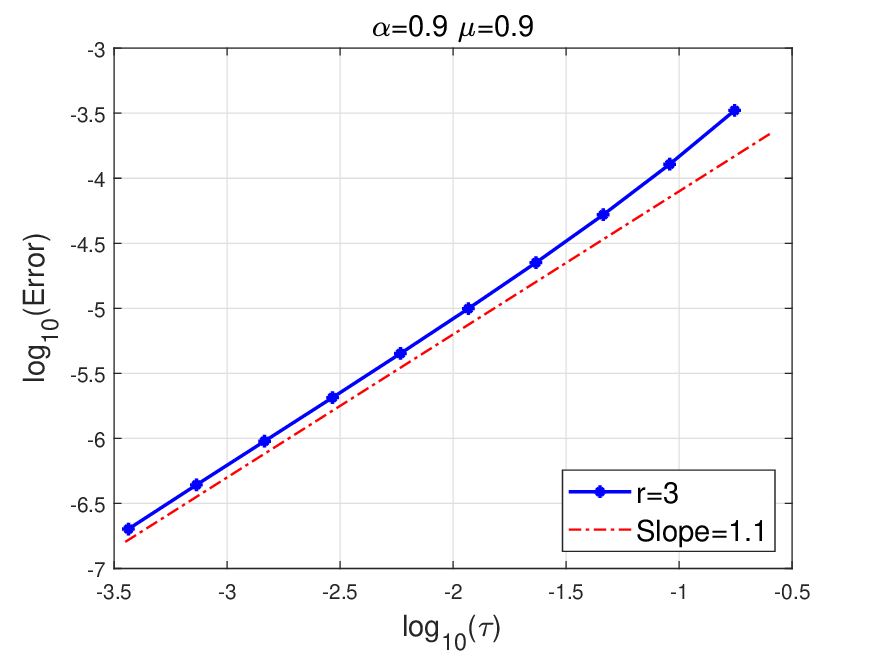}}
\centerline{(f) $r=3$}
\end{minipage}
\caption{(Example \ref{expl2}) Error history for the L1-CN scheme with $C_0=0$ for different values of the mesh parameter $r$:
the left side figures correspond to $\alpha=0.5$ and $\mu=0.4$;
the right side are for $\alpha=0.9$ and $\mu=0.9$.
}\label{fig2}
\end{figure*}

\begin{example}\label{expl3}
Consider the time fractional Allen-Cahn Neumann problem in the domain $(-1,1)\times(-1,1)$ with the double well potential
$F(\phi):=\dfrac{(\phi^2-1)^{2}}{4}$ and the initial condition
$\phi(\x,0)=\cos(4\pi x)\cos(4\pi y)$.
\end{example}

Again, we use the Legendre Galerkin method for the spatial
discretization with high enough mode number to avoid possible spatial error contamination.
The error behavior is investigated by comparing the computed
solutions to the one obtained by using the L1-CN scheme \eqref{L1_CN} with $r=3, M=5\times 10^{4}$
and Legendre polynomial space of degree $32\times32$, considered as the ``exact" solution.
In Figure \ref{fig3}, we
plot the $L^{\infty}$ errors at $T=1$ in log-log scale with respect to the time discretization parameter $\tau$. Once again the
observed error behavior in Figure \ref{fig3} demonstrates that
the proposed L1-CN scheme \eqref{L1_CN} attains the convergence rate $\min\{\alpha r,2-\alpha\}$ for all tested $\alpha$ and $r$.

The stability of the schemes is investigated using this example through a long time
calculation up to $T = 100$.
In view of the possible singularity feature of the solution,
we split the interval $[0,T]$ into two parts: $[0,1]$ and $(1,T]$.
We first compute the solution by using the graded mesh with $r=\frac{2-\alpha}{\alpha}$ in the first subinterval $[0,1]$,
then use the uniform mesh with the time step size $\Delta t$ in the second subinterval $(1,T]$.
The computed discrete energies, both $E(\phi^{n})$ and $\widetilde{E}^{n}_{\varepsilon,\theta}$,
by using the L1-CN scheme are presented in Figure \ref{fig4} (b), (d), and (f),
showing dissipative
feature during the running time for all tested values of $\alpha, M$ and $\Delta t$.
This confirms the unconditional stability of the proposed scheme.
We would like to emphasize that while the modified energy $\widetilde{E}^{n}_{\varepsilon,\theta}$ keeps dissipative
all the time, the original energy $E(\phi^{n})$ may exhibit oscillation at some time instants, especially when large time step size is
used.
This has been a well known fact in the SAV approach; see, e.g., \cite{HAX19} and the reference therein.
Thus it is suggested that, in order to avoid undesirable oscillation of the original energy in real applications,
the time step size should not be taken too large,
even though such oscillations do not necessarily make the calculation unstable.

\begin{figure*}[htbp]
\begin{minipage}[t]{0.49\linewidth}
\centerline{\includegraphics[scale=0.55]{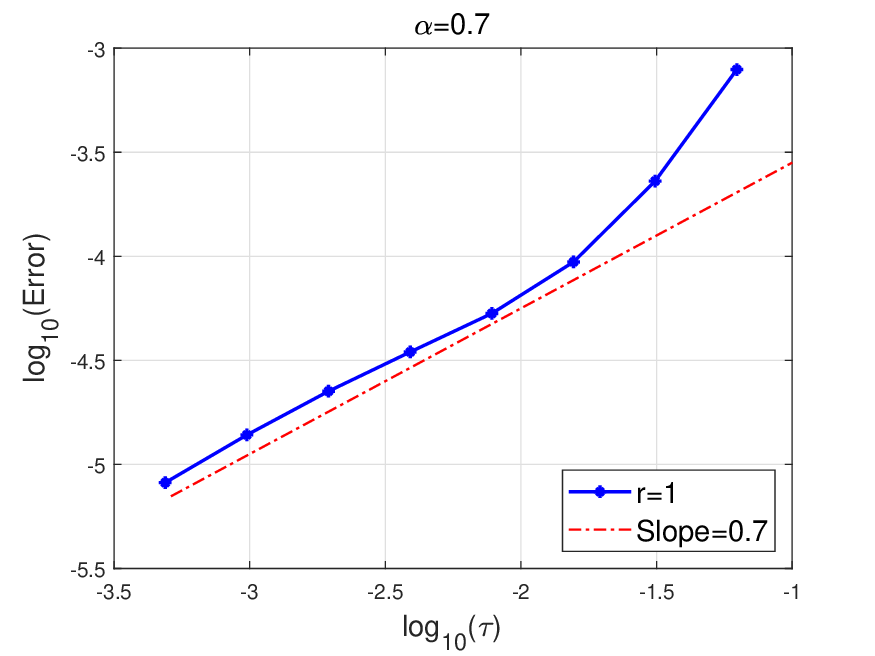}}
\centerline{(a) $r=1$}
\end{minipage}
\begin{minipage}[t]{0.49\linewidth}
\centerline{\includegraphics[scale=0.55]{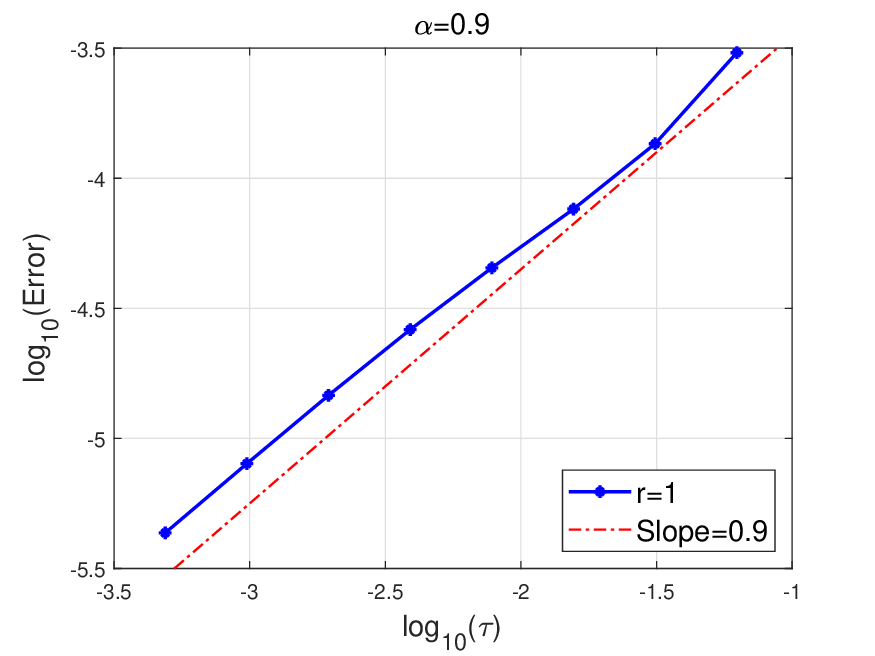}}
\centerline{(b) $r=1$}
\end{minipage}
\vskip 3mm
\begin{minipage}[t]{0.49\linewidth}
\centerline{\includegraphics[scale=0.55]{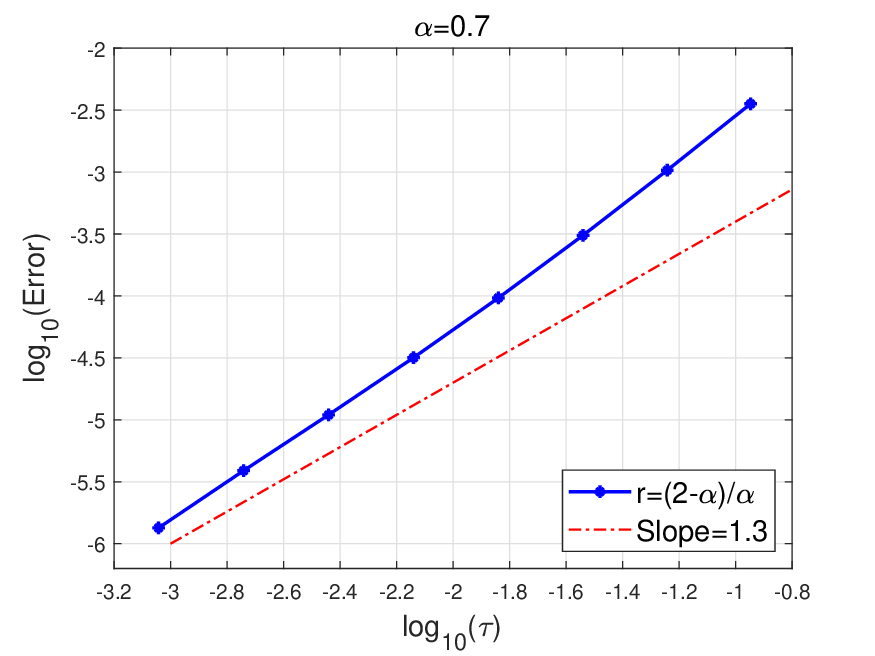}}
\centerline{(c) $r=\dfrac{2-\alpha}{\alpha}$}
\end{minipage}
\begin{minipage}[t]{0.49\linewidth}
\centerline{\includegraphics[scale=0.55]{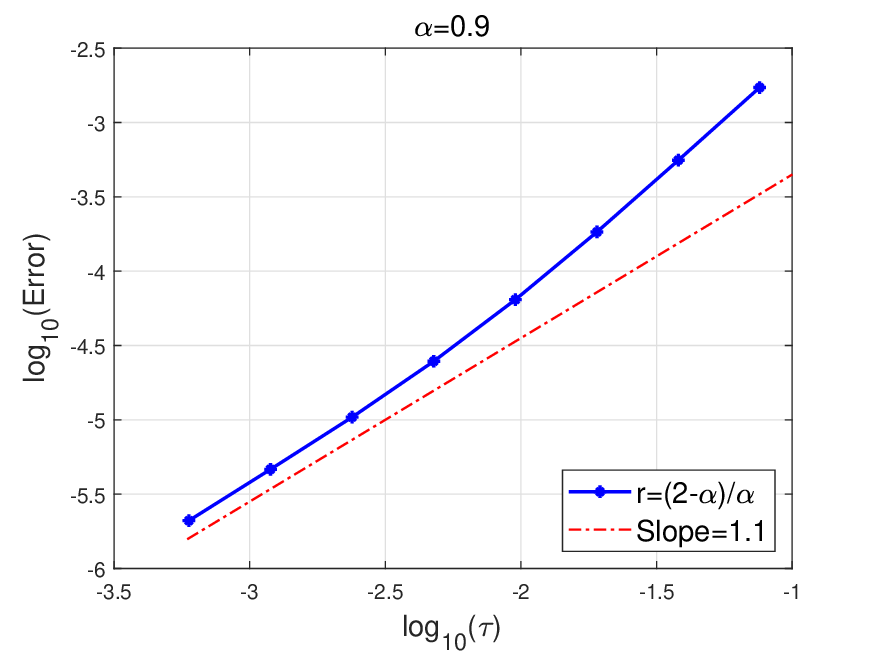}}
\centerline{(d) $r=\dfrac{2-\alpha}{\alpha}$}
\end{minipage}
\vskip 3mm
\begin{minipage}[t]{0.49\linewidth}
\centerline{\includegraphics[scale=0.55]{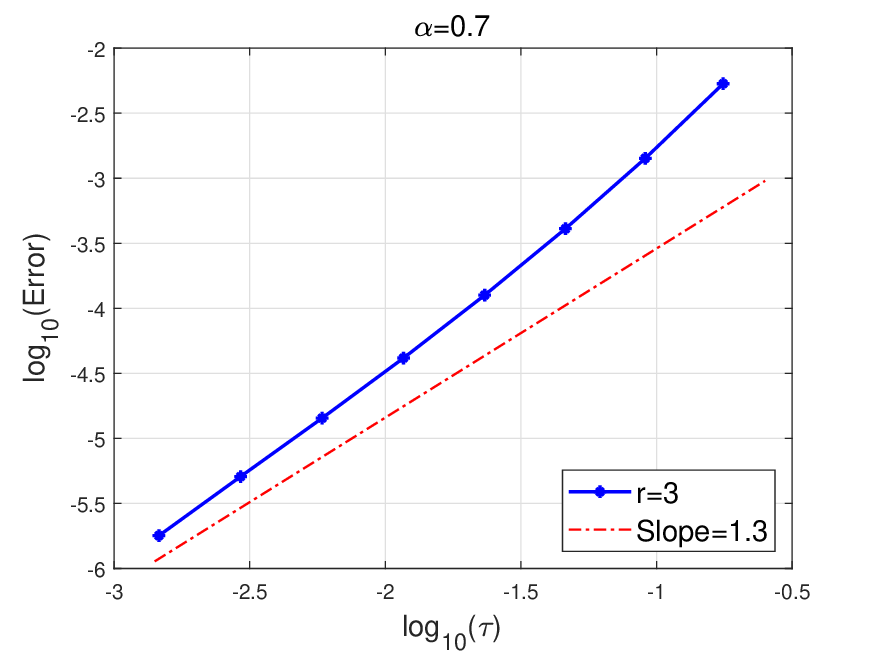}}
\centerline{(e) $r=3$}
\end{minipage}
\begin{minipage}[t]{0.49\linewidth}
\centerline{\includegraphics[scale=0.55]{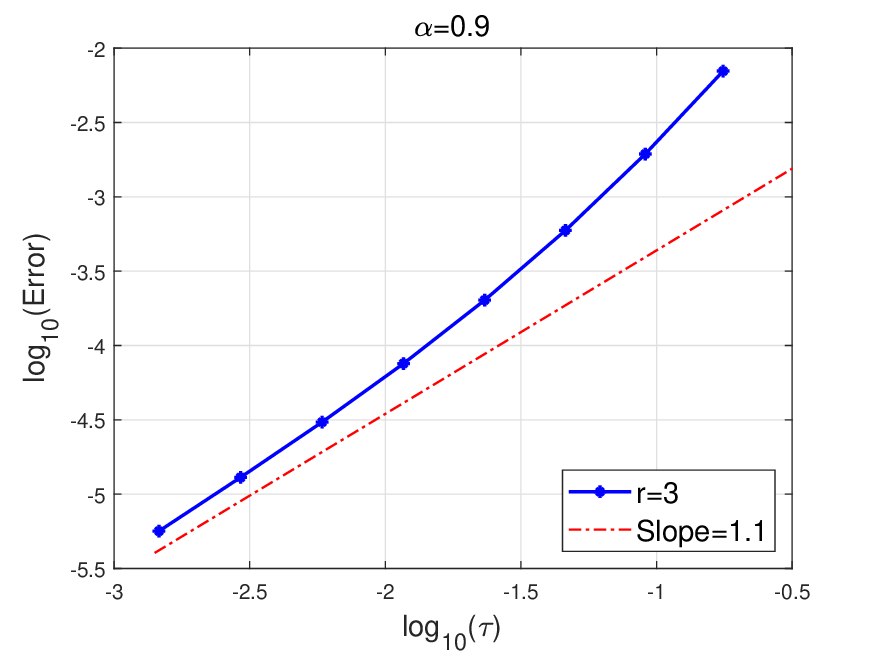}}
\centerline{(f) $r=3$}
\end{minipage}
\caption{(Example \ref{expl3}) Effect of the mesh parameter on the convergence rate of the L1-CN scheme:
$\alpha=0.7$ (left) and  $\alpha=0.9$ (right).
}\label{fig3}
\end{figure*}

\begin{figure*}[htbp]
\begin{minipage}[t]{0.49\linewidth}
\centerline{\includegraphics[scale=0.55]{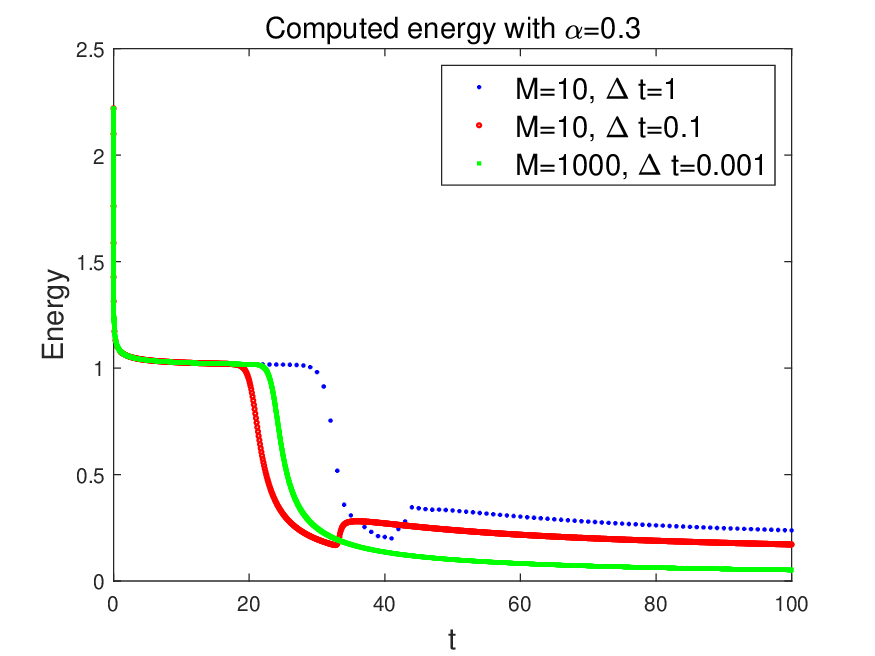}}
\centerline{(a) $\alpha=0.3$}
\end{minipage}
\begin{minipage}[t]{0.49\linewidth}
\centerline{\includegraphics[scale=0.55]{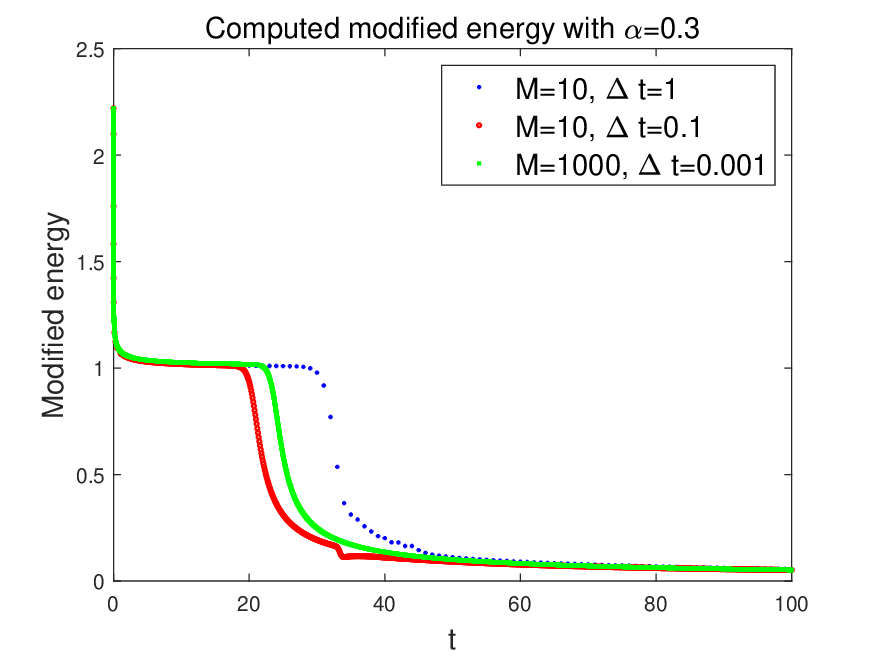}}
\centerline{(b) $\alpha=0.3$}
\end{minipage}
\vskip 3mm
\begin{minipage}[t]{0.49\linewidth}
\centerline{\includegraphics[scale=0.55]{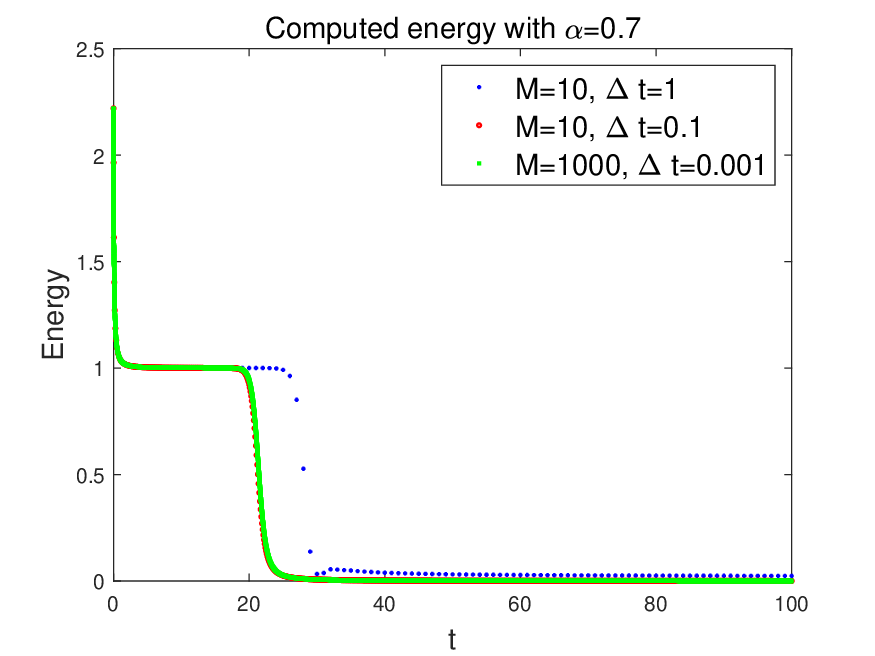}}
\centerline{(c) $\alpha=0.7$}
\end{minipage}
\begin{minipage}[t]{0.49\linewidth}
\centerline{\includegraphics[scale=0.55]{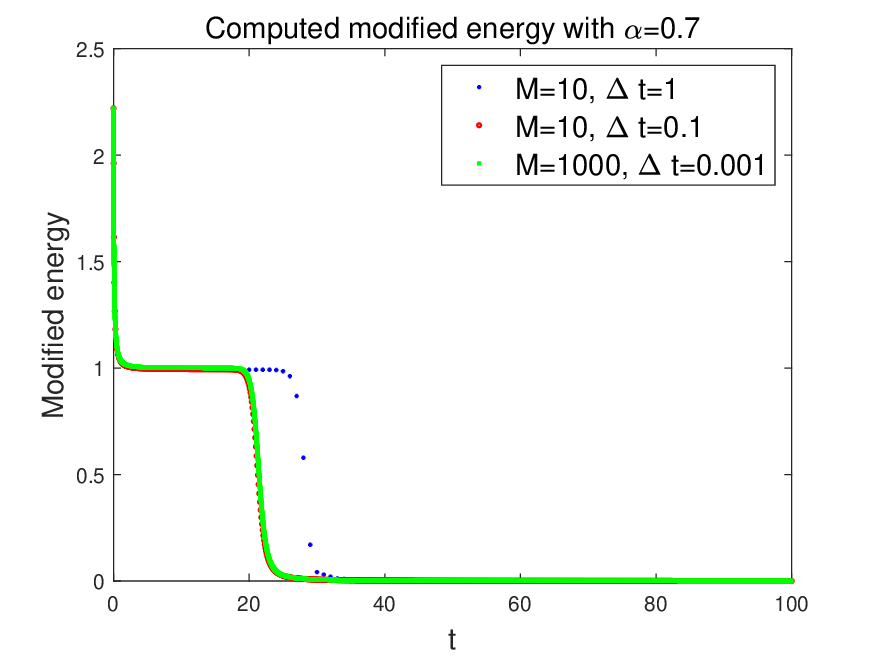}}
\centerline{(d) $\alpha=0.7$}
\end{minipage}
\vskip 3mm
\begin{minipage}[t]{0.49\linewidth}
\centerline{\includegraphics[scale=0.55]{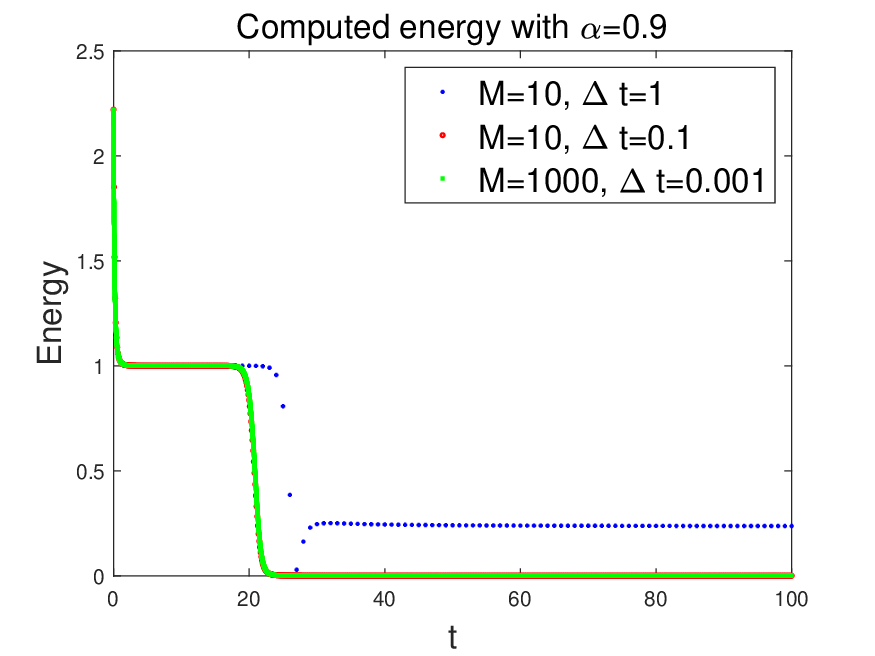}}
\centerline{(e) $\alpha=0.9$}
\end{minipage}
\begin{minipage}[t]{0.49\linewidth}
\centerline{\includegraphics[scale=0.55]{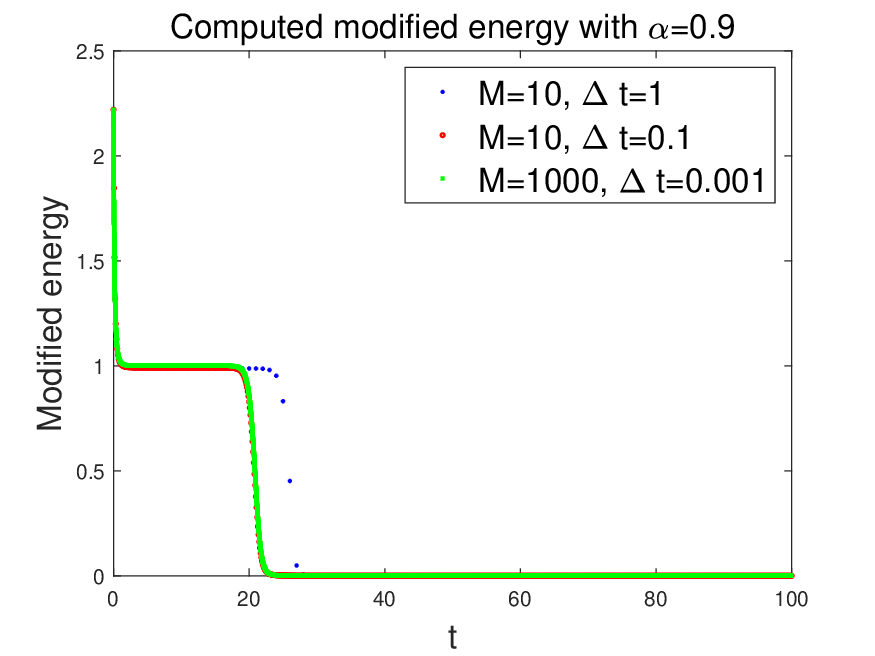}}
\centerline{(f) $\alpha=0.9$}
\end{minipage}
\caption{(Example \ref{expl3}) Dissipation feature of the computed energy $E(\phi^{n})$ and modified energy $\widetilde{E}^{n}_{\varepsilon,\theta}$
with different time step sizes using the L1-CN scheme.
}\label{fig4}
\end{figure*}

\subsection{Order sensibility of a benchmark problem}

Now we consider an interface moving problem governed by
the following fractional Allen-Cahn equation in the rectangular domain $\Omega:=(-32,32)^2$
\bq\label{bkt}
\begin{cases}
\begin{array}{r@{}l}
&\dps^{C}_{0}{}\!\!D^{\alpha}_{t}\phi- \Delta \phi-\phi(1-\phi^{2})=0,\\[9pt]
&\dps\frac{\partial\phi}{\partial \n}\big|_{\partial\Omega}=0, \phi(\x,0)=\phi_{0}(\x).
\end{array}
\end{cases}
\eq
The initial state of the interface is a circle of the radius $R_{0}$. In the classical case, i.e., $\alpha=1$,
this interface moving problem has been frequently served as benchmark problem to validate numerical methods.
In this case it has been known \cite{Chen98} that such a circle interface is unstable, and
will shrink and disappear due to the driving force. In fact,
it was shown in \cite{Yan16} that
the radius $R(t)$ of the circle at the given time $t$ evolves as
\bq\label{eqr}
R(t)=\sqrt{R_{0}^{2}-2t}.
\eq
The purpose of this example is twofold: demonstration of the efficiency of the proposed schemes, and sensibility investigation
of how the fractional order $\alpha$ effects the interface evolution.
To this end we consider the circle interface of the radius $R_{0}=8$. We first transform
the computational domain into the reference domain $(-1,1)^2$, then
the equation to be solved reads:
\be\label{ex4}
&\dps^{C}_{0}{}\!\!D^{\alpha}_{t}\phi-\varepsilon^{2}\Delta \phi-\phi(1-\phi^{2})=0,
\ee
where $\varepsilon=0.0313$. The initial condition becomes
\beq
\phi_{0}(\x)=\begin{cases}
\begin{array}{r@{}l}
1,&\quad |\x|^2<\big(\frac{8}{32}\big)^{2},\\[9pt]
-1,&\quad |\x|^2\geq\big(\frac{8}{32}\big)^{2}.
\end{array}
\end{cases}
\eeq
We numerically solve this problem by using
the L1-CN time scheme \eqref{L1_CN} based on the graded mesh with $r=\frac{2-\alpha}{\alpha}$ in the
subinterval $[0,1]$ and the uniform mesh with the time step size $\Delta t$ in the subinterval $(1,T]$.
The Legendre spectral method for the spatial discretization make use of polynomials of degree $128$ in each direction.
In Figure \ref{fig5}
we investigate the impact of the fractional order $\alpha$ on the interface evolution.
First, the result for $\alpha=1$ indicates that the circle interface almost disappears at $T=32$,
which is in a perfect
agreement with the prediction \eqref{eqr}. This partially demonstrates the accuracy of the scheme.
The second observation is that
the interface shirking dynamics exists for all $\alpha$,
but the shirking speed slows down significantly
when the fractional order decreases. This is not really surprising since the fractional derivative represents
some kind of memory effect.
More details about the evolution of the circle radius are presented in Figure \ref{fig6}(a), which confirms the result
reported in \cite{DYZ19}. That is,
the circle shrinking rate seems to exhibit a power-law scaling, although the exact rate remains unknown.
Also shown in Figure \ref{fig6}(b) is the evolution of the total free energy, from which
we see that the free energy is always dissipative for all tested $\alpha$.
This further increases the credibility of the simulation results.

\begin{figure*}[htbp]
\begin{minipage}[t]{0.19\linewidth}
\centerline{\includegraphics[width=4cm,height=3.5cm]{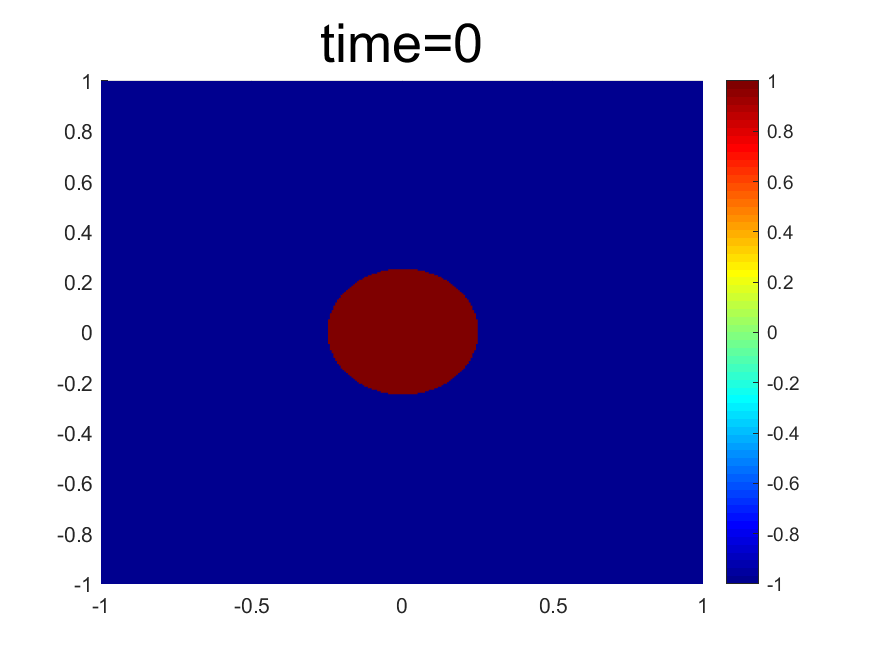}}
\centerline{}
\end{minipage}
\begin{minipage}[t]{0.19\linewidth}
\centerline{\includegraphics[width=4cm,height=3.5cm]{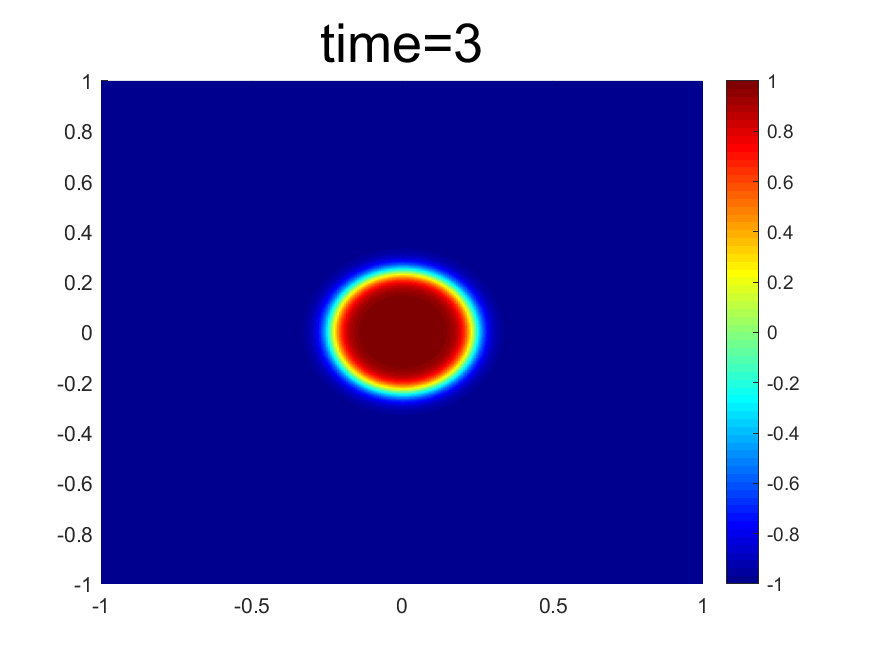}}
\centerline{}
\end{minipage}
\begin{minipage}[t]{0.19\linewidth}
\centerline{\includegraphics[width=4cm,height=3.5cm]{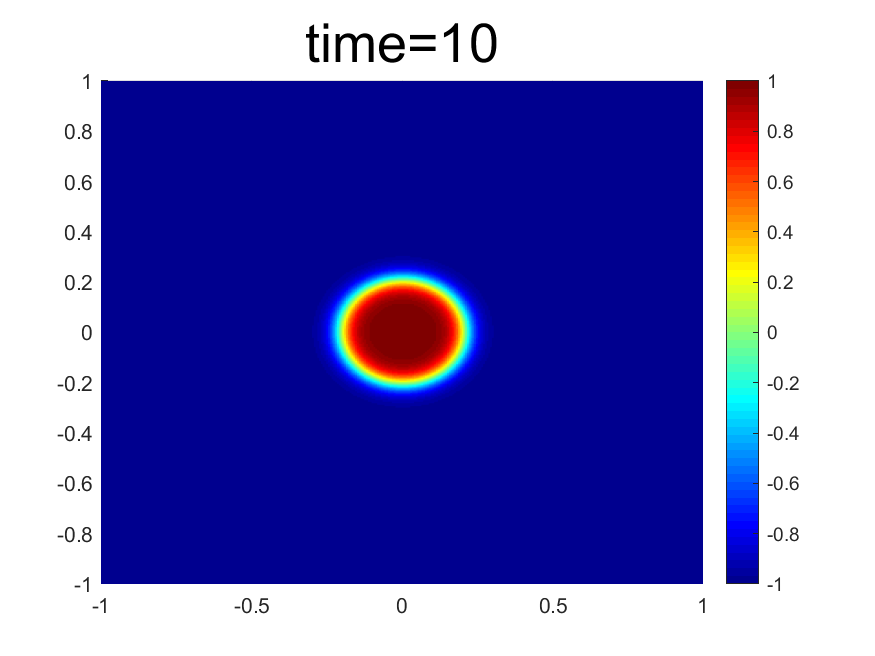}}
\centerline{ (a) $\alpha=1$ and $\Delta t=0.01$}
\end{minipage}
\begin{minipage}[t]{0.19\linewidth}
\centerline{\includegraphics[width=4cm,height=3.5cm]{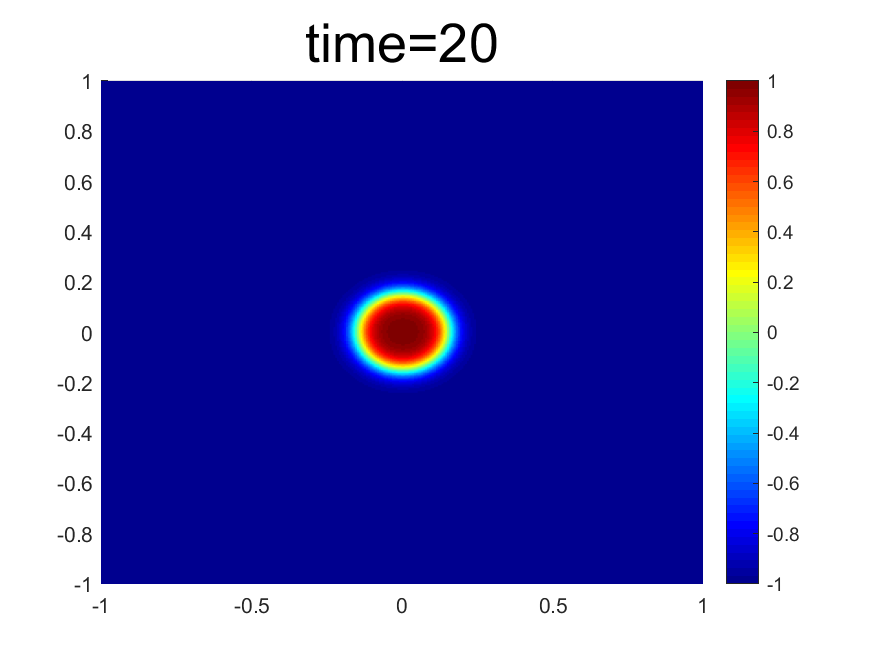}}
\centerline{}
\end{minipage}
\begin{minipage}[t]{0.19\linewidth}
\centerline{\includegraphics[width=4cm,height=3.5cm]{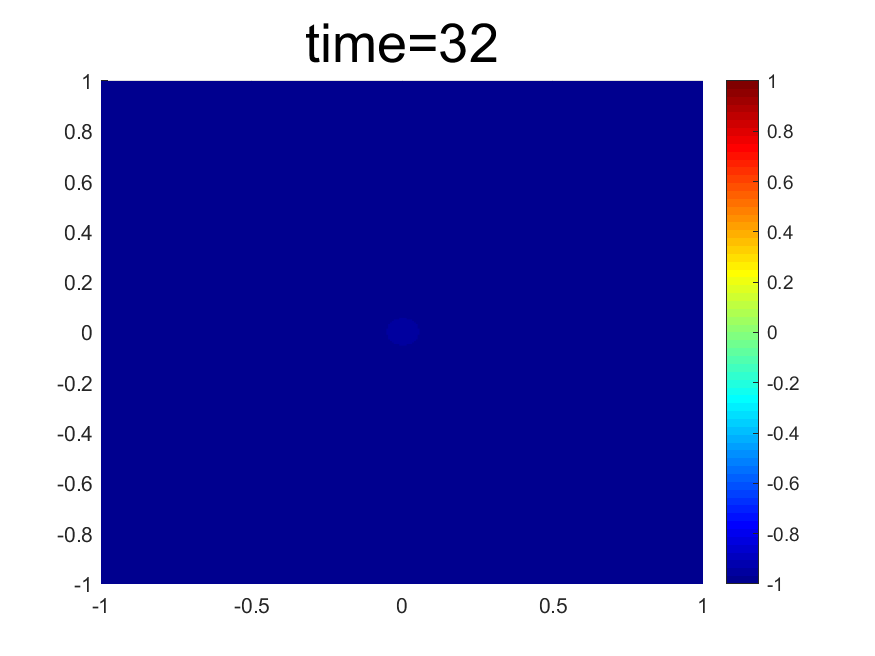}}
\centerline{}
\end{minipage}
\vskip 3mm
\begin{minipage}[t]{0.19\linewidth}
\centerline{\includegraphics[width=4cm,height=3.5cm]{FAC4_1.eps}}
\centerline{}
\end{minipage}
\begin{minipage}[t]{0.19\linewidth}
\centerline{\includegraphics[width=4cm,height=3.5cm]{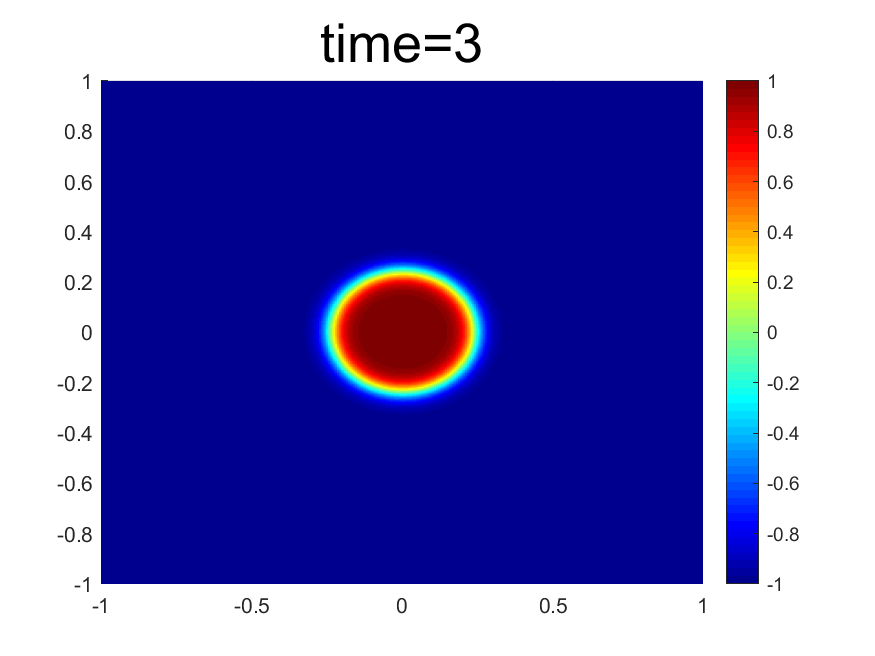}}
\centerline{}
\end{minipage}
\begin{minipage}[t]{0.19\linewidth}
\centerline{\includegraphics[width=4cm,height=3.5cm]{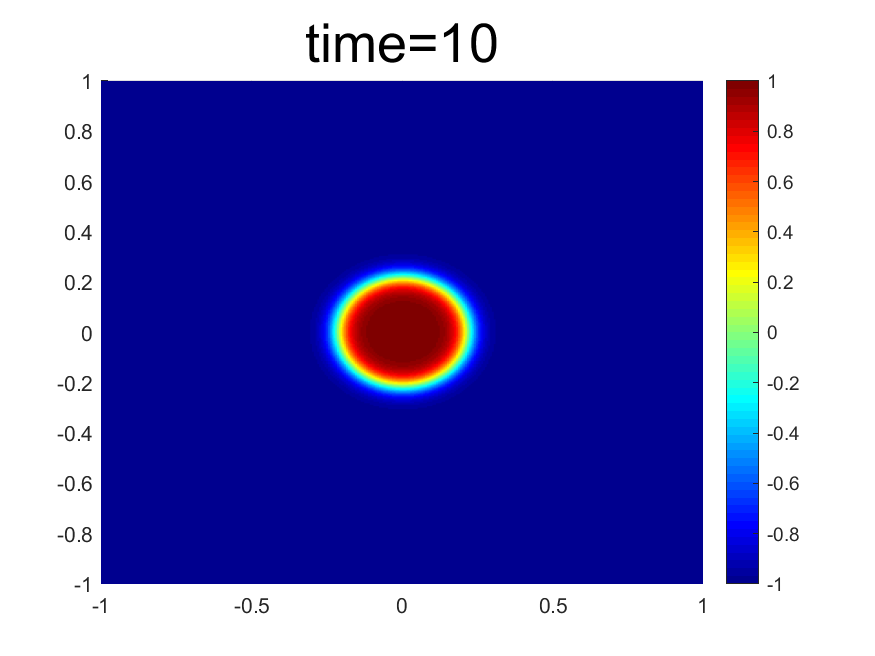}}
\centerline{ (b) $\alpha=0.9, M=100$, and $\Delta t=0.01$}
\end{minipage}
\begin{minipage}[t]{0.19\linewidth}
\centerline{\includegraphics[width=4cm,height=3.5cm]{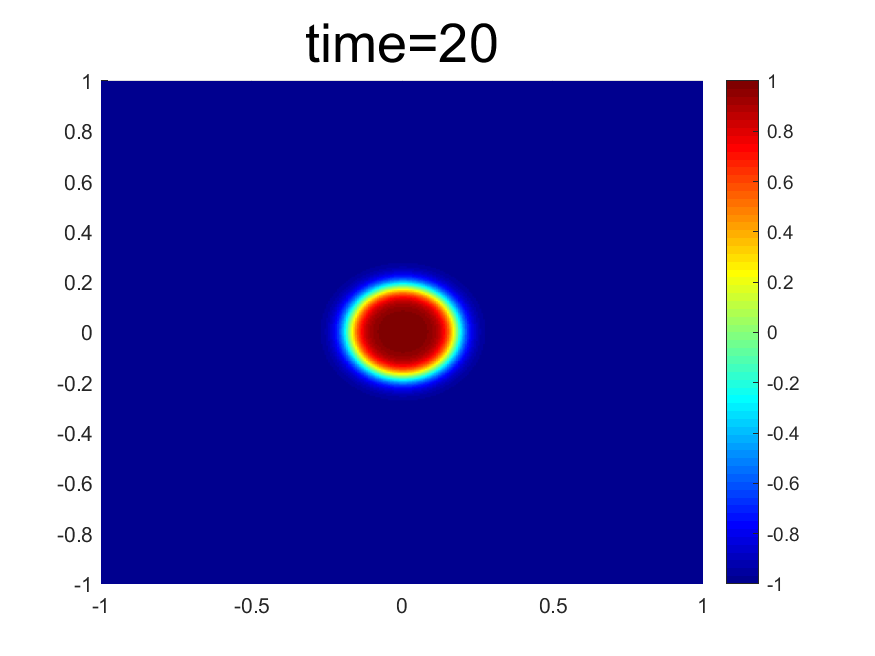}}
\centerline{}
\end{minipage}
\begin{minipage}[t]{0.19\linewidth}
\centerline{\includegraphics[width=4cm,height=3.5cm]{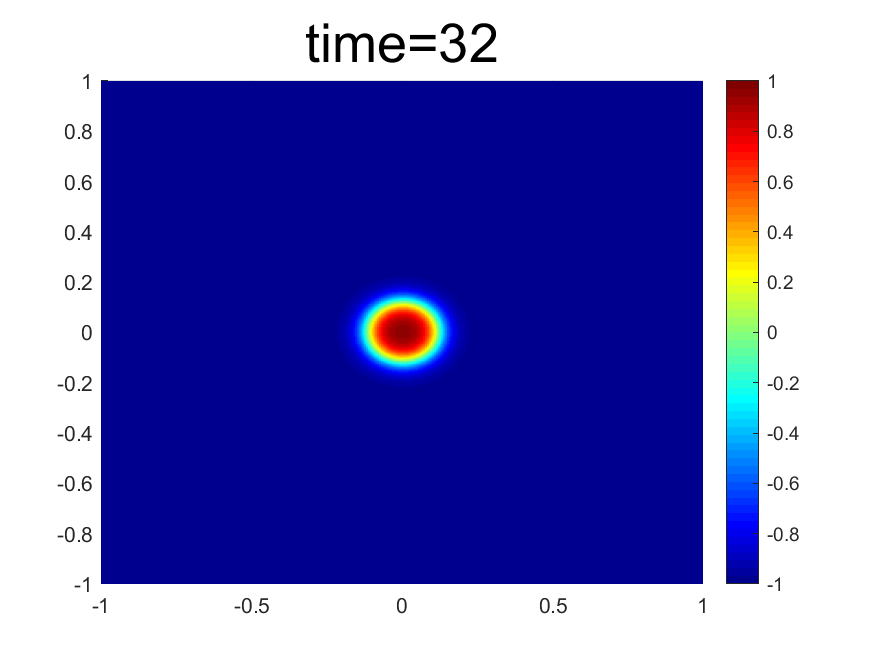}}
\centerline{}
\end{minipage}
\vskip 3mm
\begin{minipage}[t]{0.19\linewidth}
\centerline{\includegraphics[width=4cm,height=3.5cm]{FAC4_1.eps}}
\centerline{}
\end{minipage}
\begin{minipage}[t]{0.19\linewidth}
\centerline{\includegraphics[width=4cm,height=3.5cm]{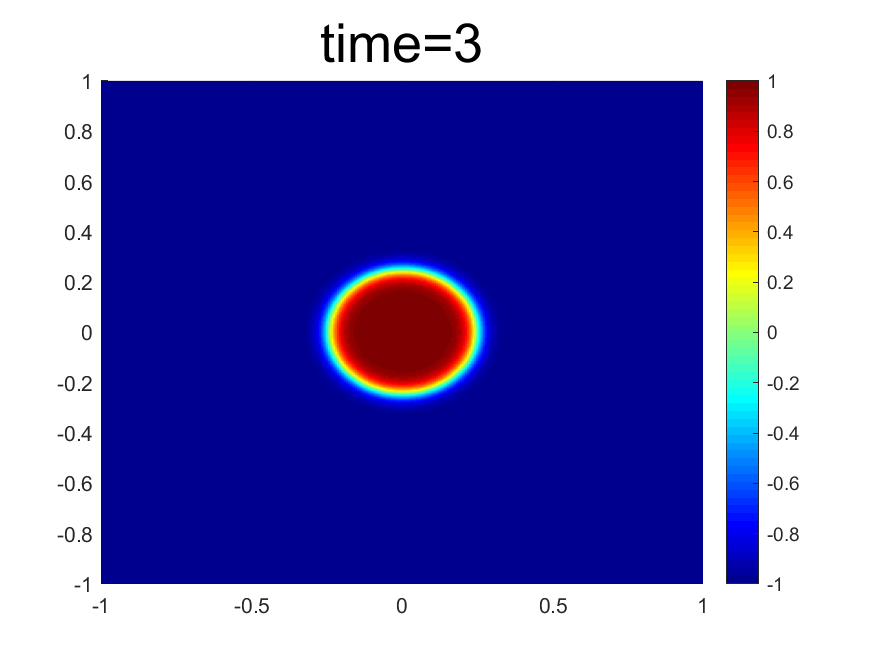}}
\centerline{}
\end{minipage}
\begin{minipage}[t]{0.19\linewidth}
\centerline{\includegraphics[width=4cm,height=3.5cm]{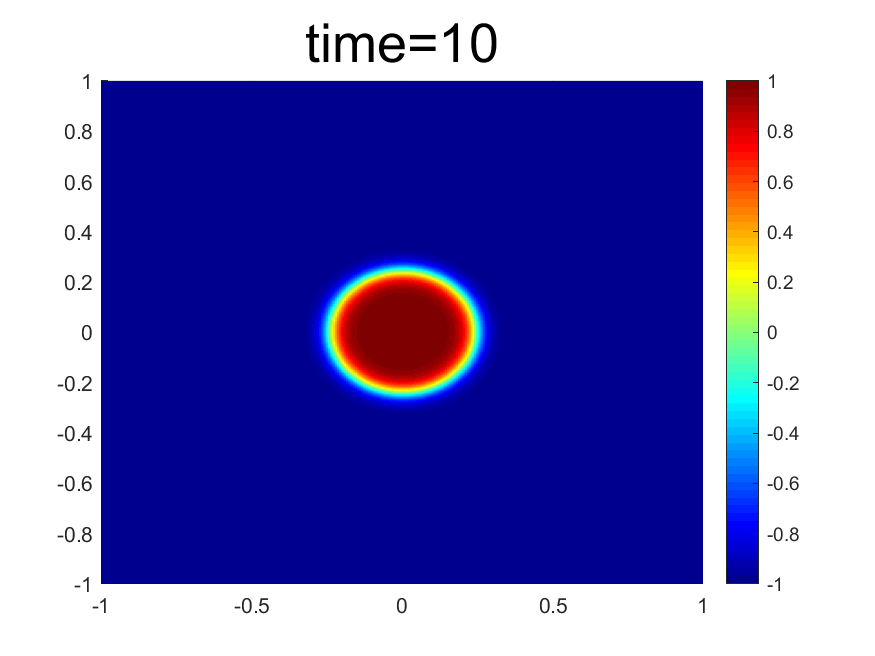}}
\centerline{ (c) $\alpha=0.4, M=1000$, and $\Delta t=0.01$}
\end{minipage}
\begin{minipage}[t]{0.19\linewidth}
\centerline{\includegraphics[width=4cm,height=3.5cm]{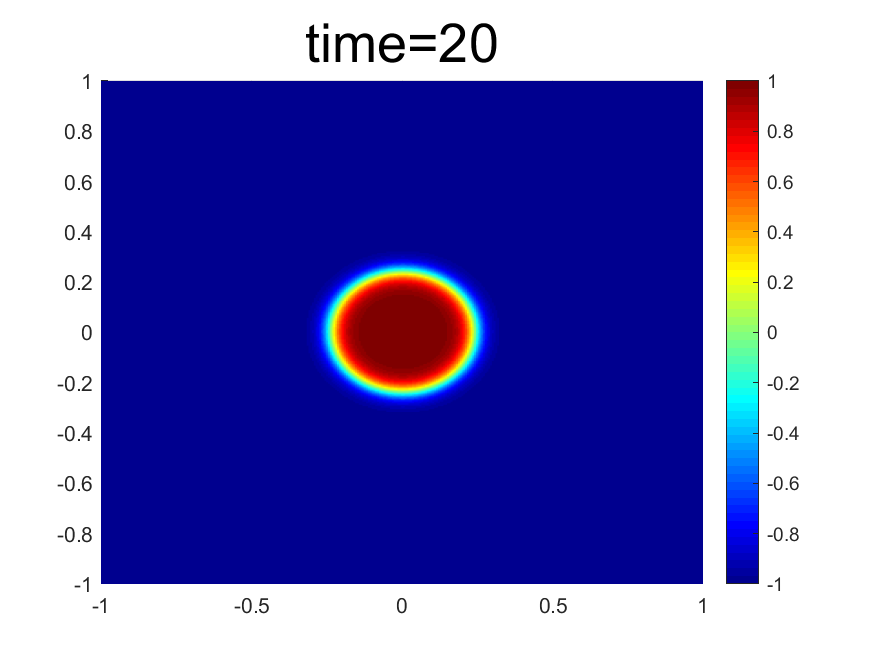}}
\centerline{}
\end{minipage}
\begin{minipage}[t]{0.19\linewidth}
\centerline{\includegraphics[width=4cm,height=3.5cm]{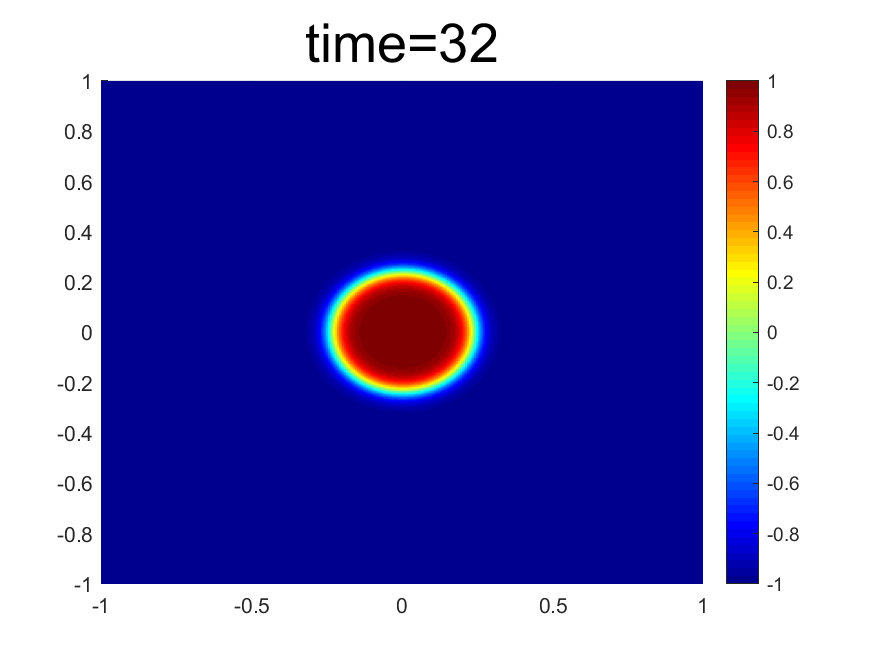}}
\centerline{}
\end{minipage}
\caption{Snapshots of the interface evolution simulated by using the L1-CN scheme for $\alpha=1, 0.9$, and $0.4$.
}\label{fig5}
\end{figure*}

\begin{figure*}[htbp]
\begin{minipage}[t]{0.49\linewidth}
\centerline{\includegraphics[scale=0.55]{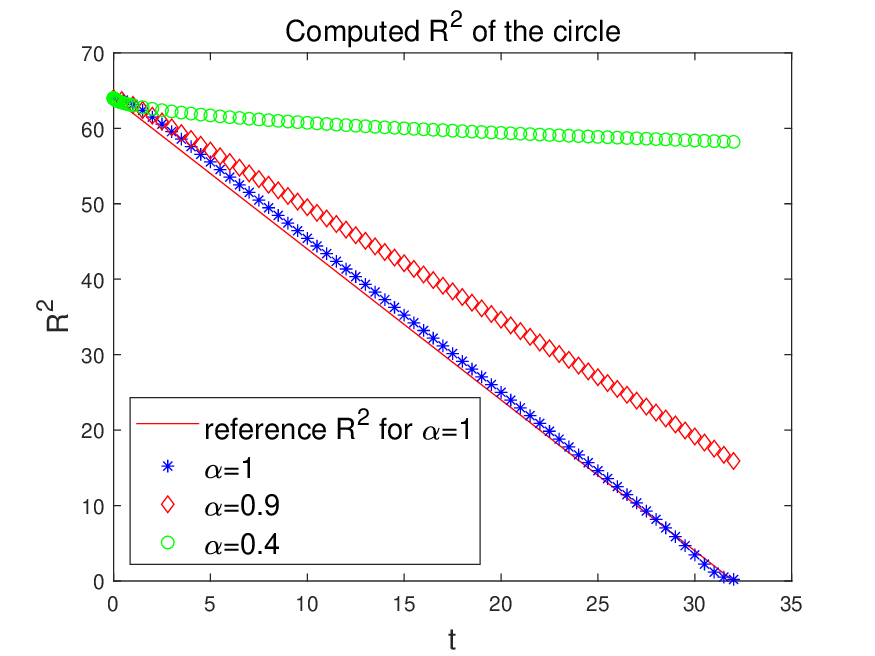}}
\centerline{(a) Decay of the interface radius}
\end{minipage}
\begin{minipage}[t]{0.49\linewidth}
\centerline{\includegraphics[scale=0.55]{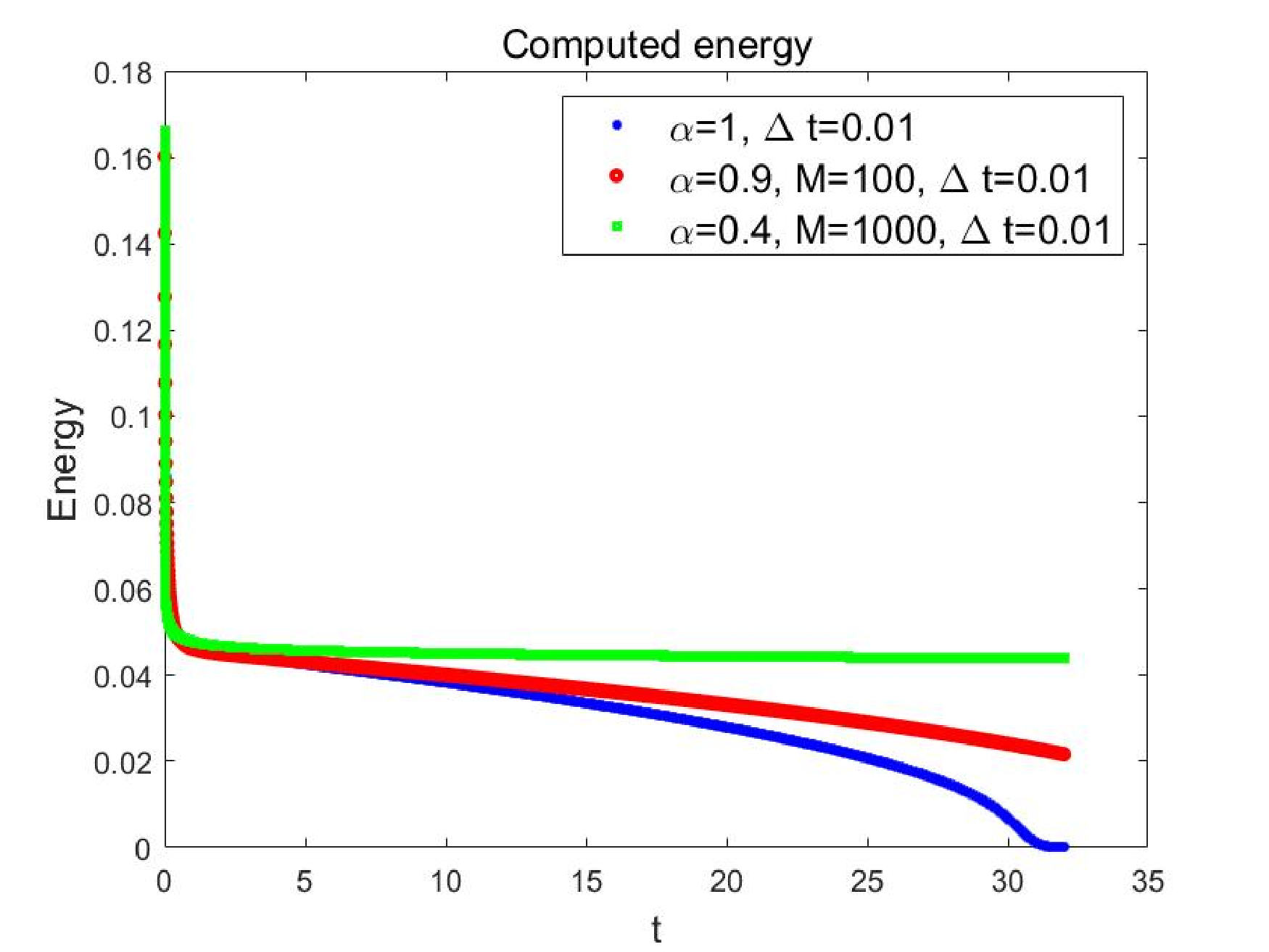}}
\centerline{(b) Total free energy}
\end{minipage}
\caption{The evolution of computed $R^{2}$ and total free energy for time fractional Allen-Cahn equation \eqref{ex4} with $\alpha=1, 0.9$ and $0.4$.
}\label{fig6}
\end{figure*}

\subsection{Coarsening dynamics}

Finally we consider the application of the proposed schemes in investigating coarsening dynamics
of two-phase problems,
governed by the time fractional Allen-Cahn equation \eqref{prob}, and
starting with a random initial data ranging
from $-0.05$ to $0.05$.

We set the computational domain $\Omega$ to be $(-1,1)^{2}$, subject to the Neumann boundary condition, $\varepsilon^{2}=0.001$.
The simulation is performed by using L1-CN scheme \eqref{L1_CN} for the time discretization in the graded mesh with
$r=\frac{2-\alpha}{\alpha}$ in the first subinterval $[0,1]$ and the uniform mesh with $\Delta t=0.01$ in the second subinterval $(1,T]$.
The spatial discretization uses the Legendre spectral method with $128\times128$ basis functions.
In Figure \ref{fig7} we compare the simulation results for different time fractional order $\alpha$. As we can observe from this figure,
there is no distinguishable difference at early time during the phase separation period (say, before $t=5$) for all tested $\alpha$.
After that the solutions apparently start to deviate, the dynamics described by the fractional Allen-Cahn equation
enters into long time scale phase coarsening.
This is further verified by the energy
evolution in Figure \ref{fig8}. Numerically, we can observe that the equations with different fractional orders have similar
dynamics for the phase separation but the coarsening dynamics becomes slower for smaller fractional orders.
\begin{figure*}[htbp]
  \begin{center}
    \begin{tabular}{C{0.9cm}cccc}
    &$\alpha=1$&$\alpha=0.9$&$\alpha=0.7$&$\alpha=0.5$\\
\tabincell{c}{$t=0$\\ \\ \\ \\ \\ \\ \\}
&\begin{minipage}[t]{0.2\linewidth}
\includegraphics[scale=0.22]{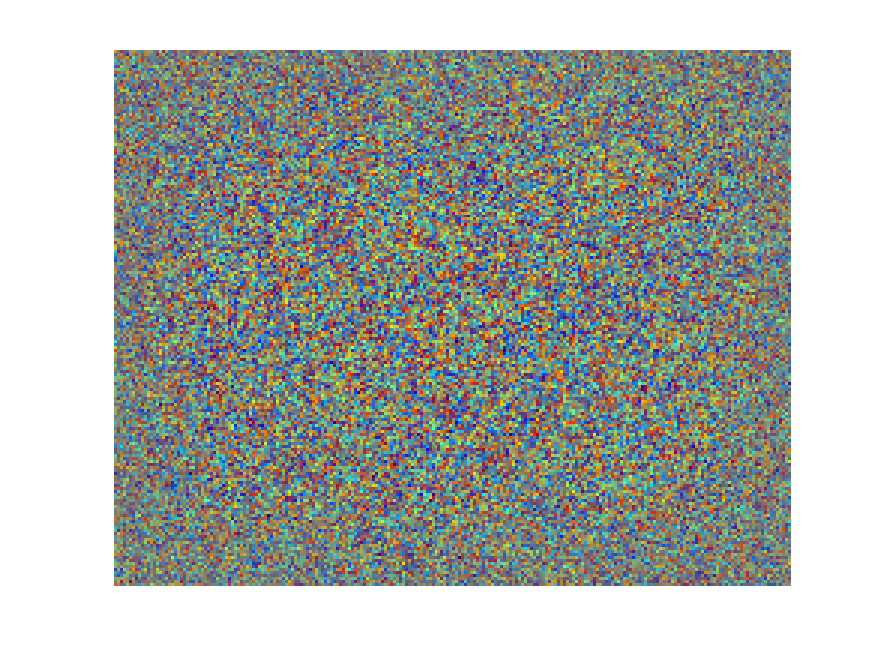}
\end{minipage}&
\begin{minipage}[t]{0.2\linewidth}
\includegraphics[scale=0.22]{FAC5_1.eps}
\end{minipage}&
\begin{minipage}[t]{0.2\linewidth}
\includegraphics[scale=0.22]{FAC5_1.eps}
\end{minipage}
&\begin{minipage}[t]{0.2\linewidth}
\includegraphics[scale=0.22]{FAC5_1.eps}
\end{minipage}\\
\specialrule{0em}{-15mm}{.001pt}
\tabincell{c}{$t=2$\\ \\ \\ \\ \\ \\ \\}
&\begin{minipage}[t]{0.2\linewidth}
\includegraphics[scale=0.22]{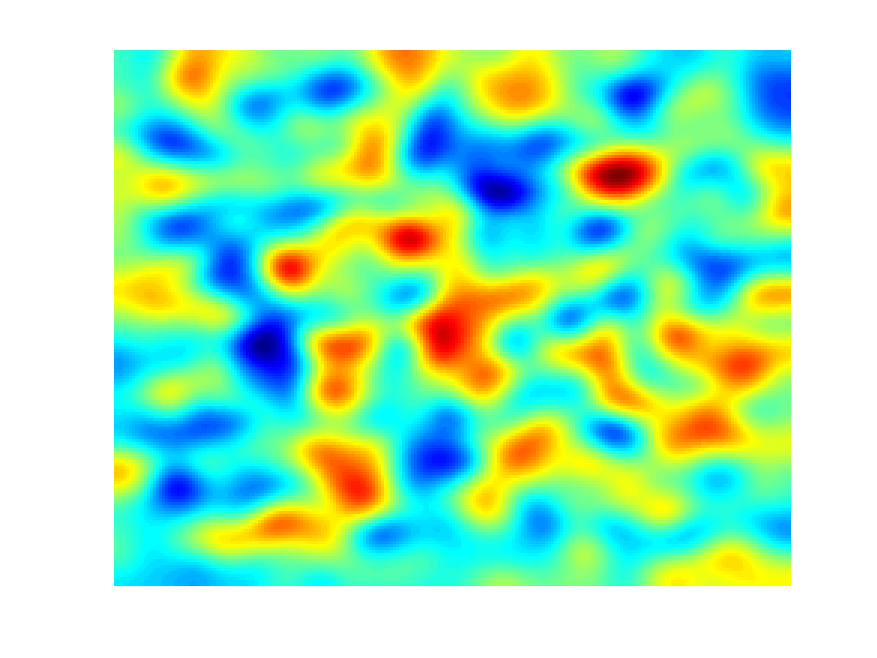}
\end{minipage}&
\begin{minipage}[t]{0.2\linewidth}
\includegraphics[scale=0.22]{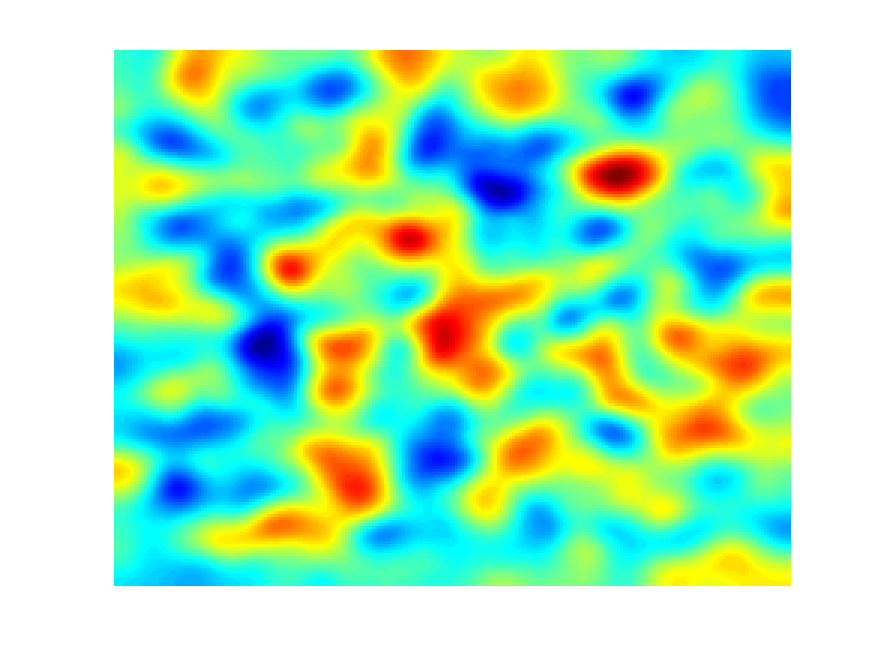}
\end{minipage}&
\begin{minipage}[t]{0.2\linewidth}
\includegraphics[scale=0.22]{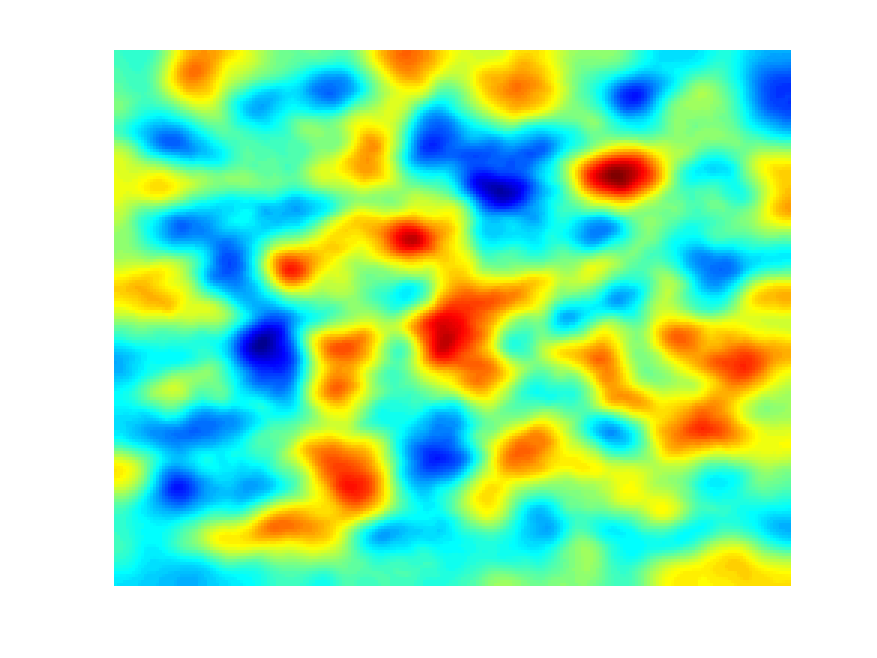}
\end{minipage}
&\begin{minipage}[t]{0.2\linewidth}
\includegraphics[scale=0.22]{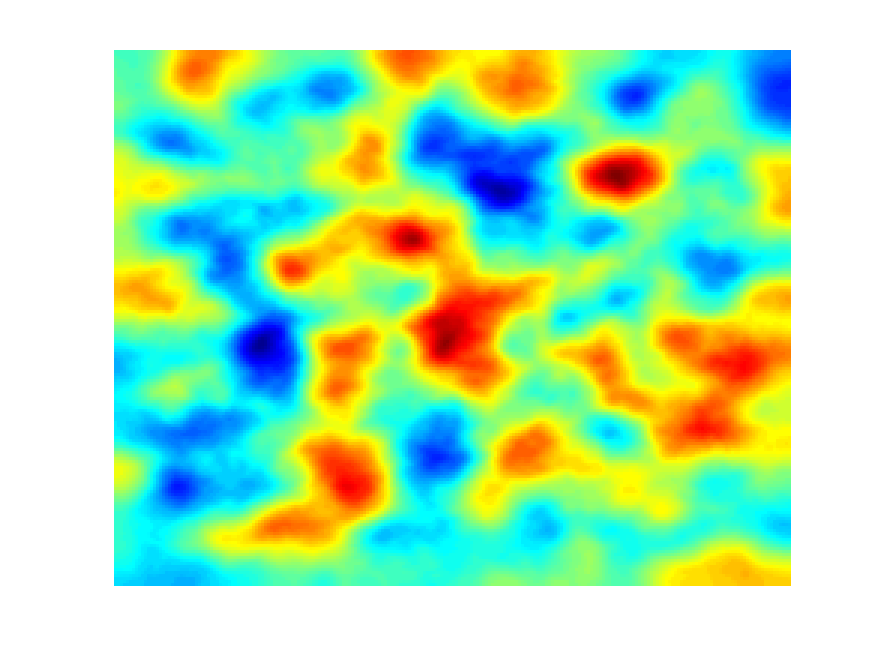}
\end{minipage}\\
\specialrule{0em}{-13mm}{.001pt}
\tabincell{c}{$t=5$\\ \\ \\ \\ \\ \\ \\}
&\begin{minipage}[t]{0.2\linewidth}
\includegraphics[scale=0.22]{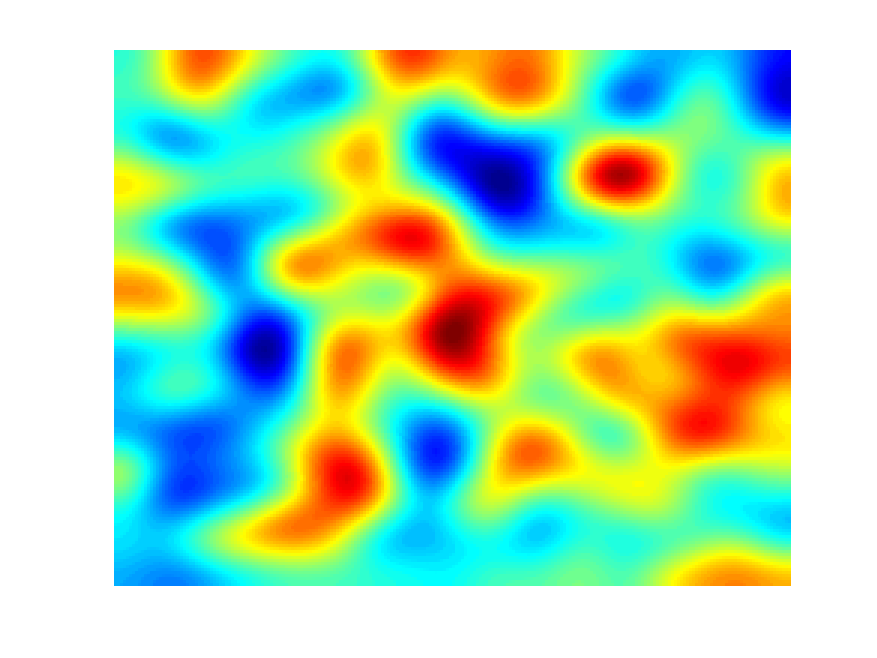}
\end{minipage}&
\begin{minipage}[t]{0.2\linewidth}
\includegraphics[scale=0.22]{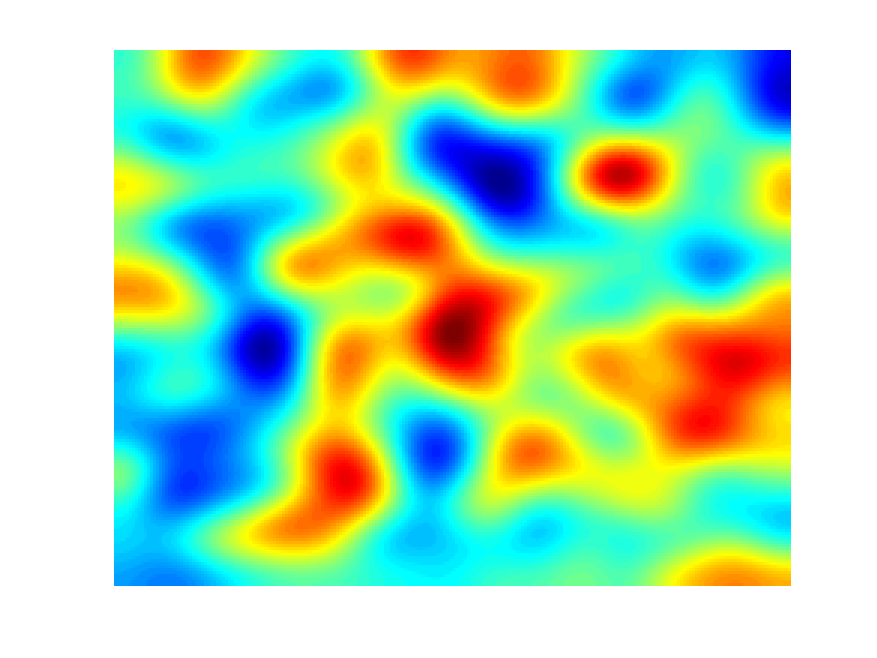}
\end{minipage}&
\begin{minipage}[t]{0.2\linewidth}
\includegraphics[scale=0.22]{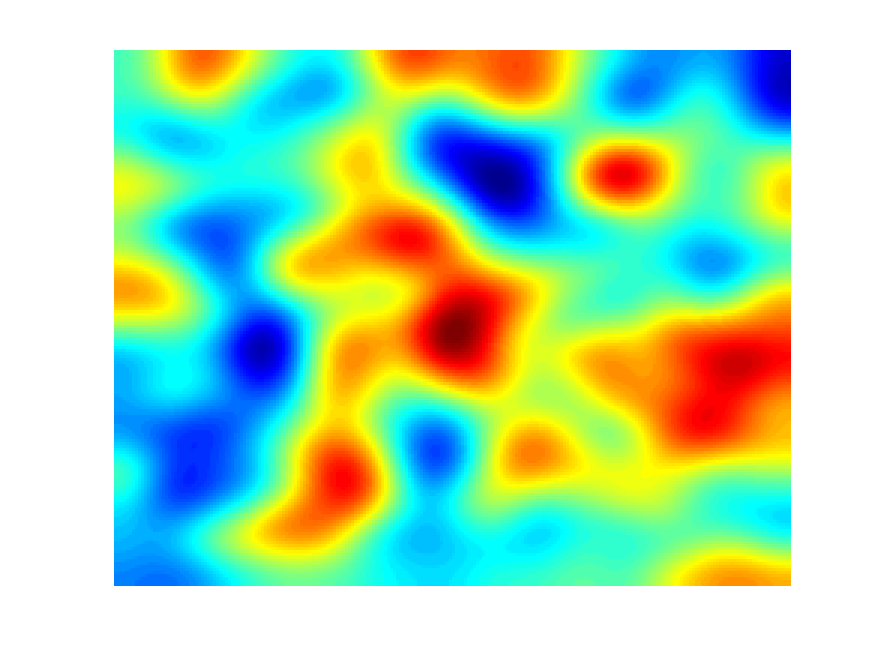}
\end{minipage}
&\begin{minipage}[t]{0.2\linewidth}
\includegraphics[scale=0.22]{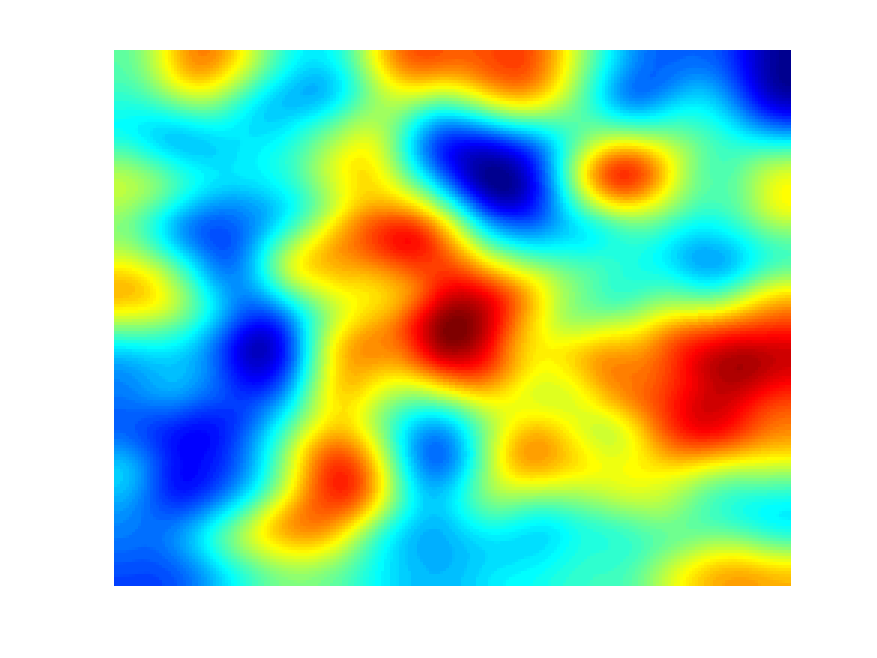}
\end{minipage}\\
\specialrule{0em}{-15mm}{.001pt}
\tabincell{c}{$t=20$\\ \\ \\ \\ \\ \\ \\}
&\begin{minipage}[t]{0.2\linewidth}
\includegraphics[scale=0.22]{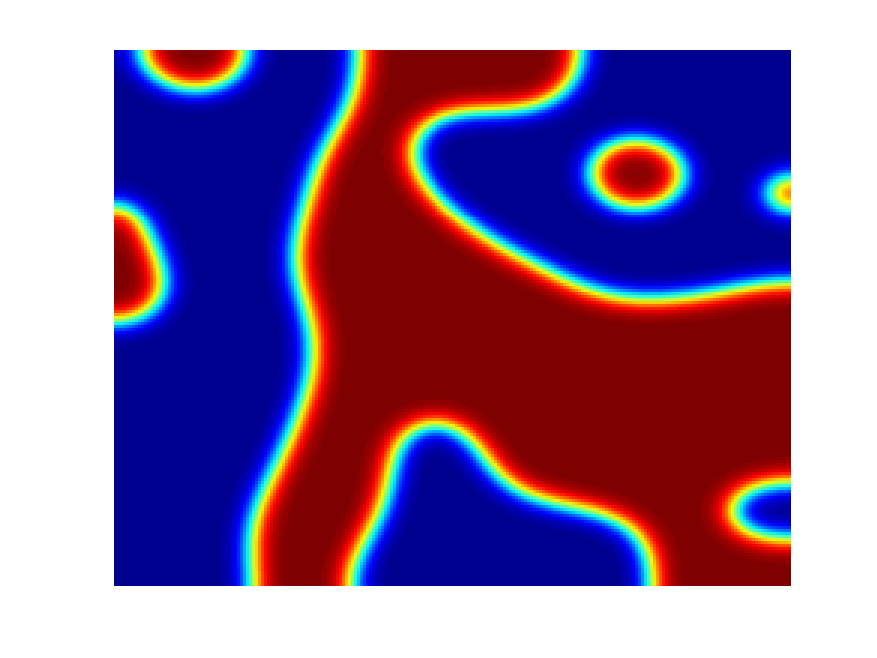}
\end{minipage}&
\begin{minipage}[t]{0.2\linewidth}
\includegraphics[scale=0.22]{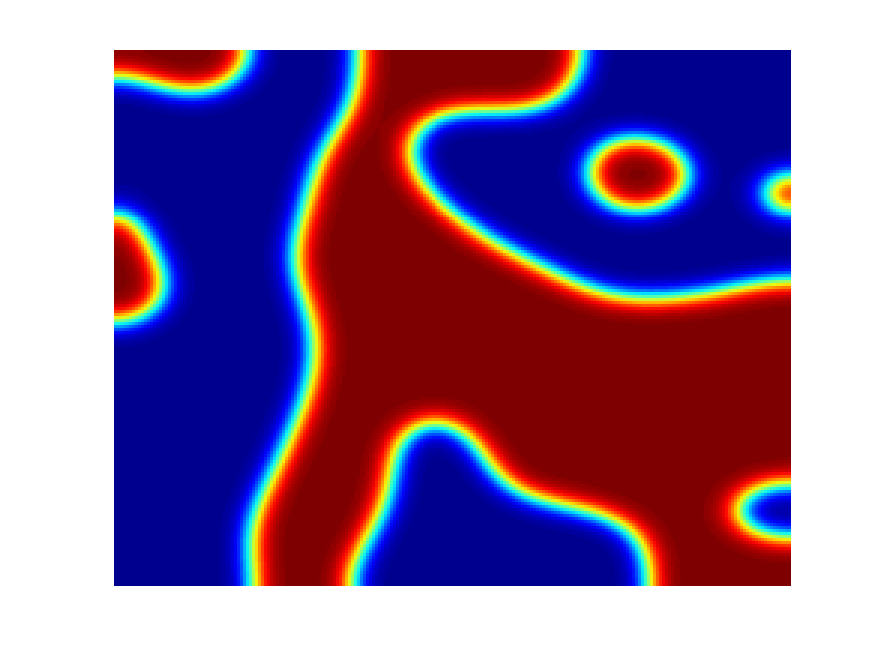}
\end{minipage}&
\begin{minipage}[t]{0.2\linewidth}
\includegraphics[scale=0.22]{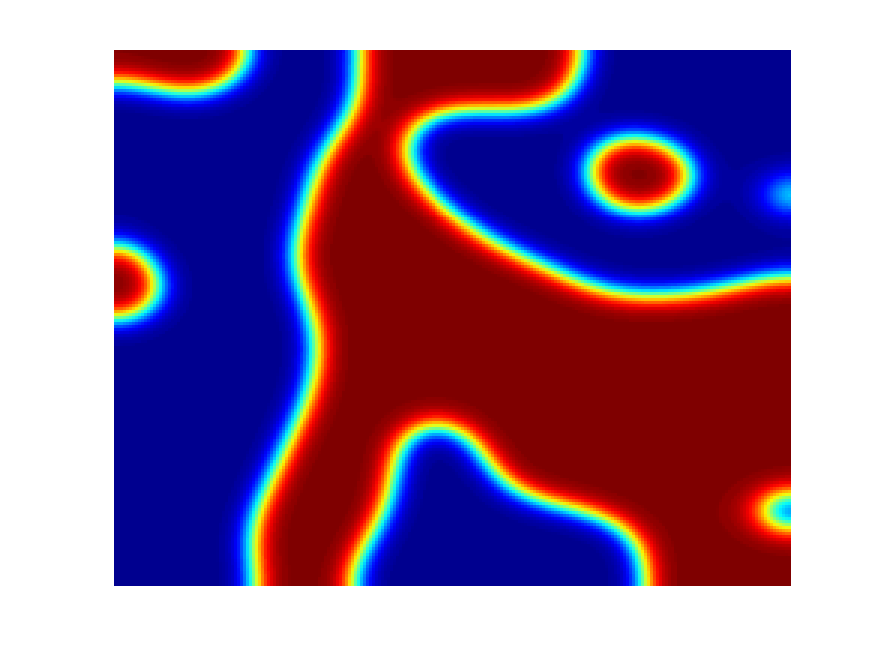}
\end{minipage}
&\begin{minipage}[t]{0.2\linewidth}
\includegraphics[scale=0.22]{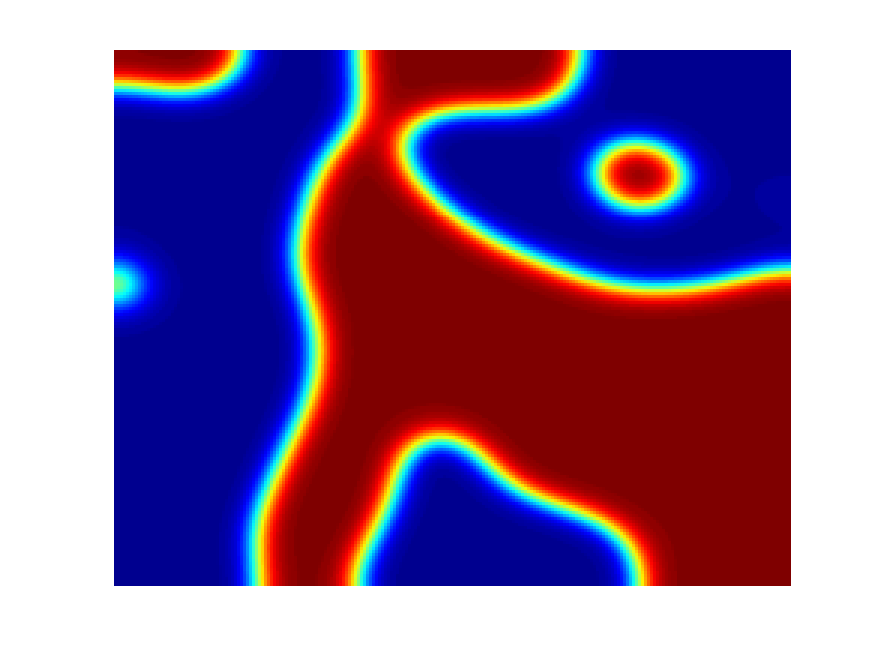}
\end{minipage}\\
\specialrule{0em}{-15mm}{.001pt}
\tabincell{c}{$t=50$\\ \\ \\ \\ \\ \\ \\}
&\begin{minipage}[t]{0.2\linewidth}
\includegraphics[scale=0.22]{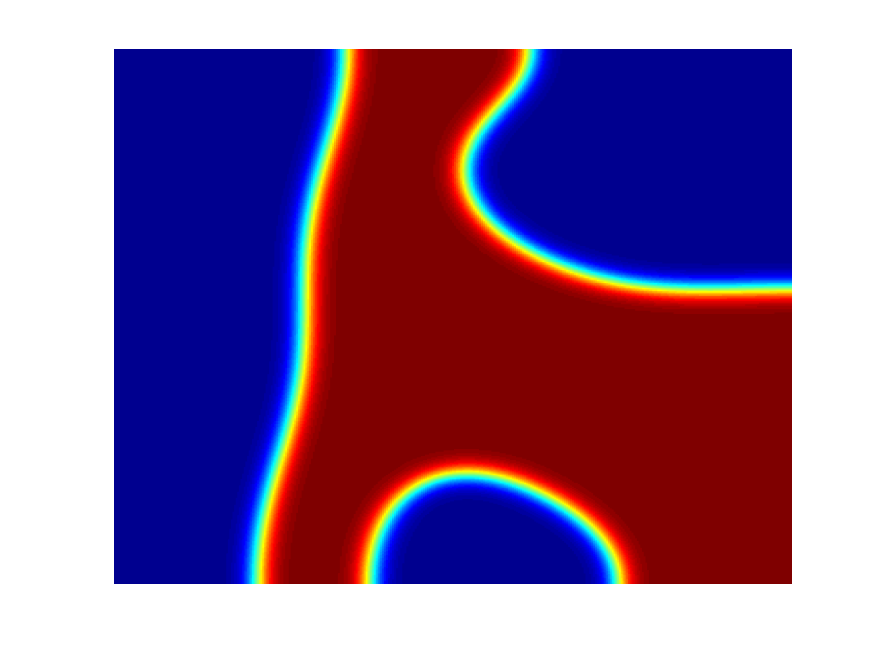}
\end{minipage}&
\begin{minipage}[t]{0.2\linewidth}
\includegraphics[scale=0.22]{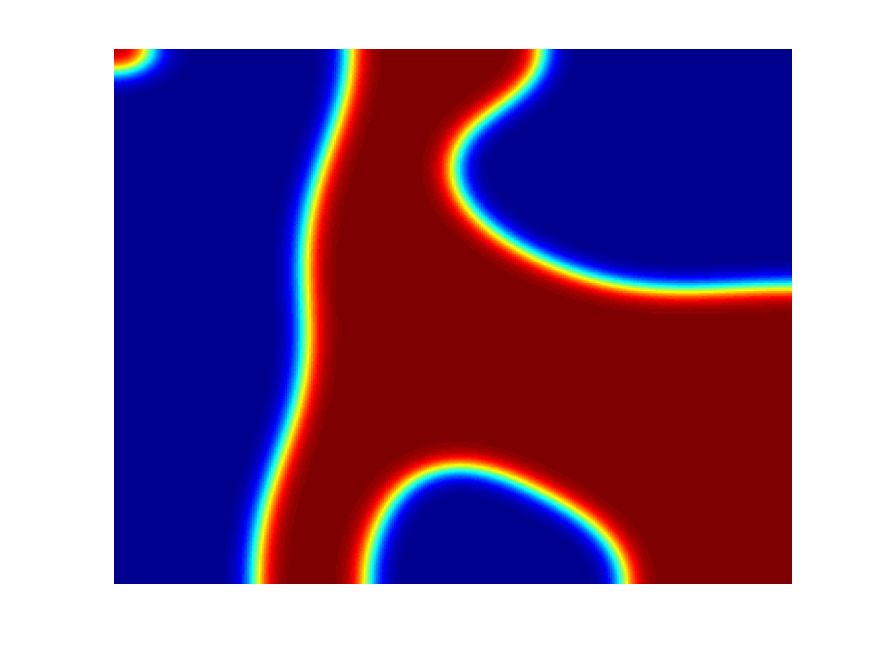}
\end{minipage}&
\begin{minipage}[t]{0.2\linewidth}
\includegraphics[scale=0.22]{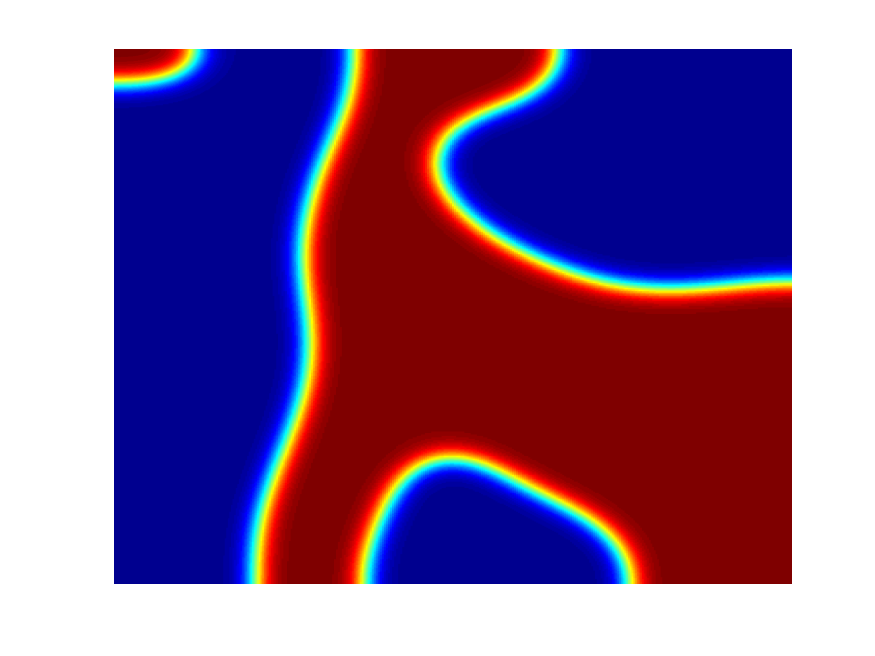}
\end{minipage}
&\begin{minipage}[t]{0.2\linewidth}
\includegraphics[scale=0.22]{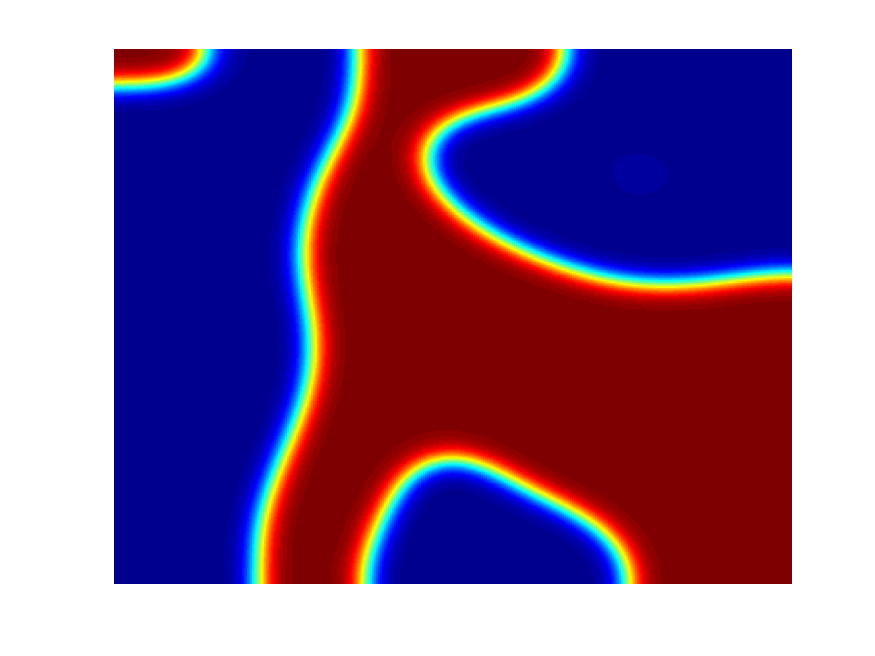}
\end{minipage}\\
\specialrule{0em}{-13mm}{.001pt}
\tabincell{c}{$t=80$\\ \\ \\ \\ \\ \\ \\}
&\begin{minipage}[t]{0.2\linewidth}
\includegraphics[scale=0.22]{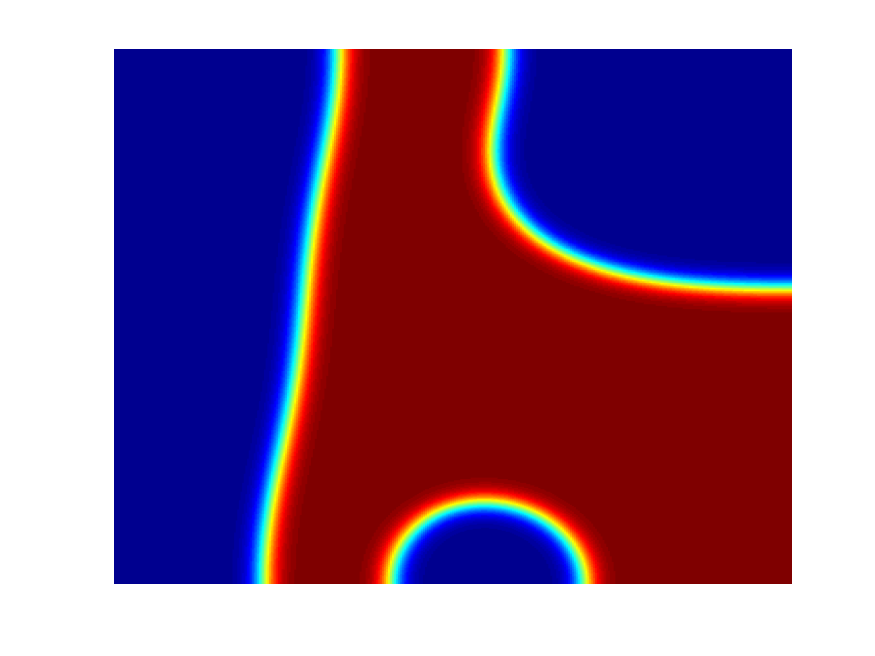}
\end{minipage}&
\begin{minipage}[t]{0.2\linewidth}
\includegraphics[scale=0.22]{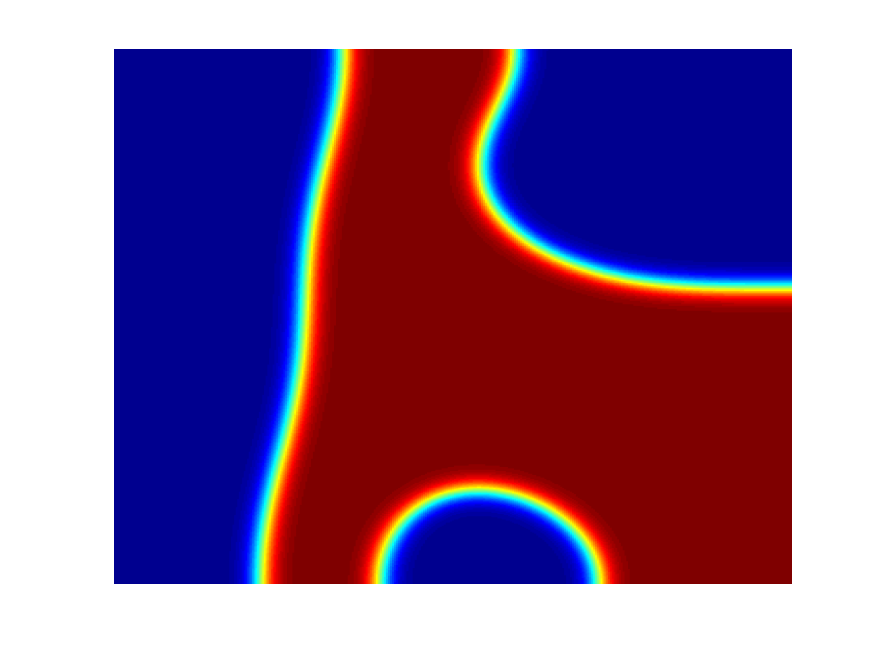}
\end{minipage}&
\begin{minipage}[t]{0.2\linewidth}
\includegraphics[scale=0.22]{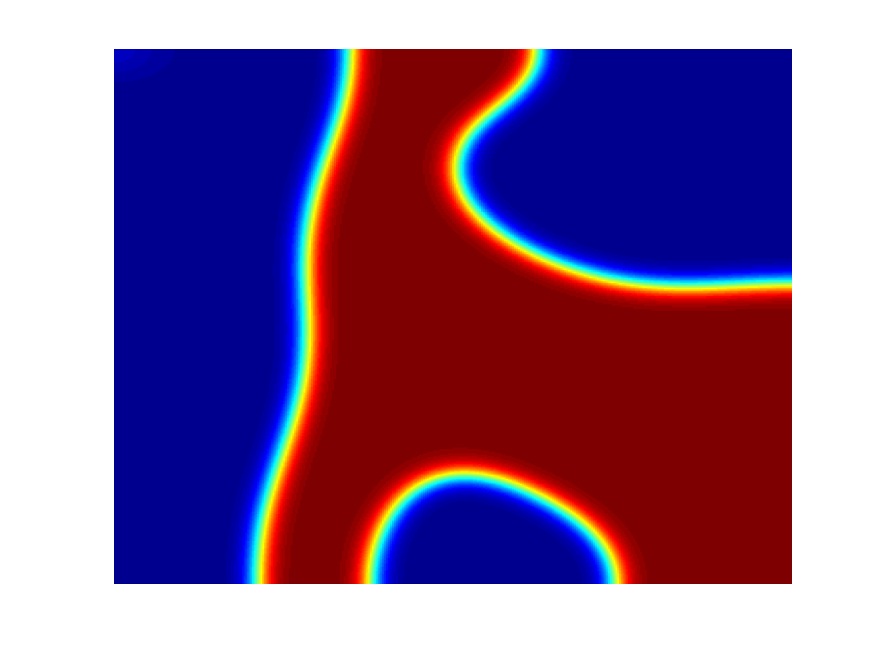}
\end{minipage}
&\begin{minipage}[t]{0.2\linewidth}
\includegraphics[scale=0.22]{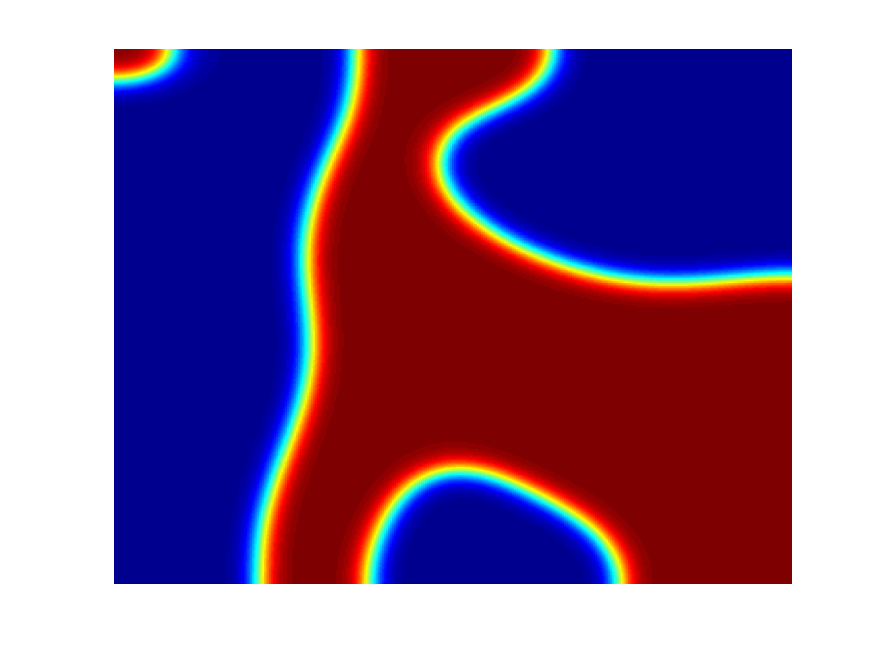}
\end{minipage}\\
\specialrule{0em}{-15mm}{.001pt}
\tabincell{c}{$t=100$\\ \\ \\ \\ \\ \\ \\}
&\begin{minipage}[t]{0.2\linewidth}
\includegraphics[scale=0.22]{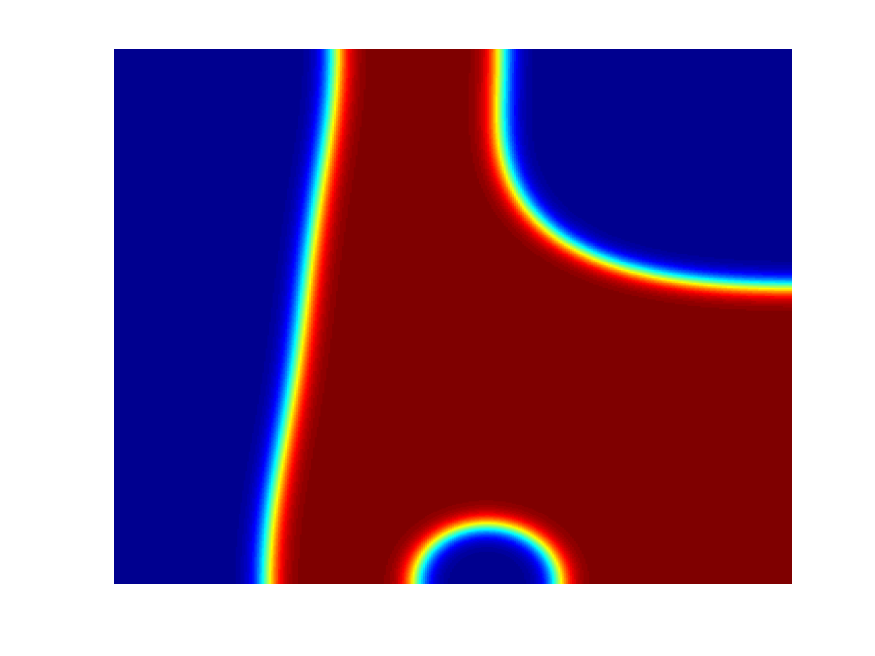}
\end{minipage}&
\begin{minipage}[t]{0.2\linewidth}
\includegraphics[scale=0.22]{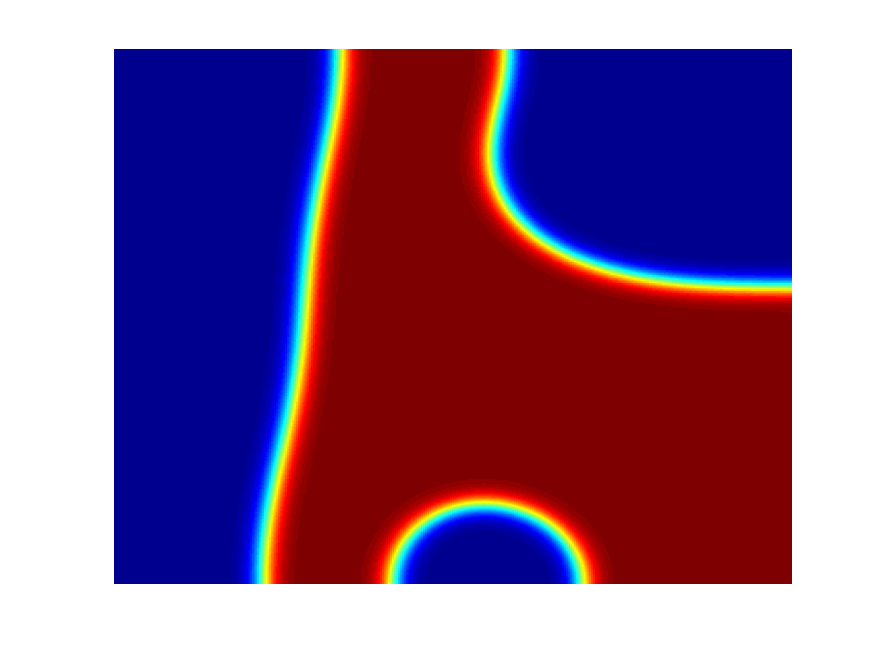}
\end{minipage}&
\begin{minipage}[t]{0.2\linewidth}
\includegraphics[scale=0.22]{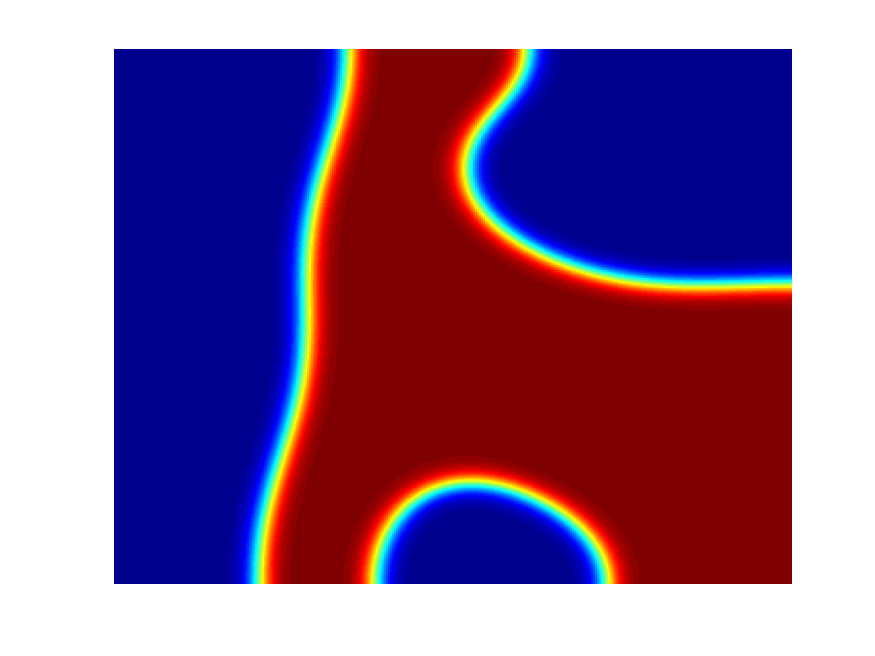}
\end{minipage}
&\begin{minipage}[t]{0.2\linewidth}
\includegraphics[scale=0.22]{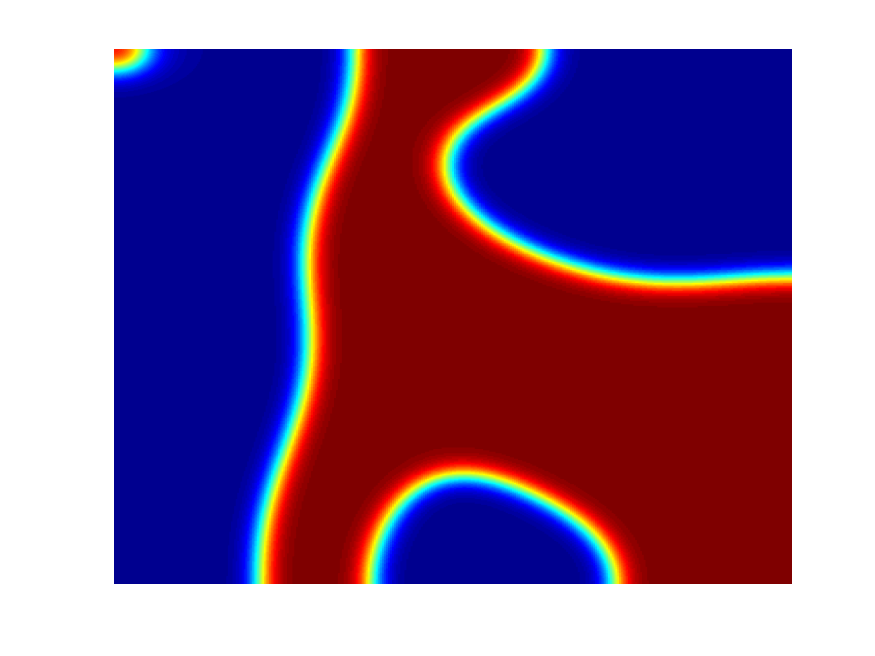}
\end{minipage}\\
\specialrule{0em}{-15mm}{.001pt}
\end{tabular}
  \end{center}
\caption{Snapshots of the simulated phase field evolution by the fractional Allen-Cahn equation with a random initial data
for $\alpha=1, 0.9, 0.7, 0.5$ with $M=100, \Delta t=0.01$.
}\label{fig7}
\end{figure*}

\begin{figure*}[htbp]
\begin{minipage}[t]{0.99\linewidth}
\centerline{\includegraphics[scale=0.8]{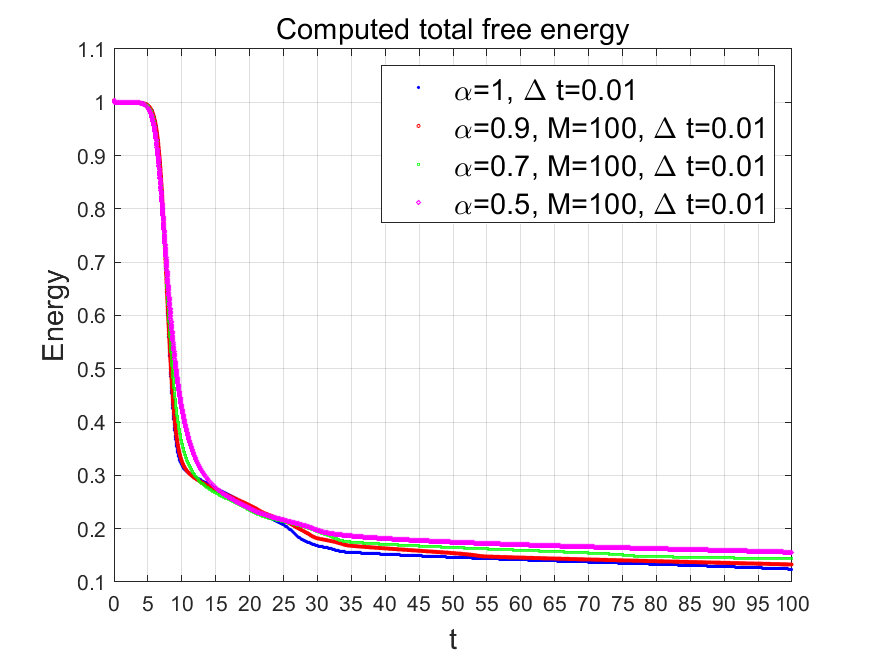}}
\centerline{}
\end{minipage}

\caption{Decay history of the computed free energy by the fractional Allen-Cahn equation for $\alpha=1, 0.9, 0.7, 0.5$.
}\label{fig8}
\end{figure*}

\section {Concluding remarks}
We have proposed two schemes for the $\alpha$ order time fractional Allen-Cahn equation:
a first order scheme and a $2-\alpha$ order scheme, both constructed on the uniform mesh and graded mesh.
Mathematically, such fractional gradient flows can be regarded as a variation of the phase field models,
used to describe the memory effect of some materials, which have attracted attentions in recent few years.
Our proposed schemes are based on the L1 discretization for the time fractional derivative and
extended SAV approach to deal with the nonlinear energy potential.
The stability property of the schemes was rigorously established, while the accuracy was carefully examined through
a series of numerical tests.
Precisely, our stability analysis showed that the schemes constructed on the uniform mesh are
unconditionally stable in the sense that the associated discrete energy is always decreasing.
As far as we know this is the first proof for the unconditional stability of a $2-\alpha$ order scheme.
The numerical experiments carried out in the paper have demonstrated
that the proposed schemes have desirable accuracy. In particular,
the initial singularity of the solution was effectively resolved by using the graded mesh.
From the coarsening simulation, we can conclude that the fractional Allen-Cahn equation
have similar dynamics for the phase separation but the coarsening dynamics slows down with decreasing fractional
order.  Numerically, it was observed that the time to the steady state becomes longer for smaller
fractional order.
{\color{black} Notice that this phenomena has already been observed in some other existing works,
but theoretical justification remains open.
It is also worth to mention that the proposed schemes can be directly applied to some other gradient flows,
such as the Cahn-Hilliard equation and molecular beam epitaxial growth models,
and the stability results can be derived similarly.}

{\color{black} To summarize, the proposed SAV-based schemes enjoy
a few good properties: unconditional stability;
only require solving linear decoupled systems with constant coefficients at each time step; only require that the free energy is bounded from below.
However there exist several theoretical issues such as
rigorous proof of the maximum principle of the numerical solution, error estimation, etc.
The difficulty is not only due to the presence of the fractional derivative,
but also comes from the SAV approach. In SAV-type schemes, the nonlinear term is treated explicitly, and an
auxiliary variable is introduced in front of the nonlinear term. This makes the numerical analysis harder than
for other approaches.
Possible answer to these issues certainly needs more effort.}

\bibliographystyle{plain}
\bibliography{ref}
\end{document}